%% file: Mod2.tex
\documentclass[11pt]{amsart}
\usepackage{amsmath,amssymb,amscd}

\usepackage{latexsym}
\usepackage[dvips]{graphics}
\usepackage[latin1]{inputenc}
\usepackage{amsthm}
\usepackage{amsfonts}
\usepackage{comment}
\usepackage{asymptote}
\usepackage{graphicx,xy,pstricks}
\input xy
\xyoption{all}

\newcommand{\executeiffilenewer}[3]{%
 \ifnum\pdfstrcmp{\pdffilemoddate{#1}}%
 {\pdffilemoddate{#2}}>0%
 {\immediate\write18{#3}}\fi%
}

\newtheorem{theorem}{Theorem}[section]
\newtheorem{proposition}[theorem]{Proposition}
\newtheorem{lemma}[theorem]{Lemma}
\newtheorem{corollary}[theorem]{Corollary}

\theoremstyle{remark}
\newtheorem{remark}{Remark}[section]

\newtheorem*{theorem*}{Theorem}

\theoremstyle{plain}      

\newcommand{\Z}{\mathbb{Z}}
\newcommand{\Q}{\mathbb{Q}}
\newcommand{\R}{\mathbb{R}}
\newcommand{\N}{\mathbb{N}}
\newcommand{\C}{\mathbb{C}}
\newcommand{\HH}{\mathbb{H}^2}

\newcommand{\mcg}{\operatorname{Mod}}

\renewcommand{\phi}{\varphi}
\renewcommand{\sinh}{\operatorname{sh}}
\renewcommand{\cosh}{\operatorname{ch}}
\renewcommand{\tanh}{\operatorname{th}}
\renewcommand{\epsilon}{\varepsilon}
\newcommand{\mo}{\mathcal{M}}
\newcommand{\tr}{\operatorname{Tr}}
\newcommand{\eu}{\operatorname{eu}}

\newcommand{\Isom}{\operatorname{Isom}}
\newcommand{\rep}{\mathcal{R}}
\newcommand{\psl}{\mathrm{PSL}_2(\mathbb{R})}
\newcommand{\psltild}{\widetilde{\mathrm{PSL}}_2(\mathbb{R})}
\newcommand{\sldeuxR}{\mathrm{SL}_2(\mathbb{R})}

\newcommand{\Aut}{\operatorname{Aut}}
\newcommand{\Out}{\operatorname{Out}}
\newcommand{\Hom}{\operatorname{Hom}}

\title{The modular action on PSL$_2(\R)$-characters in genus 2}

\author{Julien March\'e}
\address{}
\email{}
\author{Maxime Wolff}
\address{}
\email{}
\begin{document}
\maketitle
\begin{center}
Version of \today.
\end{center}
\begin{abstract}
We explore the dynamics of the action of the mapping class group
in genus $2$ on the $\psl$-character variety. We prove that this
action is ergodic on the connected components of Euler class $\pm 1$,
as it was conjectured by Goldman.
In the connected component of Euler class $0$ there are
two invariant open subsets, on one of them the action is ergodic.
In this process we give a partial answer to a question of Bowditch.

\vspace{0.2cm}

\end{abstract}

\section{Introduction and statements}

\input{Intro}

\section{Reminders and preliminary considerations}\label{SectionPrelimin}

\input{Prelim}

\section{Non-ergodicity on the component of Euler class $0$}\label{SectionNonErgod}

\input{NonErgod}

\section{Coordinates on representation varieties}\label{SectionParam}

\input{Param}

\section{Search for a non-hyperbolic curve}\label{SectionChasseElliptique}

\input{ChasseElliptique1}
\input{ChasseElliptique2}

\section{Ergodicity}\label{SectionErgod}

\input{Ergo}


\bibliographystyle{plain}

\bibliography{Biblio}

\end{document}

%% file: Intro.tex
In all this text, for every $g\geq 2$, $\Sigma_g$ will be a
genus $g$ compact connected oriented surface without boundary,
and $\Gamma_g$ will be a fundamental group of $\Sigma_g$. Let
\[\Gamma_g=\langle a_1,\ldots,b_g\, |\, [a_1,b_1]\cdots[a_g,b_g]\rangle \]
be a standard presentation of $\Gamma_g$.
The space $\mathcal{R}(\Gamma_g)=\Hom(\Gamma_g,\psl)$ is
naturally a real algebraic variety in $\psl^{2g}$.
The group $\psl$ acts on $\mathcal{R}(\Gamma_g)$ by conjugation, and preserves
the set $\mathcal{R}_{ne}(\Gamma_g)$ of non-elementary representations
(ie the set of representations of Zariski dense image in $\psl$).
For simplicity, we will denote by $\mathcal{M}(\Gamma_g)$ the quotient
$\mathcal{R}_{ne}(\Gamma_g)/\psl$. A theorem of Goldman \cite{Goldman84}
asserts that $\mathcal{M}(\Gamma_g)$ is a smooth symplectic manifold of dimension
$6g-6$ (see also \cite{Weil64}).

There is a natural map $\eu\colon\mathcal{R}(\Gamma_g)\rightarrow\Z$,
the {\em Euler class}, which factors through the quotient
$\eu\colon\mathcal{M}(\Gamma_g)\rightarrow\Z$. It is the map which associates
to any representation $\rho$, the Euler class of the $\R P^1$-bundle on
$\Sigma_g$ associated to $\rho$. We will come back to this
definition, and more reminders, in Section \ref{SectionPrelimin}.
The Euler class satisfies the so-called Milnor-Wood inequality,
\[ |\eu(\rho)|\leq 2g-2, \]
and Goldman proved in \cite{Goldman80} that the equality characterizes
the representations which are faithful and discrete.
Also, he proved in \cite{Goldman88} that the Euler class, subject to the
Milnor-Wood inequality, parametrizes the different connected components
of $\mathcal{R}(\Gamma_g)$. In other words,
for all $k\in\{2-2g,\ldots,2g-2\}$, the set of representations of Euler
class $k$ is nonempty and connected, and the component of Euler class
$2-2g$ (resp. $2g-2$) is identified with the Teichm\"uller space of
$\Sigma_g$ (resp. of $-\Sigma_g$, ie the surface equipped with the opposite
orientation). For $k\in\{2-2g,\ldots,2g-2\}$, we will denote by
$\mathcal{M}^k(\Gamma_g)$ the space of classes of representations of Euler class $k$.

The automorphism group $\Aut(\Gamma_g)$ acts on $\mathcal{R}(\Gamma_g)$ by
precomposition, and its index $2$ subgroup $\Aut^+(\Gamma_g)$ respecting
the orientation (ie fixing the fundamental class in $H_2(\Gamma_g,\Z)$)
preserves the Euler class. The Dehn-Nielsen-Baer theorem identifies
the quotient $\Out^+(\Gamma_g)=\Aut^+(\Gamma_g)/\mathrm{Inn}(\Gamma_g)$ with
the mapping class group
$\mcg(\Sigma_g)=\mathrm{Homeo}^+(\Sigma_g)/\mathrm{Homeo}^+_0(\Sigma_g)$,
and the action of $\Out^+(\Gamma_g)$ on the two components of Fuchsian
(ie faithful and discrete) representations is identified to the action of the mapping
class group on the Teichm\"uller space. This action is well known to
be discrete, and the quotient is the moduli space of the surface.

The discreteness of the mapping class group action on the Teichm\"uller
space comes from the interpretation of every point in this space as a
hyperbolic, or complex structure on the surface. Natural functions on
this space
yield functions on the 
Teichm\"uller space which
are invariant under the mapping class group action. The points in the other
connected components, however, do not seem to bear (by themselves)
any structure on the surface (they can be related to branched
$\C P^1$-structures, or to branched hyperbolic structures, or to
anti-de Sitter structures on $3$-manifolds related to the surface
\cite{GalloKapovichMarden,Tan,GueritaudKasselWolff},
but the corresponding moduli spaces have bigger dimension)
and Goldman conjectured in \cite{Goldman05} that the mapping class group
should act ergodically on every connected component of
$\mathcal{M}(\Gamma_g)$ of non-zero and non-extremal Euler class.

It seems to be of folklore knowledge that this conjecture is related
to a long standing question of Bowditch (\cite{Bowditch}, question C):
{\em Does every non-Fuchsian representation send some {\em simple}
closed loop to a non-hyperbolic element of $\psl$?}
However, to the authors' knowledge, this link between Bowditch's
question and Goldman's conjecture does not enjoy a precise statement
anywhere in the litterature.

This paper is devoted to the description of the action of the mapping class group
of the surface of genus $2$ on the {\em exotic} (ie, non-Fuchsian)
connected components of $\mathcal{M}(\Gamma_2)$.

Our first result is a simple and unexpected observation very particular to the genus $2$ case; it uses the
hyperelliptic involution $\varphi$, which is the only non-trivial element of
the center of $\mcg(\Sigma_2)$.
\begin{proposition}\label{nonergod}
  The space $\mo^0(\Gamma_2)$ of conjugacy classes of
  representations of Euler class $0$ is the union of two disjoint $\mcg(\Sigma_2)$-invariant
  open subsets, $\mo^0_+(\Gamma_2)$ and $\mo^0_-(\Gamma_2)$ with the following property.
  
  Let $\bar\varphi\in\Aut^+(\Gamma_2)$ be a lift of $\varphi$.
  Then for every $[\rho]\in\mo^0_+(\Gamma_2)$ (resp. $\mo^0_-(\Gamma_2)$),
  the representations
  $\rho$ and $\rho\circ\bar\varphi$ are conjugated by an
  orientation-preserving (resp. reversing) isometry of the hyperbolic plane.
\end{proposition}
An explicit, individual description of the elements of $\mo^0_\pm(\Gamma_2)$
yields the following statement.
\begin{proposition}\label{propintrononergod}{\ }
  \begin{itemize}
    \item[-] For every $[\rho]\in\mo^0_+(\Gamma_2)$ and for every
      non-separating simple curve $a$, $\rho(a)$ is either hyperbolic
      or the identity. Moreover,
      for every simple curves $a$, $b$
      such that $i(a,b)=1$, the trace of the commutator
      $[\rho(a),\rho(b)]$ is
      in $(-\infty,2]$.
      It is equal to $2$ if and only if $[\rho(a),\rho(b)]=1$.
    \item[-] For every $[\rho]\in\mo^0_-(\Gamma_2)$ and for every
      simple curves $a$, $b$ such that $i(a,b)=1$, the trace of
      the commutator $[\rho(a),\rho(b)]$ is not lower than $2$,
      and this commutator is either hyperbolic or the identity.
  \end{itemize}
\end{proposition}
In the statement above, $i(a,b)$ denotes the minimal geometric
intersection number between curves freely homotopic to $a$ and $b$.
Also, recall that the commutator of two elements of $\psl$ is a
well-defined element of $\sldeuxR$, so its trace is well-defined,
without absolute value.
\begin{corollary}\label{CorNonErgod}
  The action of $\mcg(\Sigma_2)$
  on $\mo^0(\Gamma_2)$ is not ergodic.
\end{corollary}
Proposition \ref{nonergod} is not only specific to the genus two, but also to
the dimension two. The lack of generalizations of this proposition to other
situations yields the following surprising remarks of interest independent of the rest of the paper:
\begin{itemize}
  \item[-] if $\rho\in\Hom(\Gamma_2,\psl)$,
    the map $\Gamma_2\rightarrow\R_+$, $\gamma\mapsto|\tr(\rho(\gamma))|$
    is not determined, in general, by its restriction to the set of simple closed curves;
  \item[-] if $n\geq 3$, if $\rho\in\Hom(\Gamma_2,\textrm{SL}_n(\R))$,
    the map $\Gamma_2\rightarrow\R$, $\gamma\mapsto\tr(\rho(\gamma))$
    is not determined, in general, by its restriction to the set of non-separating simple closed
    curves, neither on the set of separating simple closed curves.
\end{itemize}

We then turn to Bowditch's question in genus $2$.
\begin{theorem}\label{Bowditch}
  Let $[\rho]\in\mo(\Gamma_2)\smallsetminus\mo^0_+(\Gamma_2)$. Then
  $\rho$
  sends some simple closed curve to a
  non-hyperbolic element of $\psl$.
\end{theorem}

From this result, we derive the following theorem:

\begin{theorem}\label{ergodicite}{\ }
  \begin{itemize}
    \item[-] The action of $\mcg(\Sigma_2)$ on $\mo^k(\Gamma_2)$ is ergodic if $|k|=1$.
    \item[-] $\mo^0_-(\Gamma_2)$ is connected, and the action of $\mcg(\Sigma_2)$ on $\mo^0_-(\Gamma_2)$ is ergodic.
  \end{itemize}
\end{theorem}

Actually, the proof proceeds with a general result relating Bowditch's question to the ergodicity property.
If $g\geq 2$,
let $\mathcal{NH}^k(\Gamma_g)$ denote the subset of $\mo^k(\Gamma_g)$ consisting
of representations which send some simple closed loop to a non-hyperbolic
element. 
\begin{theorem}\label{ergodgeneral}
  Let $g\geq 2$, and
  let $k\in\{3-2g,\ldots,2g-3\}$. Suppose moreover that $(g,k)\neq(2,0)$.
  Then
  the action of $\mcg(\Sigma_g)$ on $\mathcal{NH}^k(\Gamma_g)$ is ergodic.
\end{theorem}

In particular, Goldman's conjecture is equivalent to $\mathcal{NH}^k(\Gamma_g)$ having full measure in $\mo^k(\Gamma_g)$. 

Part of the techniques used in this paper do not apply to
the study of $\mo^0_+(\Sigma_2)$, and the question of the ergodicity
of $\mcg(\Sigma_2)$ on this open set remains open. We hope to
address it in a future work.

Now we present a very brief outline of the proof of Theorem \ref{Bowditch}.
A recent result proved
independently by Deroin-Tholozan \cite{DeroinTholozan} and
Gu\'eritaud-Kassel-Wolff \cite{GueritaudKasselWolff} asserts that
for every non-Fuchsian representation $\rho$, there exists a Fuchsian
representation $j$ such that for all $\gamma\in\Gamma_g$,
$|\tr(\rho(\gamma))|\leq|\tr(j(\gamma))|$ (actually a better
inequality is true, but we will not need it here).
There is a constant $B_2>0$ such that, for every Fuchsian
representation $j$, there exists a family of three curves cutting
$\Sigma_2$ into two pairs of pants, such that the maximum of the $j$-lengths
of these three curves is not greater than $B_2$. This is called the Bers
constant (in genus $2$) and it was explicitly computed recently by
Gendulphe \cite{Gendulphe}: $\cosh(B_2/2)\simeq 4.67$. In
the search for simple curves mapped by $\rho$ to non-hyperbolic elements,
this allows us to start from a pair of pants decomposition in which all three
curves in the cut system is sent by $\rho$ to an element of trace no bigger
than $2\cosh(B_2/2)$. The last main ingredient in the proof of Theorem
\ref{Bowditch} is a theorem of Goldman asserting that every representation
of the fundamental group of the one-holed torus in $\sldeuxR$ which
maps the boundary curve to an element of trace in $(2,18]$ maps a simple
closed curve to a non-hyperbolic element \cite{Goldman03}.
An adequate parametrization of
$\mathcal{M}(\Gamma_2)$ and an involved algorithm of trace reduction, then
enable to find a simple closed curve mapped to a non-hyperbolic element.

Let us describe the organization of the paper. Section \ref{SectionPrelimin} gathers reminders as well as preliminary
considerations. It begins with a review
of the Euler class and of classical results of Goldman about the spaces of
representations. It continues with a detailed geometric description of
the commutators in $\psl$, which will be
useful for the proof of Theorem \ref{ergodicite},
and ends with a reminder of some formulas of hyperbolic trigonometry,
which will be used in the proof of Theorem \ref{Bowditch}.

We prove Proposition \ref{nonergod}, Proposition \ref{propintrononergod}, and
the subsequent remarks, in Section
\ref{SectionNonErgod}.
We then give in Section \ref{SectionParam}
an explicit parametrization of the representation varieties
in genus $2$, which will support the explicit computations of
Section \ref{SectionChasseElliptique}, where we prove Theorem \ref{Bowditch}.
Finally, the results involving ergodicity will be proved in
Section \ref{SectionErgod}. Sections \ref{SectionNonErgod}, \ref{SectionChasseElliptique} and \ref{SectionErgod} can be read independently.

This work was partially supported by the french ANR ModGroup ANR-11-BS01-0020 and SGT ANR-11-BS01-0018. The authors would like to thank G. Courtois, S. Diverio, E. Falbel and F. Kassel for their support and inspiration.

%% file: Prelim.tex
\subsection{Notation and conventions}\label{SectionNotation}

We will often need to use explicit elements of the fundamental group of a
genus two closed surface, and here we fix the notation we will use along
the paper. We fix a base point and a system of four loops $a_1$, $b_1$,
$a_2$, $b_2$ as in Figure \ref{FigureMarquage}. Whenever $\gamma_1$ and $\gamma_2$
are loops based at the base point, $\gamma_1\gamma_2$ will be the element
of $\Gamma_2$ defined by the concatenation of these two paths:
we travel along $\gamma_1$ and then along $\gamma_2$, in this chronological
order. With this notation, we have
$a_1b_1a_1^{-1}b_1^{-1}a_2b_2a_2^{-1}b_2^{-1}=1$, and in whenever
$\gamma_1,\gamma_2\in\Gamma_2$ we will denote
$[\gamma_1,\gamma_2]=\gamma_1\gamma_2\gamma_1^{-1}\gamma_2^{-1}$.
\begin{figure}[hb]
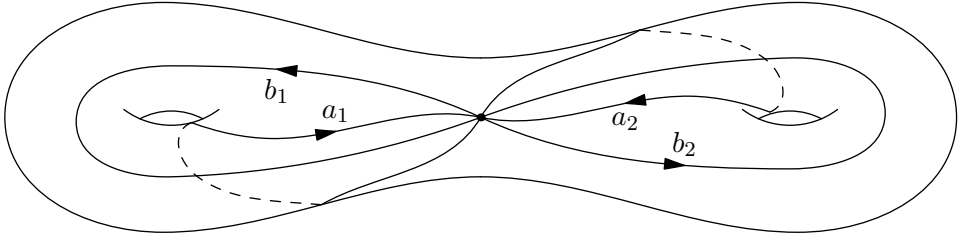

\begin{asy}
  import geometry;
  unitsize(3pt);

  //
  //
  path contg=((0,0){dir(180)}..(-40,-7){dir(180)}..(-60,7.5){dir(90)}..(-40,22){dir(0)}..(0,15){dir(0)});
  draw(contg);
  path contd=shift((0,15))*rotate(180)*contg;
  draw(contd);
  // Observation : le centre de la figure est en (0,7.5).

  //

  //
  //
  picture trou;
  draw (trou,(-6,8.5)..(0,6.5)..(6,8.5));
  draw (trou,(-4,7.3)..(0,8.3)..(4,7.3));
  //
  picture troudgauche=shift(-39,0)*trou;
  add(troudgauche);
  //
  picture trouddroite=shift(39,0)*trou;
  add(trouddroite);

  point PBas=(0,7.5);
  dot(PBas);

  // D'abord a1
  //
  path BasTrouG=shift(-39,0)*((-6,8.5)..(0,6.5)..(6,8.5));
  point MBA1=relpoint(BasTrouG,0.7);

  point BasA1=relpoint(contg,0.15);

  path DebutA1=(PBas{dir(168)}..(MBA1+(10,-2))..MBA1{dir(160)});
  path A1Suite=MBA1{dir(200)}..(MBA1+(1,-7))..BasA1{dir(-5)};
  path A1Fin=BasA1{dir(30)}..PBas{dir(60)};
  draw (reverse(DebutA1),black,Arrow(Relative(0.5)));
  label("{\small $a_1$}",relpoint(DebutA1,0.5),N,black);
  draw (A1Suite,black+dashed);
  draw (A1Fin,black);

  // Maintenant zeta
  path B1=(PBas{dir(155)}..(-39,14)..(-51,7)..(-39,0)..PBas{dir(20)});
  draw (B1,black,Arrow(Relative(0.24)));
  label("{\small $b_1$}",relpoint(B1,0.24),S,black);

  // Anse de DROITE
  path DebutA2=rotate(180,(0,7.5))*DebutA1;
  path A2Suite=rotate(180,(0,7.5))*A1Suite;
  path A2Fin=rotate(180,(0,7.5))*A1Fin;
  draw (reverse(DebutA2),black,Arrow(Relative(0.5)));
  label("{\small $a_2$}",relpoint(DebutA2,0.5),S,black);
  draw (A2Suite,black+dashed);
  draw (A2Fin,black);

  path B2=rotate(180,(0,7.5))*B1;
  draw (B2,black,Arrow(Relative(0.24)));
  label("{\small $b_2$}",relpoint(B2,0.24),N,black);

  //// Maintenant les courbes VERTES
  //point VertGauch1Bas = (-39,8.3);
  //point VertGauch1Haut = (-40,22);
  //path VertVertG1 = (VertGauch1Bas{dir(150)}..VertGauch1Haut{dir(65)});
  //path VertVertG2 = (VertGauch1Bas{dir(65)}..VertGauch1Haut{dir(150)});
  //draw(VertVertG1,heavygreen+1.1pt);
  //draw(VertVertG2,heavygreen+1.1pt+dashed);

  //path VertVertD1 = rotate(180,(0,7.5))*VertVertG1;
  //path VertVertD2 = rotate(180,(0,7.5))*VertVertG2;
  //draw(VertVertD1,heavygreen+1.1pt);
  //draw(VertVertD2,heavygreen+1.1pt+dashed);

  //path VertHorzG = ((-39,12.5){right}..(-31,7.5){down}..(-39,2.5){left}..(-47,7.5){up}..cycle);
  //draw(VertHorzG,heavygreen+1.1pt);

  //path VertHorzD = rotate(180,(0,7.5))*VertHorzG;
  //draw(VertHorzD,heavygreen+1.1pt);

  //// Maintenant la courbe ROUGE
  //path RougeGorge = ((0,9.5){left}..(-5,11)..(-39,18){left}..(-54,7.5){down});
  //path Rouge2 = reflect((0,7.5),(20,7.5))*RougeGorge;
  //path Rouge3 = rotate(180,(0,7.5))*RougeGorge;
  //path Rouge4 = reflect((0,7.5),(20,7.5))*Rouge3;

  //draw(RougeGorge,red+1.1pt);
  //draw(Rouge2,red+1.1pt);
  //draw(Rouge3,red+1.1pt);
  //draw(Rouge4,red+1.1pt);
\end{asy}
\caption{A marked surface of genus 2}
\label{FigureMarquage}
\end{figure}
On the other hand, most of the time we will think of $\psl$ as
acting on the left on $\HH$. Whenever $\pm A$ and $\pm B$ are
two elements of $\psl$ we want to think of $\pm BA$ as acting
first by $\pm A$ and then by $\pm B$. Therefore, instead of
considering morphisms from $\Gamma_2$ to $\psl$ with their
natural group structures, we will use the opposite group structure
on $\Gamma_2$. In other words, what we will call a morphism
from $\Gamma_2$ to $\psl$ will be a function $\rho$
satisfying $\rho(\gamma_1\gamma_2)=\rho(\gamma_2)\rho(\gamma_1)$,
for all $\gamma_1,\gamma_2\in\psl$. This convention is adopted
from \cite{GalloKapovichMarden}. Accordingly, if $\pm A$, $\pm B\in\psl$,
we will denote $[A,B]=B^{-1}A^{-1}BA$.

When $A\in\sldeuxR$ is a matrix and $\pm A$ is the
corresponding element in $\psl$, we will sometimes write $A$ instead
of $\pm A$, provided the distinction is not crucial.
On $\psl=\Isom^+(\HH)$, the displacement function will be denoted
by $\lambda$. Explicitely, if $\pm A\in\psl$,
\[
\lambda(A)=2\mathrm{arcch}(\max(1,\frac{|\tr(A)|}{2})).
\]
The usual functions $\mathrm{cosh}$, $\mathrm{sinh}$ will be
abreviated into $\cosh$, $\sinh$ for convenience. The unit
tangent bundle, $T_u\HH$ will be identified with
the topological space
$\psl$ via identification with the orbit
of the unit tangent vector pointing upwards at the point
$i$, in the upper half plane model. This yields the right-action of $\psl$ 
on $T_u\HH$, by right multiplication of
$\psl$ on itself. If $\ell,\theta\in\R$, the matrices
\[
T_\ell=
\left(\begin{array}{cc}e^{\ell/2} & 0 \\ 0 & e^{-\ell/2}\end{array}\right)
\text{ and }
R_\theta=
\left(\begin{array}{cc}\cos(\theta/2) & \sin(\theta/2) \\ -\sin(\theta/2) & \cos(\theta/2)\end{array}\right)
\]
act on the left on $\HH$, respectively, by translation along the axis $(0,\infty)$ and
rotation around $i$, and on the right on $T_u\HH$, respectively, by moving forward
by length $\ell$ in the direction given by the vector, and by rotating the vector by
angle $\theta$ at the same base point.
We will also use the notation
\[ R_l = R_{\frac{\pi}{2}},\quad R_r = R_{-\frac{\pi}{2}},\quad\text{and}\quad S=R_\pi.\]

\subsection{The Euler class}\label{SectionEuler}

Here we give a brief account on the Euler class. Excellent presentations
can be found in \cite{Ghys01} and in \cite{Calegari04}, Section 2, so
we refer the reader to these texts for details and complements.

Let $\psltild$ be the universal cover of $\psl$. We identify the kernel of the canonical surjection $\psltild\rightarrow\psl$ with $\pi_1(\psl)\simeq \Z$ via the path $\theta\mapsto R_\theta$ for $\theta\in [0,2\pi]$. 

The Euler class of a representation $\rho\colon\pi_1\Sigma_g\rightarrow\psl$ is
the element of $H^2(\Sigma_g,\Z)$ which measures the obstruction
to lifting $\rho$ to $\psltild$. It can be constructed as follows. Pick
an arbitrary set-theoretic section
$s\colon\psl\rightarrow\psltild$, and choose a triangulation
of $\Sigma_g$ with only one vertex, the base point of the surface.
Orient (arbitrarily) each edge of the triangulation.
To any oriented triangle $\sigma$ of the triangulation, with
boundary $\gamma_1^{\epsilon_1}\gamma_2^{\epsilon_2}\gamma_3^{\epsilon_3}$
(where the $\gamma_i's$ are edges with the chosen orientation),
associate the element
$\mathrm{eu}(\rho)(\sigma)=s(\rho(\gamma_3))^{\epsilon_3}s(\rho(\gamma_2))^{\epsilon_2}s(\rho(\gamma_1))^{\epsilon_1}$,
which is an element of the kernel of the map $\psltild\rightarrow\psl$.
This defines an element $\mathrm{eu}(\rho)\in H^2(\Sigma,\Z)$, and its evaluation on the fundamental
class defines a number, the {\em Euler class} of $\rho$, which we still
denote by $\mathrm{eu}(\rho)$.
Its parity measures the obstruction to lifting $\rho$ to $\sldeuxR$. 

If we build our surface of genus $g$ by gluing the faces of a $4g$-gon in the
standard way, this yields a practical formula, sometimes called the
{\em Milnor algorithm}:
given $\rho\colon\pi_1\Sigma_g\rightarrow\psl$, choose arbitrary
lifts $\widetilde{\rho(a_1)}$, \ldots, $\widetilde{\rho(a_g)}$ and
compute $\prod\left[\widetilde{\rho(a_i)},\widetilde{\rho(b_i)}\right]$.
As before, this defines an integer, the Euler number $\mathrm{eu}(\rho)$. 

For every hyperbolic element $\pm A$ of $\psl$, there is a natural path
from $\pm I_2$ to $\pm A$, which lies in the one-parameter subgroup
of $\psl$ containing $\pm A$. This defines a canonical lift
$\widetilde{A}\in\psltild$. Similarly, if $\Sigma$ is a surface
with boundary whose fundamental group has the presentation
\[\pi_1(\Sigma)=\langle a_1,b_1,\ldots,a_g,b_g,c_1,\ldots,c_n|[a_1,b_1]\cdots [a_g,b_g]c_1\cdots c_n\rangle,\]
and if each boundary
curve $c_i$ of $\Sigma$ is mapped by $\rho$ to a hyperbolic element, we can
define the Euler class of $\rho$ as
$$\eu(\rho)=\widetilde{C}_n\cdots\widetilde{C}_1\widetilde{B}_g^{-1}\widetilde{A}_g^{-1}\widetilde{B}_g\widetilde{A}_g\cdots\widetilde{B}_1^{-1}\widetilde{A}_1^{-1}\widetilde{B}_1\widetilde{A}_1\in \Z,$$
where $C_1,\ldots,C_n\in\psl$ are the images of the boundary curves by
$\rho$, $\widetilde{C_1}$, \ldots, $\widetilde{C_n}$ are their canonical lifts
and $\widetilde{A}_1,\ldots,\widetilde{B}_g$ are arbitrary lifts
to $\psltild$ of $\rho(a_1),\ldots, \rho(b_g)$.

This Euler class is additive in the following sense: if $\Sigma$ is the
union of two surfaces $\Sigma'$ and $\Sigma''$ glued along a family of
curves $d_1,\ldots, d_k$ then for any representation
$\rho:\pi_1(\Sigma)\to \psl$ which maps the curves $d_1,\ldots,d_k$ and the
boundary curves to hyperbolic elements one has
\[ \eu(\rho)=\eu(\rho')+\eu(\rho''), \]
where $\rho'$ and $\rho''$ stand for the restriction of $\rho$ to $\Sigma'$
and $\Sigma''$ respectively.

Also, it follows from the definition that for any representation $\rho$,
the Euler class $\mathrm{eu}(\rho)\in\Z$ does not vary in its
$\psl$-conjugacy class, it is $\Aut^+(\pi_1\Sigma)$-invariant, but its sign
changes to the opposite under conjugation by an orientation-reversing
isometry of $\HH$ or if we precompose $\rho$ by an element of $\Aut(\pi_1\Sigma)$
which does not preserve the fundamental class in $H_2(\Gamma_g,\Z)$.


As stated in the introduction, a foundational result of Goldman
asserts that for every $k$ such that $|k|\leq 2g-2$, $\eu^{-1}(k)$
is non-empty and connected, and if $k\neq 0$ it is a smooth symplectic manifold
of dimension $6g-6$.


\subsection{Miscellanea}\label{SectionReminders}

We will need a couple of results from Goldman's paper \cite{Goldman88}
on the connected components of $\psl$-representations:
\begin{theorem}[\cite{Goldman88}, Theorem 3.3]\label{fibr0}
  Let $\Sigma'$ be a genus $g-1$ surface ($g\geq 2$) with one boundary
  component, and let $k$ be an integer satisfying $|k|<2g-2$. Then the space
  $\mathcal{M}^k$ of classes of representations which are hyperbolic
  at the boundary, and of Euler class $k$, is non-empty, and connected.
\end{theorem}

The next statement is a consequence of Corollary 7.8 of \cite{Goldman88}.
\begin{theorem}[\cite{Goldman88}, Corollary 7.8]\label{fibr1}
   Let $\Sigma'$ be a genus $g-1$ surface ($g\geq 2$) with one boundary
   component; denote by $\gamma$ a curve freely homotopic to this
   boundary component. 
   Let $k$ be an integer satisfying $|k|<2g-2$.
   Let $A\colon[0,1]\rightarrow \psl$ be a continuous path taking values in the set of hyperbolic elements. 
   Then there exists a continuous path $t\mapsto\rho_t$ consisting of representations
   of relative Euler class $k$ such that for all $t$, $\rho_t(\gamma)=A(t)$.
\end{theorem}


The proof of Theorem \ref{Bowditch} will rely on the following two results.
The first one was proved recently independently by Deroin-Tholozan \cite{DeroinTholozan}
and Gu\'eritaud-Kassel-Wolff \cite{GueritaudKasselWolff} and the second one is due to Goldman,
\cite{Goldman03}.

\begin{theorem}\label{domination}
  Given a closed surface $\Sigma$ of genus $g>1$, an integer $k$ satisfying
  $|k|<2g-2$ and a representation $\rho\in \rep^k(\Sigma)$, there exists a
  Fuchsian representation $j\in \rep^{2-2g}(\Sigma)$ which dominates
  $\rho$ in the sense that for any $\gamma$ in $\pi_1(\Sigma)$ one has:
  $$|\tr \rho(\gamma)|\leq|\tr j(\gamma)|.$$
\end{theorem}

\begin{theorem}[\cite{Goldman03}, Theorem 6.3]\label{dixhuit}
  Let $\Sigma$ be a punctured torus and $\rho\in \mo^0(\Sigma)$ such that the
  boundary curve has trace in $[-18,18]$. Then there is a simple curve
  $\gamma$ in $\Sigma$ such that 
  $$|\tr \rho(\gamma)|\le 2.$$
\end{theorem}
Since this is not readily the statement of \cite{Goldman03},
we add a few lines which reduce our statement to the one of Goldman.
\begin{proof}
Fix a base point $x$ on the boundary of $\Sigma$ and choose two standard
generators $a,b$ of $\pi_1(\Sigma,x)$ so that the boundary of $\Sigma$ is
homotopic to $[a,b]$. Let $A,B\in$ SL$_2(\R)$ be two lifts of $\rho(a)$ and
$\rho(b)$. We can suppose that $|\tr([A,B])|> 2$ otherwise we are done. Set $x=\tr A$, $y=\tr B$
and $z=\tr(AB)$. A classical computation gives $\tr([A,B])=x^2+y^2+z^2-xyz-2$. The vanishing of the
Euler class implies that $\tr([A,B])>2$. We apply Theorem 6.3 of
\cite{Goldman03}: up to the action of $\mcg(\Sigma)$ we have either
$x<-2,y<-2,z<-2$ or one of the three coordinates is in $[-2,2]$. The first
case immediately implies that $\tr([A,B])>18$ hence we are in the second case
and one of the simple curves $a,b$ or $ab$ is sent to a non-hyperbolic element.
\end{proof}

\subsection{Geometry of commutators}\label{SectionCommut}

Let $A,B\in\sldeuxR$. It is well-known that the inequality $\tr([A,B])<2$ holds if and only if
$\pm A$ and $\pm B$ are two hyperbolic elements of $\psl$ whose
axes intersect at exactly one point in $\HH$
(\cite{Goldman03}, Lemma 4.4).

The aim of this section is to give a compass-and-straightedge construction
of the commutator $\pm [A,B]$ in this case. That is, to describe
the commutator in terms of two successive reflections along explicit axes.

We start with the case when the axes of $\pm A$ and $\pm B$ intersect
perpendicularly. Draw the perpendiculars to the axes
of $\pm A$ and $\pm B$ as in the Figure \ref{FigureComm},
\begin{figure}[ht]
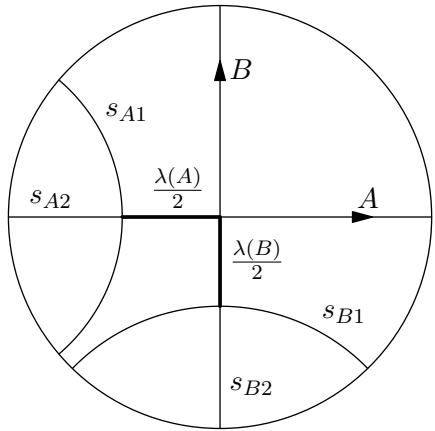

\begin{asy}
  import math;
  import hyperbolic_geometry;
  real  taille = 160;
  size(taille,taille);
  hyperbolic_point Orig = hyperbolic_point(0,0);
  hyperbolic_point PA = hyperbolic_point(1,180);
  hyperbolic_point PB = hyperbolic_point(0.9,-90);
  hyperbolic_line AxeA = hyperbolic_line(PA,Orig);
  hyperbolic_line AxeB = hyperbolic_line(PB,Orig);
  hyperbolic_line PerpA = hyperbolic_normal(AxeA,PA);
  hyperbolic_line PerpB = hyperbolic_normal(AxeB,PB);

  draw(unitcircle);
  draw(AxeA,Arrow(Relative(0.86)));
  draw(AxeB,Arrow(Relative(0.87)));
  label("{\small $A$}",relpoint(AxeA.to_path(),0.85),N);
  label("{\small $B$}",relpoint(AxeB.to_path(),0.85),E);
  draw(PerpA);
  draw(PerpB);

  label("{\small $s_{A1}$}",relpoint(PerpA.to_path(),0.2),NE);
  label("{\small $s_{B1}$}",relpoint(PerpB.to_path(),0.8),NE);
  label("{\small $s_{A2}$}",relpoint(AxeA.to_path(),0.1),N);
  label("{\small $s_{B2}$}",relpoint(AxeB.to_path(),0.1),E);

  draw(hyperbolic_segment(Orig,PA),black+1.4pt);
  label("{\small $\frac{\lambda(A)}{2}$}",relpoint(AxeA.to_path(),0.4),N);
  draw(hyperbolic_segment(Orig,PB),black+1.4pt);
  label("{\small $\frac{\lambda(B)}{2}$}",relpoint(AxeB.to_path(),0.4),E);
\end{asy}
\caption{A commutator}
\label{FigureComm}
\end{figure}
and identify the isometry $\pm A$ as a product of two
reflections:
$\pm A=s_{A2}\circ s_{A1}$. Similarly, $\pm B=s_{B2}\circ s_{B1}$.
It follows that $\pm [A,B]=s_{B1}s_{B2}s_{A1}s_{A2}s_{B2}s_{B1}s_{A2}s_{A1}$.
Whenever two lines intersect perpendicularly, the reflections along
these lines commute. It follows that $s_{B2}$ commutes with $s_{A1}$
and $s_{A2}$ and that $s_{A2}$ commutes with $s_{B1}$. Therefore,
\[ \pm [A,B]=\left(s_{B1}\circ s_{A1}\right)^2. \]
Depending on whether the axes of $s_{A1}$ and $s_{B1}$ intersect, we
are in one of the three situations pictured in Figure \ref{Figure3Comm}.
\begin{figure}[ht]
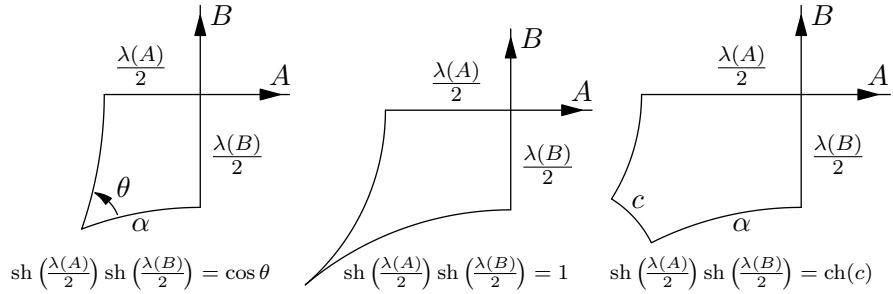

\begin{asy}
  import math;
  import hyperbolic_geometry;
  real  taille = 110;
  size(taille,taille);
  hyperbolic_point Orig = hyperbolic_point(0,0);
  hyperbolic_point PA = hyperbolic_point(0.67,180);
  hyperbolic_point PB = hyperbolic_point(0.8,-90);
  hyperbolic_line AxeA = hyperbolic_line(PA,Orig);
  hyperbolic_line AxeB = hyperbolic_line(PB,Orig);
  hyperbolic_line PerpA = hyperbolic_normal(AxeA,PA);
  hyperbolic_line PerpB = hyperbolic_normal(AxeB,PB);
  hyperbolic_point PComm = intersection(PerpA,PerpB);

  path segmentA=(PA.get_euclidean())..(0.3,0);
  draw(segmentA,Arrow(Relative(0.95)));
  label("{\small $A$}",relpoint(segmentA,0.95),N);
  path segmentB=(PB.get_euclidean())..(0,0.3);
  draw(segmentB,Arrow(Relative(0.95)));
  label("{\small $B$}",relpoint(segmentB,0.95),E);

  draw(hyperbolic_segment(PA,PComm));
  draw(hyperbolic_segment(PB,PComm));

  //draw(unitcircle);
  //draw(AxeA,Arrow(Relative(0.86)));
  //draw(AxeB,Arrow(Relative(0.87)));
  //draw(PerpA); draw(PerpB);
  //draw(hyperbolic_segment(Orig,PA)); draw(hyperbolic_segment(Orig,PB));
  //draw(hyperbolic_segment(PComm,PA)); draw(hyperbolic_segment(PComm,PB));

  //label("{\small $s_{A1}$}",relpoint(PerpA.to_path(),0.2),NE);
  //label("{\small $s_{B1}$}",relpoint(PerpB.to_path(),0.8),NE);
  //label("{\small $s_{A2}$}",relpoint(AxeA.to_path(),0.1),N);
  //label("{\small $s_{B2}$}",relpoint(AxeB.to_path(),0.1),E);

  //draw(hyperbolic_segment(Orig,PA),black+1.4pt);
  label("{\small $\frac{\lambda(A)}{2}$}",relpoint(AxeA.to_path(),0.4),N);
  //draw(hyperbolic_segment(Orig,PB),black+1.4pt);
  label("{\small $\frac{\lambda(B)}{2}$}",relpoint(AxeB.to_path(),0.4),E);

  label("{\small $\alpha$}",(-0.2,-0.44));
  draw(arc(PComm.get_euclidean(),0.13,22,70),Arrow());
  label("{\small $\theta$}",PComm.get_euclidean()+0.2*dir(45));

  label("{\scriptsize $\mathrm{sh}\left(\!\frac{\lambda(A)}{2}\!\right)\mathrm{sh}\left(\!\frac{\lambda(B)}{2}\!\right)=\cos\theta$}",(-0.2,-0.6));
\end{asy}
\begin{asy}
  import math;
  import hyperbolic_geometry;
  real  taille = 110;
  size(taille,taille);
  hyperbolic_point Orig = hyperbolic_point(0,0);
  hyperbolic_point PA = hyperbolic_point(1,180);
  // arcsinh(1/sinh(1))=0.7719368 et des brouettes.
  hyperbolic_point PB = hyperbolic_point(0.7719,-90);
  hyperbolic_line AxeA = hyperbolic_line(PA,Orig);
  hyperbolic_line AxeB = hyperbolic_line(PB,Orig);
  hyperbolic_line PerpA = hyperbolic_normal(AxeA,PA);
  hyperbolic_line PerpB = hyperbolic_normal(AxeB,PB);
  hyperbolic_point PComm = intersection(PerpA,PerpB);

  path segmentA=(PA.get_euclidean())..(0.3,0);
  draw(segmentA,Arrow(Relative(0.95)));
  label("{\small $A$}",relpoint(segmentA,0.95),N);
  path segmentB=(PB.get_euclidean())..(0,0.3);
  draw(segmentB,Arrow(Relative(0.95)));
  label("{\small $B$}",relpoint(segmentB,0.95),E);

  draw(hyperbolic_segment(PA,PComm));
  draw(hyperbolic_segment(PB,PComm));

  label("{\small $\frac{\lambda(A)}{2}$}",relpoint(AxeA.to_path(),0.4),N);
  label("{\small $\frac{\lambda(B)}{2}$}",relpoint(AxeB.to_path(),0.4),E);

  label("{\scriptsize $\mathrm{sh}\left(\!\frac{\lambda(A)}{2}\!\right)\mathrm{sh}\left(\!\frac{\lambda(B)}{2}\!\right)=1$}",(-0.2,-0.6));
\end{asy}
\begin{asy}
  import math;
  import hyperbolic_geometry;
  real  taille = 110;
  size(taille,taille);
  hyperbolic_point Orig = hyperbolic_point(0,0);
  hyperbolic_point PA = hyperbolic_point(1.2,180);
  hyperbolic_point PB = hyperbolic_point(0.8,-90);
  hyperbolic_line AxeA = hyperbolic_line(PA,Orig);
  hyperbolic_line AxeB = hyperbolic_line(PB,Orig);
  hyperbolic_line PerpA = hyperbolic_normal(AxeA,PA);
  hyperbolic_line PerpB = hyperbolic_normal(AxeB,PB);
  hyperbolic_line AxeComm = common_perpendicular(PerpA,PerpB);
  hyperbolic_point PCommA = intersection(PerpA,AxeComm);
  hyperbolic_point PCommB = intersection(PerpB,AxeComm);

  path segmentA=(PA.get_euclidean())..(0.3,0);
  draw(segmentA,Arrow(Relative(0.95)));
  label("{\small $A$}",relpoint(segmentA,0.95),N);
  path segmentB=(PB.get_euclidean())..(0,0.3);
  draw(segmentB,Arrow(Relative(0.95)));
  label("{\small $B$}",relpoint(segmentB,0.95),E);

  draw(hyperbolic_segment(PA,PCommA));
  draw(hyperbolic_segment(PB,PCommB));
  draw(hyperbolic_segment(PCommA,PCommB));

  label("{\small $\frac{\lambda(A)}{2}$}",relpoint(AxeA.to_path(),0.4),N);
  label("{\small $\frac{\lambda(B)}{2}$}",relpoint(AxeB.to_path(),0.4),E);

  label("{\small $\alpha$}",(-0.2,-0.44));
  label("{\small $c$}",(-0.55,-0.36));

  label("{\scriptsize $\mathrm{sh}\left(\!\frac{\lambda(A)}{2}\!\right)\mathrm{sh}\left(\!\frac{\lambda(B)}{2}\!\right)=\mathrm{ch}(c)$}",(-0.2,-0.6));
\end{asy}
\caption{Elliptic, parabolic or hyperbolic commutators of trace lower than $2$}
\label{Figure3Comm}
\end{figure}
\begin{remark}\label{RemarqueComm}
  In the first case, when $\pm [A,B]$ is elliptic, we also have the relation
  $\sinh(\alpha)\tanh\left(\frac{\lambda(B)}{2}\right)\tan(\theta)=1$,
  where $\alpha$ is the
  distance between the axis of $\pm A$ and the fixed point of $\pm [A,B]$,
  and where $\theta$ is one fourth of the rotation angle of $\pm [A,B]$.
  In the third case, when $\pm [A,B]$ is hyperbolic, then we have
  $\cosh(\alpha)\tanh\left(\frac{\lambda(B)}{2}\right)\tanh(c)=1$,
  where $\alpha$ is the distance between the axis of $\pm B$ and that of
  $\pm [A,B]$, and $c$ is one fourth of the displacement distance of
  $\pm [A,B]$. These formulas come readily from classical formulas in
  hyperbolic trigonometry, see eg \cite{Buser}, p. 454.
\end{remark}

Consider now the general case, when the axes of $\pm A$ and $\pm B$ need not
intersect perpendicularly.
First observe that the commutator $[A,B]=B^{-1}A^{-1}BA$ does not change
if we replace $A$ by $MA$, where $M$ is any matrix which commutes
with $B$. Up to replacing $B$ by $-B$, we may suppose that $\tr(B)>2$.
Then there is a well defined power $B^t$, for any real $t$. The element
$\pm B^t$ is then a hyperbolic displacement along the same axis as that
of $B$, with displacement depending on $t$, and $B^0={\mathrm I_2}$.
We want to understand geometrically the isometry $\pm B^tA$,
where $\pm B^t$ and $\pm A$ are hyperbolic isometries whose axes intersect
exactly once in $\HH$. This time we decompose $\pm B^t$ and $\pm A$ as
products of two rotations of angle $\pi$.
\begin{figure}[ht]
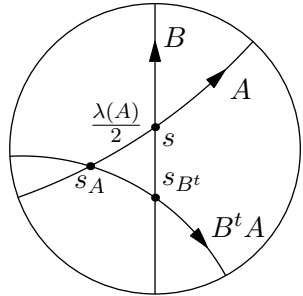

\begin{asy}
  import math;
  import hyperbolic_geometry;
  real  taille = 110;
  size(taille,taille);
  hyperbolic_point Orig = hyperbolic_point(0.3,90);
  hyperbolic_point PA = hyperbolic_point(1,-165);
  hyperbolic_point PBt = hyperbolic_point(0.7,-90);
  hyperbolic_line AxeA = hyperbolic_line(PA,Orig);
  hyperbolic_line AxeB = hyperbolic_line(PBt,Orig);
  hyperbolic_line AxeBtA = hyperbolic_line(PA,PBt);

  draw(unitcircle);
  draw(reverse(AxeA.to_path()),Arrow(Relative(0.86)));
  draw(AxeB,Arrow(Relative(0.87)));
  draw(AxeBtA,Arrow(Relative(0.87)));
  label("{\small $A$}",relpoint(reverse(AxeA.to_path()),0.85),SE);
  label("{\small $B$}",relpoint(AxeB.to_path(),0.85),dir(30));
  label("{\small $B^tA$}",relpoint(AxeBtA.to_path(),0.85),NE);
  dot(PA); dot(PBt); dot(Orig);

  label("{\small $s_{A}$}",PA.get_euclidean(),S);
  label("{\small $s$}",Orig.get_euclidean(),SE);
  label("{\small $s_{B^t}$}",PBt.get_euclidean(),NE);

  label("{\small $\frac{\lambda(A)}{2}$}",relpoint(AxeA.to_path(),0.62),N);
\end{asy}
\caption{A product of two hyperbolic elements with crossing axes}
\label{FigureHypCross}
\end{figure}
We have $\pm A=s\circ s_{A}$ and $\pm B^t=s_{B^t}\circ s$, where $s$, $s_A$
and $s_{B^t}$ are the rotations of angle $\pi$ as suggested in the above
picture.
It follows that $\pm B^t A=s_{B^t} s_A$ is a hyperbolic isometry, whose axis is
as in Figure \ref{FigureHypCross}. For a suitable $t\in\R$,
the fixed point of $s_{B^t}$ can be chosen to be the projection of the center of
$s_A$ on the axis of $\pm B$, so the axis of
$\pm B^t A$ is perpendicular to that of $\pm B$.
Finally, the construction we did
earlier enables to draw the commutator $\pm [A,B]=\pm [B^t A, B]$. Given two hyperbolic isometries $\pm A$, $\pm B$ whose axes cross, Figure \ref{FigureConstrucComm} provides a compass-and-straightedge construction of the commutator $\pm [A,B]$.
\begin{figure}[ht]
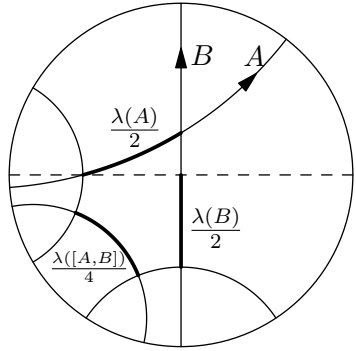

\begin{asy}
  import math;
  import hyperbolic_geometry;
  real  taille = 130;
  size(taille,taille);
  hyperbolic_point Orig = hyperbolic_point(0,0);
  hyperbolic_point PA1 = hyperbolic_point(1.3,180);
  hyperbolic_point PB = hyperbolic_point(1.2,-90);
  hyperbolic_line AxeA0 = hyperbolic_line(PA1,Orig);
  hyperbolic_line AxeB = hyperbolic_line(PB,Orig);
  hyperbolic_line PerpA = hyperbolic_normal(AxeA0,PA1);
  hyperbolic_line PerpB = hyperbolic_normal(AxeB,PB);
  hyperbolic_point PA2 = hyperbolic_point(0.5,90);
  hyperbolic_line AxeA = hyperbolic_line(PA1,PA2);
  hyperbolic_line AxeComm = common_perpendicular(PerpA,PerpB);
  hyperbolic_point PCommA = intersection(PerpA,AxeComm);
  hyperbolic_point PCommB = intersection(PerpB,AxeComm);

  draw(unitcircle);
  draw(reverse(AxeA.to_path()),Arrow(Relative(0.86)));
  draw(AxeB,Arrow(Relative(0.87)));
  label("{\small $A$}",relpoint(AxeA.to_path(),1-0.85),N);
  label("{\small $B$}",relpoint(AxeB.to_path(),0.85),E);
  draw(PerpA); draw(PerpB); draw(AxeComm);
  draw(AxeA0,black+dashed);

  //label("{\small $s_{A1}$}",relpoint(PerpA.to_path(),0.2),NE);
  //label("{\small $s_{B1}$}",relpoint(PerpB.to_path(),0.8),NE);
  //label("{\small $s_{A2}$}",relpoint(AxeA.to_path(),0.1),N);
  //label("{\small $s_{B2}$}",relpoint(AxeB.to_path(),0.1),E);

  draw(hyperbolic_segment(PA1,PA2),black+1.4pt);
  label("{\small $\frac{\lambda(A)}{2}$}",relpoint(AxeA.to_path(),0.6),N);
  draw(hyperbolic_segment(Orig,PB),black+1.4pt);
  label("{\small $\frac{\lambda(B)}{2}$}",relpoint(AxeB.to_path(),0.28),dir(40));
  draw(hyperbolic_segment(PCommA,PCommB),black+1.4pt);
  label("{\scriptsize $\frac{\lambda([A,B])}{4}$}",relpoint(AxeComm.to_path(),0.4)+(-0.05,-0.1),S);
\end{asy}
\caption{Compass-and-straightedge construction of a commutator}
\label{FigureConstrucComm}
\end{figure}

\subsection{Hyperbolic trigonometry}
Consider the polygons in $\HH$ presented in Figure \ref{hexagon}. We collect
here some formulas which will be used in the sequel, the most important of
which is Heron's formula and its variants. All of these formulas are
discussed in detail for instance in \cite{Fenchel} Chap. VI. 
\begin{figure}[htbp]
\centering
  \def\svgwidth{12cm}
 \executeiffilenewer{hexagon.svg}{hexagon.pdf}%
 {inkscape -z -D --file=hexagon.svg %
 --export-pdf=hexagon.pdf --export-latex}%
 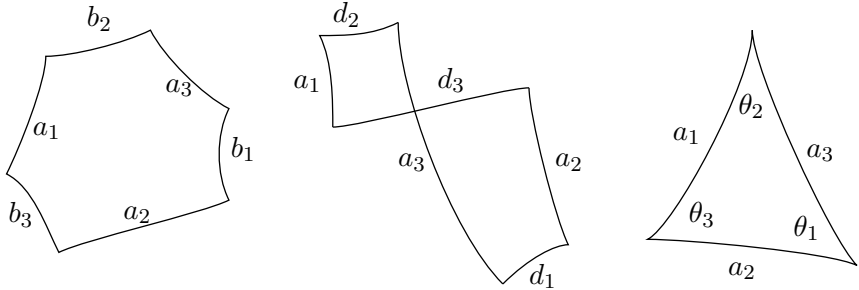%

  \caption{Polygons in hyperbolic plane}
  \label{hexagon} 
\end{figure}

\subsubsection{The right-angled hexagon}
We use three distinct indices $i,j,k$ in $\{1,2,3\}$. The length $b_i$ is given by the formula:
\begin{equation}\label{longhex}
\cosh(b_i)=\frac{\cosh(a_i)+\cosh(a_j)\cosh(a_k)}{\sinh(a_j)\sinh(a_k)}.
\end{equation}
Define the positive number $D'$ by the following formula where $2s=a_1+a_2+a_3$:
\begin{eqnarray*}
D'^2&=&4\cosh(s)\cosh(s-a_1)\cosh(s-a_2)\cosh(s-a_3)\\
&=&2\cosh(a_1)\cosh(a_2)\cosh(a_3)+\cosh(a_1)^2+\cosh(a_2)^2+\cosh(a_3)^2-1.
\end{eqnarray*}
This quantity satisfies the equation:
\begin{equation}\label{heronhex}
D'=\sinh(b_i)\sinh(a_j)\sinh(a_k).
\end{equation}

\subsubsection{The triangle}
In that case, Formula \eqref{longhex} becomes the following:
\begin{equation}\label{longtri}
\cos(\theta_i)=\frac{\cosh(a_j)\cosh(a_k)-\cosh(a_i)}{\sinh(a_j)\sinh(a_k)}.
\end{equation}
We define the positive number $D$ by the following formula:
\begin{eqnarray*}
D^2&=&4\sinh(s)\sinh(s-a_1)\sinh(s-a_2)\sinh(s-a_3)\\
&=&2\cosh(a_1)\cosh(a_2)\cosh(a_3)-\cosh(a_1)^2-\cosh(a_2)^2-\cosh(a_3)^2+1.
\end{eqnarray*}
The Heron formula becomes the following equation:
\begin{equation}\label{herontri}
D=\sin(\theta_i)\sinh(a_j)\sinh(a_k).
\end{equation}

\subsubsection{The self-intersecting hexagon}
In this case, the formulas are the same as in the case of the triangle,
up to some sign changes. For instance, if we denote by $a_3$
the length of the long edge, we have:
\begin{equation}\label{longself}
\cosh(d_3)=\frac{\cosh(a_3)-\cosh(a_1)\cosh(a_2)}{\sinh(a_1)\sinh(a_2)},\quad
\cosh(d_1)=\frac{\cosh(a_2)\cosh(a_3)-\cosh(a_1)}{\sinh(a_2)\sinh(a_3)},
\end{equation}
the formula for $d_2$ being similar to that of $d_1$. The Heron formula becomes the following:
\begin{equation}\label{heronself}
D=\sinh(d_i)\sinh(a_j)\sinh(a_k).
\end{equation}

%% file: hexagon.pdf_tex
\begingroup%
  \makeatletter%
  \providecommand\color[2][]{%
    \errmessage{(Inkscape) Color is used for the text in Inkscape, but the package 'color.sty' is not loaded}%
    \renewcommand\color[2][]{}%
  }%
  \providecommand\transparent[1]{%
    \errmessage{(Inkscape) Transparency is used (non-zero) for the text in Inkscape, but the package 'transparent.sty' is not loaded}%
    \renewcommand\transparent[1]{}%
  }%
  \providecommand\rotatebox[2]{#2}%
  \ifx\svgwidth\undefined%
    \setlength{\unitlength}{551.91328125bp}%
    \ifx\svgscale\undefined%
      \relax%
    \else%
      \setlength{\unitlength}{\unitlength * \real{\svgscale}}%
    \fi%
  \else%
    \setlength{\unitlength}{\svgwidth}%
  \fi%
  \global\let\svgwidth\undefined%
  \global\let\svgscale\undefined%
  \makeatother%
  \begin{picture}(1,0.34288533)%
    \put(0,0){\includegraphics[width=\unitlength]{hexagon.pdf}}%
    \put(0.0299413,0.19442124){\color[rgb]{0,0,0}\makebox(0,0)[lb]{\smash{$a_1$}}}%
    \put(0.17701722,0.24216093){\color[rgb]{0,0,0}\makebox(0,0)[lb]{\smash{$a_3$}}}%
    \put(0.13045528,0.10146521){\color[rgb]{0,0,0}\makebox(0,0)[lb]{\smash{$a_2$}}}%
    \put(0.24949238,0.16968576){\color[rgb]{0,0,0}\makebox(0,0)[lb]{\smash{$b_1$}}}%
    \put(0.00307685,0.09721063){\color[rgb]{0,0,0}\makebox(0,0)[lb]{\smash{$b_3$}}}%
    \put(0.09004704,0.31463608){\color[rgb]{0,0,0}\makebox(0,0)[lb]{\smash{$b_2$}}}%
    \put(0.87277869,0.0827156){\color[rgb]{0,0,0}\makebox(0,0)[lb]{\smash{$\theta_1$}}}%
    \put(0.75681849,0.09721063){\color[rgb]{0,0,0}\makebox(0,0)[lb]{\smash{$\theta_3$}}}%
    \put(0.81172175,0.21742546){\color[rgb]{0,0,0}\makebox(0,0)[lb]{\smash{$\theta_2$}}}%
    \put(0.73924657,0.18843538){\color[rgb]{0,0,0}\makebox(0,0)[lb]{\smash{$a_1$}}}%
    \put(0.8872737,0.16968577){\color[rgb]{0,0,0}\makebox(0,0)[lb]{\smash{$a_3$}}}%
    \put(0.80030356,0.0392305){\color[rgb]{0,0,0}\makebox(0,0)[lb]{\smash{$a_2$}}}%
    \put(0.31889068,0.24641553){\color[rgb]{0,0,0}\makebox(0,0)[lb]{\smash{$a_1$}}}%
    \put(0.60879132,0.15944533){\color[rgb]{0,0,0}\makebox(0,0)[lb]{\smash{$a_2$}}}%
    \put(0.43485093,0.15944534){\color[rgb]{0,0,0}\makebox(0,0)[lb]{\smash{$a_3$}}}%
    \put(0.57980126,0.02899006){\color[rgb]{0,0,0}\makebox(0,0)[lb]{\smash{$d_1$}}}%
    \put(0.36237575,0.31889068){\color[rgb]{0,0,0}\makebox(0,0)[lb]{\smash{$d_2$}}}%
    \put(0.47833603,0.24641553){\color[rgb]{0,0,0}\makebox(0,0)[lb]{\smash{$d_3$}}}%
  \end{picture}%
\endgroup%

%% file: NonErgod.tex
The aim of this section is to prove Proposition \ref{nonergod}.

Recall that a surface of genus $2$ admits a hyperelliptic involution $\varphi$,
well-defined in $\Out^+(\Gamma_2)$; actually it generates the
center of $\Out^+(\Gamma_2)$.

A lift of this hyperelliptic involution,
$\overline\varphi\in\Aut(\Gamma_2)$, can be chosen as follows
(with the notation of Section \ref{SectionNotation}):
$\overline\varphi(a_1)=b_1a_1^{-1}b_1^{-1}$,
$\overline\varphi(b_1)=b_1^{-1}$,
$\overline\varphi(a_2)=\gamma b_2a_2^{-1}b_2^{-1}\gamma^{-1}$,
$\overline\varphi(b_2)=\gamma b_2^{-1}\gamma^{-1}$,
with $\gamma=b_1a_1^{-1}b_1^{-1}a_2$.

This element of the mapping class group has a very remarkable property:
it preserves every simple curve. More precisely, for every $\gamma\in\Gamma_2$
which is homotopic to a simple closed curve, $\overline{\varphi}(\gamma)$
is conjugate in $\Gamma_2$ to $\gamma$ or $\gamma^{-1}$ whether $\gamma$ is
separating or not. This was first observed in \cite{HaasSusskind}.

This can be used to define a continuous
map $s\colon\mo^0\rightarrow\{\pm 1\}$, as follows.
Let $\mathcal{R}^0_{ne}$ denote the subspace of $\Hom(\Gamma_2,\psl)$
consisting of non-elementary representations of Euler class $0$
and let $\rho$ be an element of $\mathcal{R}^0_{ne}$. Since it has Euler class $0$, it admits
a lift $\bar{\rho}\colon\Gamma_2\rightarrow\sldeuxR$.
For every $\gamma\in\Gamma_2$
homotopic to a simple closed curve, $\overline{\varphi}(\gamma)$ being
conjugate to $\gamma$ or $\gamma^{-1}$, we have
$\tr(\bar\rho(\overline{\varphi}(\gamma)))=\tr(\bar\rho(\gamma))$. Indeed,
in $\sldeuxR$, every element has the same trace
as its inverse.

It is well-known that two representations $\rho,\rho'$ in $\Hom(\Gamma_2,\sldeuxR)$
satisfy $\tr\rho(\gamma)=\tr \rho'(\gamma)$ for all simple curves $\gamma$ if and only if they satisfy the same equation 
for all $\gamma$ in $\Gamma_2$. This fact is proven in \cite{GoldmanXia11} Theorem 2.1, and also follows from
trace identities by induction on the number of double points of $\gamma$. 

This implies that for all $\gamma\in\Gamma_2$ we have: 
$\tr(\bar\rho(\gamma))=\tr(\bar\rho\circ\overline{\varphi}(\gamma))$.
Since $\rho$ is non-elementary,
there exists $g\in\mathrm{GL}_2(\R)$ such that for all $\gamma\in\Gamma_2$,
$\bar\rho(\overline{\varphi}(\gamma))=g\bar\rho(\gamma)g^{-1}$. This element $g$ is well-defined up to right-multiplication by the centralizer of $\rho$ which is trivial in this case because $\rho$ is non-elementary. 
Writing $c(\rho)=\pm g$, we define a map $c\colon \mathcal{R}^0_{ne}\rightarrow\textrm{PGL}_2(\R)$.
Notice that this map does not depend on the lift $\bar\rho$ of $\rho$. 
Moreover, if we replace $\bar\phi$ with $\gamma \bar\phi \gamma^{-1}$, $c(\rho)$ is replaced with $\rho(\gamma)c(\rho)$. 

Now, the group $\textrm{PGL}_2(\R)=\Isom^+(\HH)\cup\Isom^-(\HH)$
has two connected components, so the formula $s(\rho)=\operatorname{sign} \det(c(\rho))$ defines a $\Aut^+(\Gamma_2)$-invariant map
\[ s\colon \mathcal{R}^0_{ne}\rightarrow\{-1,+1\}. \]

The map $\rho\mapsto \rho\circ\phi$ is continuous and commutes with the projection $\pi:\mathcal{R}^0_{ne}\to \mathcal{R}^0_{ne}/\textrm{PGL}_2(\R)$ which is a principal $\textrm{PGL}_2(\R)$-bundle. Hence the map $c$ (and therefore the map $s$) which satisfies $\rho\circ\bar\phi=c(\rho)\rho c(\rho)^{-1}$ is continuous. 
A more explicit proof of the continuity of $s$ goes as follows.

Pick a representation $\rho\in \mathcal{R}^0_{ne}$. The group $\rho(\Gamma_2)$,
being non-elementary, contains Schottky groups. In such a group,
there exists two hyperbolic elements $A,B\in\psl$ with crossing
axes. Let $\gamma_1$, $\gamma_2\in\Gamma_2$ be such
that $\rho(\gamma_1)=A$ and $\rho(\gamma_2)=B$. Orient the axes
of $A$ and $B$ forward the attractive points of $A$ and $B$.
Then at the intersection of these axes, these two directions define
an orientation of the hyperbolic plane.
As above, there exists a unique $\pm g\in\Isom(\HH)$ such that
$\rho\circ\overline{\varphi}=g\rho g^{-1}$. In particular,
$g$ conjugates $(\rho(\gamma_1),\rho(\gamma_2))$ into
$(\rho(\overline{\varphi}(\gamma_1)),\rho(\overline{\varphi}(\gamma_1)))$
so the property of whether $\pm g\in\Isom(\HH)$ preserves, or
reverses the orientation, is prescribed by the elements $\rho(\gamma_i)$
and $\rho(\overline{\varphi}(\gamma_i))$, which vary continuously with $\rho$.
This proves that $s$ is continuous.

In particular, it descends to a continuous function,
$s\colon\mo^0(\Gamma_2)\rightarrow\{-1,+1\}$, which is $\mcg(\Sigma_2)$-invariant.
We define
\[ \mo^0_+(\Gamma_2)=s^{-1}(+1)\textrm{ and } \mo^0_-(\Gamma_2)=s^{-1}(-1).\]
In order to complete the proof of Proposition \ref{nonergod}, it
remains to prove that these two open sets are non empty.
This amounts to studying representations individually, so now we turn to the
proof of Proposition
\ref{propintrononergod}.

Let $\rho$ be a non-elementary representation of Euler class $0$.
Suppose first that $\rho(a_1)$ and $\rho(b_1)$ are hyperbolic elements
with crossing axes. Note that the couples $(a_1,b_1)$ and
$(\bar\varphi(a_1^{-1}),\bar\varphi(b_1^{-1}))$ are conjugated by $b_1$.
It follows that $(\rho(a_1),\rho(b_1))$ and
$(\rho\circ\bar\varphi(a_1),\rho\circ\bar\varphi(b_1))$ are both
couples of hyperbolic elements whose axes intersect once in $\HH$,
and their oriented axes both define
the same orientation of the plane. In particular, the element of
$\Isom(\HH)$ conjugating $(\rho(a_1),\rho(b_1))$ into
$(\rho\circ\bar\varphi(a_1),\rho\circ\bar\varphi(b_1))$ is
orientation-preserving, hence $\rho\in\mo^0_+(\Sigma_2)$.

Similarly,
we observe that:
\begin{itemize}
  \item[-] if $\rho(a_1)$ and $\rho(b_1)$ are hyperbolic
    elements whose axes are disjoint in $\overline{\HH}$ then
    $\rho\in\mo^0_-(\Sigma_2)$;
  \item[-] if $\rho(a_1)$ is elliptic but not of order 2,
    or parabolic, and if $\rho(b_1)$ is hyperbolic and its
    axis does not contain the fixed point of $\rho(a_1)$,
    then $\rho\in\mo^0_-(\Sigma_2)$.
\end{itemize}
If a non-elementary representation $\rho$ sends some non-separating
simple closed curve (say, $a$) to
an elliptic element, then up to deforming $\rho$ we can suppose
that $\rho(a)$ has infinite order. Then it is easy to produce a
non-separating simple closed curve $b$ such that $i(a,b)=1$
and such that $\rho(b)$ is hyperbolic and its axis does not contain
the fixed point of $\rho(a)$, hence $\rho$ belongs to $\mo^0_-(\Gamma_2)$.
In particular, elements of $\mo^0_+(\Gamma_2)$ do not map any
non-separating simple closed curve to an elliptic or parabolic
element. And such a representation has to map $\rho(a_1)$
and $\rho(b_1)$ to hyperbolic elements with crossing axes,
(in which case $\tr([\rho(a_1),\rho(b_1)])<2$), unless
their commutator is mapped to $1$.
This finishes to prove the part of Proposition \ref{propintrononergod}
which concerns $\mo^0_+(\Gamma_2)$.

In order to finish the proof of Proposition \ref{propintrononergod},
it remains to recall two elementary and well-known facts. One of them was
recalled in Section \ref{SectionCommut}: if $\tr([\rho(a_1),\rho(b_1)])<2$
then $\rho(a_1)$ and $\rho(b_1)$ are hyperbolic and their axes cross
once in $\HH$. This implies $\rho\in\mo^0_+(\Gamma_2)$. The
second is that if $[\rho(a_1),\rho(b_1)]$ is parabolic, and of trace
$2$ (not $-2$), then $\rho$ is an elementary representation.

\medskip

Finally, let us discuss the remarks following Corollary \ref{CorNonErgod}.

In $\textrm{SL}(n,\R)$, for $n\geq 3$, matrices are generically not
conjugate to their inverses. Let $\tau$ be the composition of the transposition
and the inversion in $\textrm{SL}(n,\R)$. As above,
let $\bar\varphi$ be a lift in $\Aut^+(\Gamma_2)$ of the
hyperelliptic involution. If $\rho\in\Hom(\Gamma_2,\textrm{SL}(n,\R))$,
let $\rho'=\rho\circ\bar\varphi$ and
$\rho''=\tau\circ\rho\circ\bar\varphi$.
Then $\rho$ and $\rho'$ send separating simple closed curves to elements
of the same trace, and $\rho$ and $\rho''$ send non-separating simple
closed curves to elements of the same trace. However, $\rho$ and $\rho'$
(resp. $\rho''$) can very well be non-conjugate: for this it suffices
to define a representation sending $a_1$ (resp. $[a_1,b_1]$) to
some element in $\textrm{SL}(n,\R)$ which is not conjugate to its inverse.

For what concerns the first remark, we can easily produce a representation
$\rho\in\Hom(\Gamma_2,\psl)$ which sends $a_1$, $b_1$ to
hyperbolic elements whose axes cross once in $\HH$ and such that
$\tr([\rho(a_1),\rho(b_1)])<-2$, and which sends $a_2$, $b_2$ to
hyperbolic elements whose axes do not cross in $\overline{\HH}$. If $\rho$ is such a representation,
$\rho\circ\bar\varphi$ and $\rho$ have the same traces on
simple closed loops, but it follows from the preceding discussion that they cannot be conjugated
neither by an element of $\Isom^+(\HH)$ nor by an element of $\Isom^-(\HH)$. It then follows eg from
\cite{Maxime}, Proposition 2.15, that the corresponding
trace functions do not coincide.


%% file: Param.tex
Given a closed surface of genus 2, one can decompose it as the union of two
pairs of pants $P_1$ and $P_2$ whose common boundary is the disjoint union of
three curves $\gamma_1,\gamma_2,\gamma_3$. A representation
$\rho:\Gamma_2\to\psl$ such that $\rho(\gamma_1),\rho(\gamma_2)$ and
$\rho(\gamma_3)$ are hyperbolic will restrict on $P_1$ (resp. $P_2$) to a
representation $\rho_1$ (resp. $\rho_2)$ of Euler class $-1,0$ or $1$.
These representations are almost completely determined by the traces
of the boundary curves: hence we will build our coordinate system by finding
an explicit representation for each case.

\subsection{Pants representations}
Let $P$ be a pair of pants obtained by gluing together two hexagons as in
Figure \ref{pants}. 
\begin{figure}[htbp]
\centering
  \def\svgwidth{6cm}
 \executeiffilenewer{pantalon.svg}{pantalon.pdf}%
 {inkscape -z -D --file=pantalon.svg %
 --export-pdf=pantalon.pdf --export-latex}%
 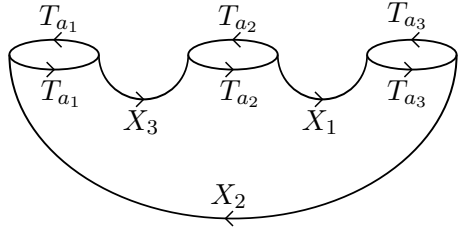%

  \caption{Decomposition of a pair of pants}
  \label{pants} 
\end{figure}
We view $P$ as a cellular complex and construct a 
1-cocycle on $P$ with values in $\psl$. Precisely, a 1-cocycle is a tuple
$(g_e)$ of elements of $\psl$ indexed by oriented edges of $P$ and
satisfying the following two relations: 
\begin{itemize}
\item[-] if $-e$ denotes the edge $e$ with opposite orientation then
  $g_{-e}=g_e^{-1}$ 
\item[-] if $c$ is a cell whose boundary is the composition of the oriented
  paths $e_1,\ldots,e_k$ then $g_{e_k}\cdots g_{e_1}=1$. 
\end{itemize}
Let $a_1,a_2,a_3$ be three positive real numbers.
%
Given three elements $X_1$, $X_2$, $X_3\in \psl$, the data given in Figure \ref{pants}
defines a 1-cocycle if and only if the two following conditions are satisfied
(one for each hexagon):
\begin{equation}\label{cocycle}
T_{a_2}X_3T_{a_1}X_2T_{a_3}X_1=\pm I_2,\quad T_{-a_2}X_3T_{-a_1}X_2T_{-a_3}X_1=\pm I_2.
\end{equation}

{\bf Case $\eu=\pm 1$.} \\
Let $H$ be the right-angled hexagon in $\HH$ shown in Figure \ref{hexagon}
(unique up to isometry) whose lengths are cyclically given by
$a_1,b_2,a_3,b_1,a_2,b_3$, the lengths $b_1,b_2,b_3$ being
determined by $a_1,a_2,a_3$ through Equation \eqref{longhex}.
%
Remembering the action of the matrices $R_\theta$ and $T_\ell$ on the right
on $T_u\HH$, and writing that
a unit vector which follows the sides of an hexagon comes back to its initial
value, we get the following relation:
$$T_{a_2}R_rT_{b_3}R_rT_{a_1}R_rT_{b_2}R_rT_{a_3}R_rT_{b_1}R_r=\pm I_2.$$
This implies the first part of Equation \eqref{cocycle}. Conjugating the
equation by $U=\begin{pmatrix} 1& 0\\ 0 & -1\end{pmatrix}$ we get the same
equation where $R_r$ has been replaced by $R_l$. Using $R_r=SR_l$ and
$ST_aS=-T_{-a}$ we derive the second equation. 

It remains to compute the Euler class of such a representation, but it is
well-known to be equal to $\pm 1$, as it is the holonomy of the hyperbolic
structure on the pair of pants given by gluing two isometric hexagons.

If we reflect the hexagon, the $b_i's$ and the Euler class change signs. 

{\bf Case $\eu=0$, $\Delta\neq 0$.} \\
For pairs of pants of Euler class $0$, we will need to consider the
following quantity, which is positive if and only if no $a_i$ is greater than the sum of the two others.
$$\Delta=2\cosh(a_1)\cosh(a_2)\cosh(a_3)-\cosh(a_1)^2-\cosh(a_2)^2-\cosh(a_3)^2+1.$$

We consider first the case when $\Delta$ is positive. 
In that case, there exists a hyperbolic triangle whose sides have lengths
$a_1,a_2,a_3$ and the opposite angles are
$\theta_1,\theta_2,\theta_3\in (0,\pi)$. As before, walking a unit vector
along the sides of a triangle gives the formula:
$$T_{a_2}R_{\pi+\theta_3}T_{a_1}R_{\pi+\theta_2}T_{a_3}R_{\pi+\theta_1}=\pm I_2.$$
Putting $X_1=SR_{\theta_1},X_2=SR_{\theta_2}$ and $X_3=SR_{\theta_3}$ we get
a solution of the first equation. Computing the transpose and inverse of this
equation, we get the second equation and hence a 1-cocycle as expected.

In the case when $a_3> a_1+a_2$, there is a self-intersecting right-angled
hexagon as in Figure \ref{hexagon}. Walking along the hexagon we obtain the
equation: 
$$T_{a_2}R_lT_{d_3}R_rT_{a_1}R_rT_{d_2}R_rT_{a_3}R_lT_{d_1}R_l=\pm I_2.$$
It follows that setting $X_1=R_lT_{d_1}R_l,X_2=R_rT_{d_2}R_r$ and
$X_3=R_lT_{d_3}R_r$ we get the first equation for the 1-cocycle.
Again, conjugating by $U$ and writing $R_l=SR_r$ and $R_r=-SR_l$
we see that the first equation implies the second one.

The remaining cases are treated in the same way by cyclically permuting the
indices. In order to compute the Euler class in all these cases, we observe
that we can deform them to the case when $a_3=a_1+a_2$, that is
$\theta_1=\theta_2=0$ and $\theta_3=\pi$. In that case, $X_1=X_2=S$ and
$X_3=1$; it follows that the representation that we construct is diagonal
and hence has Euler class 0. 

If we reflect the figure, the angles $\theta_i$ (resp. the lengths $d_i$)
are changed to their opposite whereas the Euler class still vanishes.
We will refer to these to cases by writing $\eu=0^+$ in the first case
and $\eu=0^-$ in the second case. This makes a subtle difference between
the cases $\eu=0$ and $\eu=\pm 1$.  We almost obtained the complete list of
representations of the pair of pants as it is stated in the following lemma. 

\begin{lemma}
  Consider a pair of pants $P$ with boundary curves $\gamma_1,\gamma_2,\gamma_3$ and three positive real numbers
  $a_1,a_2,a_3$ such that $\Delta\neq 0$.
  There are up to conjugation 4 representations $\rho:\pi_1(P)\to \psl$ with $|\tr \rho(\gamma_i)|=2\cosh(a_i)$ for
  $i\in \{1,2,3\}$. These representations are the ones described above and have respective Euler class $-1,1,0^+$ and $0^-$. 
\end{lemma}
\begin{proof}
Recall that $\pi_1(P)$ is a free group with two generators $a$ and $b$. Taking
$A,B\in$ SL$_2(\R)$ two lifts of $\rho(a)$ and $\rho(b)$, we define a
representation $\overline{\rho}:F_2\to\textrm{SL}_2(\R)$ by setting $\overline{\rho}(a)=A$ and $\overline{\rho}(b)=B$. They should satisfy
$$|\tr(A)|=2\cosh(a_1),\quad |\tr(B)|=2\cosh(a_2),\quad|\tr(AB)|=2\cosh(a_3).$$
By changing the signs of $A$ and $B$, one can suppose that
$\tr(A)=2\cosh(a_1)$ and $\tr(B)=2\cosh(a_2)$. The sign
of $\tr(AB)$ is $(-1)^{\eu(\rho)}$, see Section \ref{SectionEuler}. Hence it remains to recall the
fact that a representation $\rho:F_2\to$SL$_2(\R)$ is determined up to
conjugacy in GL$_2(\R)$ by the quantities $u=\tr(a),v=\tr(b)$ and $w=\tr(ab)$ provided that
$\tr([a,b])=u^2+v^2+w^2-uvw-2\ne 2$, see for instance
Proposition 1.5.2 of \cite{CullerShalen}.
\end{proof}

{\bf Case $\eu=0$, $\Delta=0$.} \\
In this case, up to cyclically permuting the indices, we may suppose that $a_3=a_1+a_2$.
The representation of $\pi_1(P)$ can be diagonal,
upper triangular, or lower triangular. The diagonal case is already
covered by putting $d_1=d_2=d_3$ above. In the upper triangular case
(but not diagonal)
a quick computation yields:
\[
X_1=S\left(\begin{array}{cc}1 & \epsilon\sinh a_1 \\ 0 & 1\end{array}\right),
\text{ }
X_2=\left(\begin{array}{cc}1 & -\epsilon\sinh a_2 \\ 0 & 1\end{array}\right) S
\text{ and }
X_3=\left(\begin{array}{cc}1 & \epsilon\sinh a_1 \\ 0 & 1\end{array}\right),
\]
where $\epsilon=-1$ or $+1$.
The lower diagonal case is obtained with the same matrices, simply by
replacing each parabolic matrix above by its transpose.

\subsection{Gluing pants representations}\label{SectionGluing}
Let $\Sigma$ be a genus 2 surface obtained by gluing two pairs of pants $P_1$
and $P_2$ along three curves $\gamma_1,\gamma_2$ and $\gamma_3$.
Let $k$ be an integer and $\rho\in\mo^{k}(\Sigma)$ be a representation such
that $\rho(\gamma_i)$ is hyperbolic for all $i\in\{1,2,3\}$. Pick $a_i> 0$ so
that $\tr\rho(\gamma_i)=2\cosh(a_i)$.

Let $\epsilon_1\in\{-1,0,1\}$ (resp. $\epsilon_2$) be the Euler class of the
restriction $\rho_1$ of $\rho$ to $P_1$ (resp. $\rho_2$). The additivity
of the Euler class implies the inequality $\epsilon_1+\epsilon_2=k$. Moreover,
we know from the preceding lemma that the representation $\rho_i$ is
conjugate to exactly one of our pants representations. 

To be more precise, we fix a base point on $\gamma_1$ and let $\rho_1$ and
$\rho_2$ be the pants representations on $P_1$ and $P_2$ corresponding to
the choice of Euler class. In Figure \ref{genus2}, we will denote by
$X_1,X_2,X_3$ and $Y_1,Y_2,Y_3$ the matrices associated to these
pants.
Then there exists $g_1$ and $g_2$ so that $\rho_1=g_1 \rho'_1 g_1^{-1}$ and
$\rho_2=g_2 \rho'_2 g_2^{-1}$. 
We have from the identity $\rho'_1(\gamma_1)=\rho'_2(\gamma_2)=T_{2a_1}$
the following relation:
$$\rho(\gamma_1)=\rho_1(\gamma_1)=g_1T_{2a_1}g_1^{-1}=\rho_2(\gamma_2)=g_2T_{2a_1}g_2^{-1}.$$
We deduce from this that $g_2^{-1}g_1$ commutes with $T_{2a_1}$ and hence has
the form $T_{t_1}$ for some $t_1\in \R$. Applying the same trick for the three
boundary curves, we find that the initial representation $\rho$ can be
described by the following cocycle represented in Figure \ref{genus2}.

 \begin{figure}[htbp]
\centering
  \def\svgwidth{10cm}
 \executeiffilenewer{genus2.svg}{genus2.pdf}%
 {inkscape -z -D --file=genus2.svg %
 --export-pdf=genus2.pdf --export-latex}%
 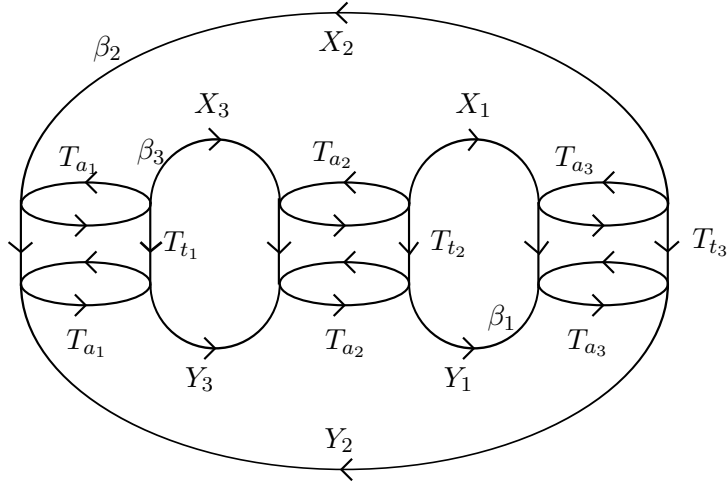%

  \caption{Cocycle on a genus 2 surface}
  \label{genus2} 
\end{figure}
We will denote by
$\rho=\rho^{\epsilon_1,\epsilon_2}_{a_1,a_2,a_3}(t_1,t_2,t_3)$ the
representation associated to this cocycle.
We notice that the Dehn twist $\tau_1$ along $\gamma_1$ acts on $\rho$
in the following way:
$$\tau_1. \rho=\rho^{\epsilon_1,\epsilon_2}_{a_1,a_2,a_3}(t_1\pm 2a_1,t_2,t_3).$$
We have the corresponding formulas for the two other Dehn twists.

\subsection{Explicit computations}\label{explicit}
We provide here explicit formulas for traces of curves which be useful in the next section. Let $\beta_1,\beta_2,\beta_3$ be the three curves appearing in Figure \ref{genus2}.

We have $\rho(\beta_1)=X_1^{-1}T_{-t_3}Y_1T_{t_2}$ and $\rho(\gamma_2)=T_{2a_2}$. Moreover, these two curves intersect once hence their commutator is a separating curve that we denote by $\delta_3$. 

\subsubsection{Case $\epsilon_1=1,\epsilon_2=-1$} We have in that case:
$$X_1=Y_1=R_r T_{b_1} R_r, \quad X_2=Y_2=R_r T_{b_2} R_r, \quad X_3=Y_3=R_r T_{b_3} R_r.$$
A direct computation gives:
\begin{eqnarray*}\label{betaunmoinsun}
\tr\rho(\beta_1)&=&\tr(X_1^{-1}T_{-t_3}Y_1T_{t_2})\\
&=&2\cosh(\frac{b_1}{2})^2\cosh(\frac{t_2+t_3}{2})-2\sinh(\frac{b_1}{2})^2\cosh(\frac{t_2-t_3}{2})\\
&=&2\cosh(\frac{t_2}{2})\cosh(\frac{t_3}{2})+2\cosh(b_1)\sinh(\frac{t_2}{2})\sinh(\frac{t_3}{2}).
\end{eqnarray*}
Similarly, one has
\begin{eqnarray}\label{epsunmoinsun}
\tr \rho(\delta_3)&=&2\cosh(a_2)^2-2\cosh(b_1)^2\sinh(a_2)^2+2\sinh(b_1)^2\sinh(a_2)^2\cosh(t_3)\notag{}\\
&=&2+4\sinh(a_2)^2\sinh(b_1)^2\sinh(\frac{t_3}{2})^2.
\end{eqnarray}
In particular, it follows that $\rho\in \mo^0_{-}$.

\subsubsection{Case $\epsilon_1=\epsilon_2=0^+$ and $\Delta<0$} In that case, we can suppose up to a cyclic change of indices that $a_3>a_1+a_2$. We set 
$$X_1=Y_1=R_lT_{d_1}R_l,\quad X_2=Y_2=R_rT_{d_2}R_r,\quad X_3=Y_3=R_lT_{d_3}R_r.$$
A direct computation gives:
\begin{eqnarray}\label{ep0deltaneg}
\tr(\delta_3)&=&2\cosh(a_2)^2-2\cosh(d_1)^2\sinh(a_2)^2+2\cosh(t_3)\sinh(a_2)^2\sinh(d_1)^2\notag{}\\
&=&2+4\sinh(\frac{t_3}{2})^2\sinh(d_1)^2\sinh(a_2)^2.
\end{eqnarray}
In particular, $\rho$ belongs again to $\mo^0_{-}$. 

\subsubsection{Case $\epsilon_1=0^+,\epsilon_2=0^-$ and $\Delta<0$}\label{LeCas3}
This case is the same as the preceding one for $X_1,X_2,X_3$ but we have now:
$$Y_1=R_lT_{-d_1}R_l,\quad Y_2=R_rT_{-d_2}R_r,\quad Y_3=R_lT_{-d_3}R_r.$$
\begin{eqnarray*}
\tr \rho(\beta_1)&=&2\sinh(\frac{t_2}{2})\sinh(\frac{t_3}{2})+2\cosh(d_1)\cosh(\frac{t_2}{2})\cosh(\frac{t_3}{2})\\
\tr \rho(\beta_2)&=&2\sinh(\frac{t_1}{2})\sinh(\frac{t_3}{2})+2\cosh(d_2)\cosh(\frac{t_1}{2})\cosh(\frac{t_3}{2})\\
\tr \rho(\beta_3)&=&-2\sinh(\frac{t_2}{2})\sinh(\frac{t_1}{2})+2\cosh(d_3)\cosh(\frac{t_2}{2})\cosh(\frac{t_1}{2}).
\end{eqnarray*}
This gives after a direct computation:
\begin{equation}\label{ep0mdeltaneg}
\tr\rho(\delta_3)=2-4\cosh(\frac{t_3}{2})^2\sinh(d_1)^2\sinh(a_2)^2.
\end{equation}
In particular, $\rho$ belongs this time to $\mo^0_{+}$.

\subsubsection{Case $\epsilon_1=\epsilon_2=0^+$ and $\Delta>0$}\label{LeCas4}
$$X_1=Y_1=SR_{\theta_1},\quad  X_2=Y_2=SR_{\theta_2},\quad X_3=Y_3=SR_{\theta_3}.$$
A direct computation gives 
\begin{eqnarray}\label{betaeu0} 
\tr \rho(\beta_1)&=&2\cos(\frac{\theta_1}{2})^2\cosh(\frac{t_2+t_3}{2})+2\sin(\frac{\theta_1}{2})^2\cosh(\frac{t_2-t_3}{2})\notag \\
&=&2\cosh(\frac{t_2}{2})\cosh(\frac{t_3}{2})+2\cos(\theta_1)\sinh(\frac{t_2}{2})\sinh(\frac{t_3}{2}).
\end{eqnarray}
We now look for an expression for $\rho(\delta_3)$:
\begin{eqnarray}\label{deltaeu0}
\tr\rho(\delta_3)&=&2\cosh(a_2)^2-2\cos(\theta_1)^2\sinh(a_2)^2-2\sin(\theta_1)^2\sinh(a_2)^2\cosh(t_3)\notag \\
&=&2-4\sin(\theta_1)^2\sinh(a_2)^2\sinh(\frac{t_3}{2})^2.
\end{eqnarray}
In particular, $\rho$ belongs this time to $\mo^0_{+}$.
\subsubsection{Case $\epsilon_1=0^+,\epsilon_2=0^-$ and $\Delta>0$} 
The $X_i$'s are the same as before but we have now:
$$Y_1=SR_{-\theta_1},\quad  Y_2=SR_{-\theta_2},\quad Y_3=SR_{-\theta_3}.$$
\begin{equation}\label{betaeu0prime} 
\tr \rho(\beta_1)=2\sinh(\frac{t_2}{2})\sinh(\frac{t_3}{2})+2\cos(\theta_1)\cosh(\frac{t_2}{2})\cosh(\frac{t_3}{2}).
\end{equation}
We now look for an expression for $\rho(\delta_3)$:
\begin{equation}\label{deltaeu0prime}
\tr\rho(\delta_3)=2+4\sin(\theta_1)^2\sinh(a_2)^2\cosh(\frac{t_3}{2})^2.
\end{equation}
In particular, $\rho$ belongs this time to $\mo^0_{+}$ and we have listed all cases of Euler class 0
and such that $\Delta\neq 0$.

\subsubsection{Case $\Delta=0$ and $\epsilon_1=\epsilon_2=0$}\label{LeCas5}
In this case, since $\rho$ is not elementary, one of the two pants has to
yield an upper triangular representation, while the other has to yield a
lower triangular one. Up to conjugating $\rho$, we thus have
\[
X_1=S\cdot\left(\begin{array}{cc}1 & u\sinh a_1 \\ 0 & 1\end{array}\right)
\text{ and }
Y_1=S\cdot\left(\begin{array}{cc}1 & 0 \\ v\sinh a_1 & 1\end{array}\right)
\]
with $u,v\in\{-1,+1\}$. A direct computation gives
\begin{equation}\label{DeltaNulEuZero}
\tr \rho(\delta_3)=2+4uv\sinh(a_1)^2e^{-t_3}.
\end{equation}
In particular, $\rho$ belongs to $\mo^0_+$ or $\mo^0_-$ depending on the sign of $uv$. 

\subsubsection{Case $\epsilon_1=0^{\pm},\epsilon_2=-1$ and $\Delta<0$}\label{CasUnDeltaNeg}
In that case, we can choose:
$$Y_1=R_r T_{b_1} R_r, \quad Y_2=R_r T_{b_2} R_r, \quad Y_3=R_r T_{b_3} R_r\quad \text{and}$$
$$X_1=R_lT_{\pm d_1}R_l,\quad X_2=R_rT_{\pm d_2}R_r,\quad X_3=R_lT_{\pm d_3}R_r.$$
Hence, we consider the case when $\epsilon_1=0^+$, the other case being obtained by changing the signs of the $d_i's$. 
This gives finally:
$$\tr\rho(\delta_3)=2\cosh(a_2)^2-2\cosh(b_1)\cosh(d_1)\sinh(a_2)^2-2\cosh(t_3)\sinh(a_2)^2\sinh(b_1)\sinh(d_1).$$
Together with Formula \eqref{longself},
this yields the following equality: 
\begin{equation}\label{comdeltaneg}
\tr\rho(\delta_3)=2\frac{\sinh(a_1)^2-\sinh(a_2)^2}{\sinh(a_3)^2}-2\cosh(t_3)\sinh(a_2)^2\sinh(b_1)\sinh(d_1).
\end{equation}

\subsubsection{Case $\epsilon_1=0^{\pm},\epsilon_2=-1$ and $\Delta>0$}
$$Y_1=R_r T_{b_1} R_r, \quad Y_2=R_r T_{b_2} R_r, \quad Y_3=R_r T_{b_3} R_r\quad \text{and}$$
$$X_1=SR_{\pm\theta_1},\quad X_2=SR_{\pm\theta_2},\quad X_3=SR_{\pm\theta_3}.$$
As before, we consider only the case $\epsilon_1=0^+$, the other case being obtained by changing the signs of the $\theta_i's$. A direct computation gives:
\begin{equation}\label{beta1} 
\tr \rho(\beta_1)=-2\cos(\frac{\theta_1}{2})\cosh(\frac{b_1}{2})\cosh(\frac{t_2+t_3}{2})-2\sin(\frac{\theta_1}{2})\sinh(\frac{b_1}{2})\sinh(\frac{t_2-t_3}{2})
\end{equation}
$$\tr\rho(\delta_3)=2\cosh(a_2)^2-2\cos(\theta_1)\cosh(b_1)\sinh(a_2)^2+2\sin(\theta_1)\sinh(b_1)\sinh(a_2)^2\sinh(t_3).$$
Using Formula \eqref{longtri} this finally gives:
\begin{equation}\label{delta1}
\tr\rho(\delta_3)=2\frac{\sinh(a_1)^2-\sinh(a_2)^2}{\sinh(a_3)^2}+2\sin(\theta_1)\sinh(b_1)\sinh(a_2)^2\sinh(t_3).
\end{equation}

\subsubsection{Case $\Delta=0$, $\epsilon_1=0$, $\epsilon_2=-1$}
If the restriction of $\rho$ to $\pi_1(P_1)$ is diagonal, the equations are
already computed in case \ref{CasUnDeltaNeg}, with $d_1=d_2=d_3=0$.
We write here the case when this restriction is upper diagonal. The
case when it is lower diagonal is deduced by taking the inverse of the
transposition: the formulas obtained are the same, under changing the
signs of all the $a_i's$. We have
\[
X_1=S\left(\begin{array}{cc}1 & u\sinh a_1 \\ 0 & 1\end{array}\right)
\text{ and }
Y_1=R_rT_{b_1}R_r.
\]
A direct computation gives
\begin{equation}\label{DeltaNulEuUn}
\tr \rho(\delta_3)=2-4\sinh(b_1/2)^2\sinh(a_2)^2+2u\sinh(a_1)\sinh(a_2)^2\sinh(b_1)e^{-t_3}.
\end{equation}

\subsection{A second proof of Corollary \ref{CorNonErgod}}
It is noteworthy that the formulas of the preceding section yield a computational proof of
Corollary \ref{CorNonErgod}. 

Let $\mo^{00}$ be the set of classes of representations $[\rho]$ which send no separating simple curve to the identity. The subset $\mo^{00}$ has full measure in $\mo^0$. Let $\Gamma=(\gamma_1,\gamma_2,\gamma_3)$ be a triple of non-separating disjoint curves and define $\mo^{00}_\Gamma$ to be the set of classes of representations $[\rho]$ in $\mo^{00}$ mapping $\gamma_i$ to hyperbolic elements and satisfying $\tr\rho(\delta_3)<2$ where $\delta_3$ is the curve described in the preceding section. 
We show in this section that for any $\Gamma,\Gamma'$ we have $\mo^{00}_{\Gamma}=\mo^{00}_{\Gamma'}$ which implies that $\mo^{00}_\Gamma$ is invariant by the mapping class group - and coincides with $\mo^{00}\cap\mo^0_+$. 

In order to prove this, it is sufficient to prove the equality $\mo^{00}_{\Gamma}=\mo^{00}_{\tau_\zeta\Gamma}$ for $\zeta\in \{\gamma_1,\gamma_2,\gamma_3,\beta_1,\beta_2,\beta_3\}$ because these Dehn twists generate the mapping class group. 
Starting from a representation $[\rho]$ in $\mo^{00}_\Gamma$, the different descriptions given in the preceding sections indicate that $\rho$ falls into the cases \ref{LeCas3}, \ref{LeCas4} and \ref{LeCas5}. As the sign of $\tr \rho(\delta_3)-2$ is independent of $t_1,t_2,t_3$, it follows that the condition defining $\mo^{00}_\Gamma$ is invariant under the Dehn twists on $\gamma_1,\gamma_2$ and $\gamma_3$ and hence depends only on $\Gamma$ as the notation suggests. Moreover straightforward computations shows a symmetry in the formulas for $\tr \rho(\delta_3)$ proving that $\delta_3$ can be replaced by $\delta_1$ or $\delta_2$. 

If we apply the Dehn twist $\tau_{\beta_1}$, we get the new curves $\tau_{\gamma_1}^{-1}\beta_1,\tau_{\gamma_2}^{-1}\beta_1,\gamma_3$ and the explicit formulas of the corresponding sections show that these curves are sent to hyperbolic elements. At the same time, $\beta_3$ and $\gamma_1$ are unchanged hence their commutator $\delta_2$ still satisfies $\tr\rho(\delta_2)<2$ and the statement is proved.

%% file: pantalon.pdf_tex
\begingroup%
  \makeatletter%
  \providecommand\color[2][]{%
    \errmessage{(Inkscape) Color is used for the text in Inkscape, but the package 'color.sty' is not loaded}%
    \renewcommand\color[2][]{}%
  }%
  \providecommand\transparent[1]{%
    \errmessage{(Inkscape) Transparency is used (non-zero) for the text in Inkscape, but the package 'transparent.sty' is not loaded}%
    \renewcommand\transparent[1]{}%
  }%
  \providecommand\rotatebox[2]{#2}%
  \ifx\svgwidth\undefined%
    \setlength{\unitlength}{484.42216797bp}%
    \ifx\svgscale\undefined%
      \relax%
    \else%
      \setlength{\unitlength}{\unitlength * \real{\svgscale}}%
    \fi%
  \else%
    \setlength{\unitlength}{\svgwidth}%
  \fi%
  \global\let\svgwidth\undefined%
  \global\let\svgscale\undefined%
  \makeatother%
  \begin{picture}(1,0.48592904)%
    \put(0,0){\includegraphics[width=\unitlength]{pantalon.pdf}}%
    \put(0.07431534,0.28074684){\color[rgb]{0,0,0}\makebox(0,0)[lb]{\smash{$T_{a_1}$
}}}%
    \put(0.47066386,0.28074687){\color[rgb]{0,0,0}\makebox(0,0)[lb]{\smash{$T_{a_2}$}}}%
    \put(0.84190222,0.28124174){\color[rgb]{0,0,0}\makebox(0,0)[lb]{\smash{$T_{a_3}$}}}%
    \put(0.65232359,0.21468878){\color[rgb]{0,0,0}\makebox(0,0)[lb]{\smash{$X_1$}}}%
    \put(0.45045645,0.05119502){\color[rgb]{0,0,0}\makebox(0,0)[lb]{\smash{$X_2$}}}%
    \put(0.25597508,0.21468874){\color[rgb]{0,0,0}\makebox(0,0)[lb]{\smash{$X_3$}}}%
    \put(0.07062245,0.44754351){\color[rgb]{0,0,0}\makebox(0,0)[lb]{\smash{$T_{a_1}$
}}}%
    \put(0.46697095,0.44754351){\color[rgb]{0,0,0}\makebox(0,0)[lb]{\smash{$T_{a_2}$}}}%
    \put(0.84190222,0.4546442){\color[rgb]{0,0,0}\makebox(0,0)[lb]{\smash{$T_{a_3}$}}}%
  \end{picture}%
\endgroup%

%% file: genus2.pdf_tex
\begingroup%
  \makeatletter%
  \providecommand\color[2][]{%
    \errmessage{(Inkscape) Color is used for the text in Inkscape, but the package 'color.sty' is not loaded}%
    \renewcommand\color[2][]{}%
  }%
  \providecommand\transparent[1]{%
    \errmessage{(Inkscape) Transparency is used (non-zero) for the text in Inkscape, but the package 'transparent.sty' is not loaded}%
    \renewcommand\transparent[1]{}%
  }%
  \providecommand\rotatebox[2]{#2}%
  \ifx\svgwidth\undefined%
    \setlength{\unitlength}{558.14482422bp}%
    \ifx\svgscale\undefined%
      \relax%
    \else%
      \setlength{\unitlength}{\unitlength * \real{\svgscale}}%
    \fi%
  \else%
    \setlength{\unitlength}{\svgwidth}%
  \fi%
  \global\let\svgwidth\undefined%
  \global\let\svgscale\undefined%
  \makeatother%
  \begin{picture}(1,0.63782729)%
    \put(0,0){\includegraphics[width=\unitlength]{genus2.pdf}}%
    \put(0.07796169,0.17519691){\color[rgb]{0,0,0}\makebox(0,0)[lb]{\smash{$T_{a_1}$
}}}%
    \put(0.42195845,0.17519693){\color[rgb]{0,0,0}\makebox(0,0)[lb]{\smash{$T_{a_2}$}}}%
    \put(0.74416172,0.17562644){\color[rgb]{0,0,0}\makebox(0,0)[lb]{\smash{$T_{a_3}$}}}%
    \put(0.57962362,0.12582746){\color[rgb]{0,0,0}\makebox(0,0)[lb]{\smash{$Y_1$}}}%
    \put(0.42107065,0.04401542){\color[rgb]{0,0,0}\makebox(0,0)[lb]{\smash{$Y_2$}}}%
    \put(0.23562687,0.12582739){\color[rgb]{0,0,0}\makebox(0,0)[lb]{\smash{$Y_3$}}}%
    \put(0.06792846,0.41886128){\color[rgb]{0,0,0}\makebox(0,0)[lb]{\smash{$T_{a_1}$
}}}%
    \put(0.40475859,0.42602789){\color[rgb]{0,0,0}\makebox(0,0)[lb]{\smash{$T_{a_2}$}}}%
    \put(0.72725556,0.41886128){\color[rgb]{0,0,0}\makebox(0,0)[lb]{\smash{$T_{a_3}$}}}%
    \put(0.4128533,0.57245439){\color[rgb]{0,0,0}\makebox(0,0)[lb]{\smash{$X_2$}}}%
    \put(0.2476188,0.49028434){\color[rgb]{0,0,0}\makebox(0,0)[lb]{\smash{$X_3$}}}%
    \put(0.59498678,0.49028431){\color[rgb]{0,0,0}\makebox(0,0)[lb]{\smash{$X_1$}}}%
    \put(0.20666679,0.30564905){\color[rgb]{0,0,0}\makebox(0,0)[lb]{\smash{$T_{t_1}$}}}%
    \put(0.56192534,0.30769665){\color[rgb]{0,0,0}\makebox(0,0)[lb]{\smash{$T_{t_2}$}}}%
    \put(0.91001728,0.30974426){\color[rgb]{0,0,0}\makebox(0,0)[lb]{\smash{$T_{t_3}$}}}%
    \put(0.63760276,0.20735209){\color[rgb]{0,0,0}\makebox(0,0)[lb]{\smash{$\beta_1$}}}%
    \put(0.1134733,0.5657003){\color[rgb]{0,0,0}\makebox(0,0)[lb]{\smash{$\beta_2$}}}%
    \put(0.17138813,0.42815251){\color[rgb]{0,0,0}\makebox(0,0)[lb]{\smash{$\beta_3$}}}%
  \end{picture}%
\endgroup%

%% file: ChasseElliptique1.tex
For simplicity we will say, in the rest of this text, that $\rho$
{\em has a non-hyperbolic curve} if there exists a simple closed
curve $\gamma$ such that $|\tr(\rho(\gamma))|\leq 2$.

\subsection{The general strategy}\label{StrategieGenerale}

Let $[\rho]\in\mo^1\cup\mo^{-1}\cup\mo_-^0$.
Let $j\in \mo^{-2}$ be a representation dominating $\rho$ as given by Theorem
\ref{domination}. Let $B_2$ be the Bers constant in genus 2: this means that
one can find three curves $\gamma_1,\gamma_2,\gamma_3$ bounding two pairs of
pants $P_1$ and $P_2$ such that
$|\tr \rho(\gamma_i)|\le |\tr j(\gamma_i)|\le 2\cosh(B_2/2)$ for $i=1,2,3$.
The constant $B_2$ is known explicitly: $B_2\simeq 4.45$, see \cite{Gendulphe}. 

Suppose first that one of these curves is separating, say $\gamma_1$.
Then the restriction of $\rho$ to one of the two one-holed tori obtained
by cutting $\Sigma_2$ along $\gamma_1$ has Euler class $0$
(in $\mo^{\pm 1}$ this is obvious, and in Euler class $0$ this is the
condition of not being in $\mo^0_+$).
Since we have $2<\tr \rho(\gamma_1)\le 2\cosh(B_2/2)\simeq 9.3$,
Goldman's Theorem \ref{dixhuit} implies that there exists a non-separating
curve
which is sent to a non-hyperbolic element. 

Therefore we may suppose that $\gamma_1,\gamma_2,\gamma_3$ are non-separating,
so that they bound two pairs of pants $P_1$ and $P_2$,
and we can suppose for all $i$ that $|\tr\rho(\gamma_i)|>2$, otherwise the conclusion
follows immediately. We set $|\tr\rho(\gamma_i)|=2\cosh(a_i)$.

Now we can use the formulas of Section \ref{SectionParam}.
Our systematic strategy will be as follows: we will deal separately with all
the pertinent cases from Section \ref{SectionParam}, and our aim will be
either to prove that there exists a non-hyperbolic curve, or to prove that
we can find three disjoint curves $\beta_1$, $\beta_2$, $\beta_3$ such
that
\[
\max|\tr(\rho(\beta_i))|\leq \max|\tr(\rho(\gamma_i))|-\mu,
\]
where $\mu$ is some positive constant independent of $\rho$. This suffices
to prove that after finitely many steps, we are able to find a simple closed
curve mapped to a non-hyperbolic element.

\subsection{The case of $\mo^0_{-}$}

Let $\rho$ be a representation in $\mo^0_{-}$. We show in this part that
$\rho$ has a non-hyperbolic curve. By the defining property of $\mo^0_{-}$,
this curve has to be non-separating.

From Subsection \ref{explicit}, we have three cases to consider up to symmetries. 
\subsubsection{The case $\epsilon_1=\epsilon_2=0^+,\Delta<0$}\label{cas1}
We start from Equation \eqref{ep0deltaneg}
\begin{eqnarray*}
\tr\rho(\delta_3)&=&2+4\sinh(\frac{t_3}{2})^2\sinh(d_1)^2\sinh(a_2)^2\\
&=&2+4\sinh(\frac{t_3}{2})^2\frac{\sinh(d_3)^2\sinh(a_2)^2\sinh(a_1)^2}{\sinh(a_3)^2}
\end{eqnarray*}
where we used Heron formula \eqref{heronhex}. Using Formula \eqref{longself},
we deduce that $\cosh(d_3)\le \frac{2\sinh(\frac{a_3}{2})^2}{\sinh(a_1)\sinh(a_2)}$.
Using the inequality $\sinh(x)\le \cosh(x)$ we get 
\begin{equation}\label{com3}
2<\tr\rho(\delta_3)\le 2+16\frac{\sinh(\frac{t_3}{2})^2\sinh(\frac{a_3}{2})^4}{\sinh(a_3)^2}.
\end{equation}
Up to applying a power of the Dehn twist $\tau_3$ to $\delta_3$, we can suppose that $t_3$ belongs to $[-a_3,a_3]$. Hence, the maximal value of $\tr(\delta_3)$ is $2+4\frac{\sinh(\frac{a_3}{2})^4}{\cosh(\frac{a_3}{2})^2}$ which is an increasing function of $a_3$ not exceeding $6.8$ for $a_3\le \frac{B_2}{2}$. 

The curve $\delta_3$ is separating and using Goldman's Theorem \ref{dixhuit}, we can conclude that $\rho$ has a non-hyperbolic curve.  

\subsubsection{The case $\epsilon_1=0^+,\epsilon_2=0^-,\Delta>0$}\label{cas2}
Recall Equation \eqref{deltaeu0prime}:
$$\tr\rho(\delta_3)=2+4\sin(\theta_1)^2\sinh(a_2)^2\cosh(\frac{t_3}{2})^2.$$
Goldman's Theorem \ref{dixhuit} provides a non-hyperbolic curve if
$|\sin(\theta_1)|\sinh(a_2)\cosh(\frac{a_3}{2})\le 2$. 
Hence, it is sufficient to have $|\sin(\theta_i)|\leq\frac{2}{\sinh(\frac{B_2}{2})\cosh(\frac{B_2}{4})}$:
this deals with the cases when the triangle of lengths $(a_1,a_2,a_3)$ is
close to be flat.

In general, recall Formula \eqref{betaeu0prime} and its cyclical companions: 
$$\tr \rho(\beta_1)=2\sinh(\frac{t_2}{2})\sinh(\frac{t_3}{2})+2\cos(\theta_1)\cosh(\frac{t_2}{2})\cosh(\frac{t_3}{2}).$$
This gives the inequality
$|\tr \rho(\beta_1)|\leq 2\cosh(\frac{|t_2|+|t_3|}{2})\le 2\cosh(a_3)$.
When the triangle of lengths $(a_1,a_2,a_3)$ is not close to be flat,
$|\cos(\theta_1)|$ is far from $1$, hence the inequality
$|\tr\rho(\beta_1)|\leq 2\cosh(a_3)$ can actually be improved by a positive
constant $\mu$. Similarly, we have $|\tr(\rho(\beta_i)|\leq 2\cosh(a_3)-\mu$ for
all $i$.

\subsubsection{The case $\epsilon_1=\epsilon_2=0,\Delta=0$}\label{cas3}
From Equation \eqref{DeltaNulEuZero}, by taking $t_3$ large enough
(by applying Dehn twists $\tau_3$ to $\delta_3$) we get
$|\tr\rho(\delta_3)|\leq 18$, hence Goldman's Theorem \ref{dixhuit}
yields a non-hyperbolic curve.
%
%

\subsubsection{The case $\epsilon_1=1,\epsilon_2=-1$}\label{cas4}
This case is, by far, the most involved of the four cases corresponding
to $\mo^0_-$. Depending on the triple $(a_1,a_2,a_3)$, we need to adopt
different strategies, described in the following lemmas.
For the sake of readability, we will not incorporate the quantity $\mu$
in these lemmas, but we will come back to this in the conclusion
of this case.

Using Heron formula \eqref{heronhex} and Formula \eqref{epsunmoinsun}, we get: 
\begin{equation}\label{com3bis}
\tr\rho(\delta_3)=2+4\frac{D'^2\sinh(\frac{t_3}{2})^2}{\sinh(a_3)^2}.
\end{equation}
Theorem \ref{dixhuit} can conclude that case if $D'|\sinh(\frac{t_3}{2})|\le 2 \sinh(a_3).$
Moreover, as $t_3$ can be chosen in $[-a_3,a_3]$, it is sufficient to have $D'\le 4\cosh(\frac{a_3}{2})$.

\begin{lemma}[Bandwidth]\label{bandwidth}
If the following estimate holds:
$$\tanh(\frac{a_1}{2})\le \sinh(\frac{|t_3|}{2})\sinh(b_2)\le \coth(\frac{a_1}{2})$$
then up to a translation of $t_1$ by a multiple of $2a_1$, we have $|\tr \rho(\beta_2)|\le 2$. In particular, $\rho$ has a non-hyperbolic curve. 
\end{lemma}
\begin{proof}
Recall the expression for $\tr\rho(\beta_2)$:
$$\tr \rho(\beta_2)=2\cosh(\frac{t_1}{2})\cosh(\frac{t_3}{2})+2\cosh(b_2)\sinh(\frac{t_1}{2})\sinh(\frac{t_3}{2})=A e^{\frac{t_1}{2}}+Be^{-\frac{t_1}{2}}$$
writing $A=\cosh(\frac{t_3}{2})+\cosh(b_2)\sinh(\frac{t_3}{2})$ and $B=\cosh(\frac{t_3}{2})-\cosh(b_2)\sinh(\frac{t_3}{2})$.
We thus have $1-AB=\sinh(\frac{t_3}{2})^2\sinh(b_2)^2$.

Suppose that $AB<0$: then the map
$\Phi:t_1\mapsto A e^{\frac{t_1}{2}}+Be^{-\frac{t_1}{2}}$ is a diffeomorphism
from $\R$ to $\R$ hence there exists $t_1^\pm$ so that $\Phi(t_1^{\pm})=\pm 2$
and the statement of the lemma is equivalent to the following one:
$$|t_1^+ - t_1^-|\ge 2a_1.$$
We compute explicitly $t_1^{\pm}=2\ln\frac{\pm 1+\sqrt{1-AB}}{A}$ and
$|t_1^+ - t_1^-|=2\ln\frac{\sqrt{1-AB}+1}{\sqrt{1-AB}-1}=
4\operatorname{arccoth}(\sinh(\frac{|t_3|}{2})\sinh(b_2))$.
By the second part of the inequality, the result follows.
Suppose that $AB>0$.
Up to considering $-t_1$ instead of $t_1$, we may suppose that $t_3\geq 0$,
and that $A>0$, $B>0$.
In that case, the equation $\Phi(t_1)=2$ has
two solutions $t_1^-<t_1^+$ given by $t_1^\pm=2\ln\frac{1\pm \sqrt{1-AB}}{A}$.
We have now
$|t_1^+ - t_1^-|=4\operatorname{arctanh}(\sin(\frac{|t_3|}{2})\sinh(b_2))$.
By the first part of the inequality, the result follows.
If $A$ or $B$ vanishes, the result still holds true as $\Phi$ goes to $0$
when $t_1$ goes to $+\infty$ or $-\infty$.  
\end{proof}

\begin{lemma}[Intervals strategy]\label{intervalles}
Suppose that $a_1\le a_2\le a_3$ and $\rho$ does not have any non-hyperbolic curve. Then either
$\cosh(a_1)>3$ or $\cosh(a_2)>\cosh(a_1)+2$. 
\end{lemma}
\begin{proof}
We first use Dehn twists along $\gamma_3$ in order to have $t_3\in[-a_3,a_3]$.
Using Heron formula, the criterium of Lemma \ref{bandwidth} can be reformulated
in the following way:
$$\cosh(a_1)-1\le \frac{D'\sinh(\frac{|t_3|}{2})}{\sinh(a_3)}\le \cosh(a_1)+1.$$
Hence, if the quantity $u_3=\frac{D'\sinh(\frac{|t_3|}{2})}{\sinh(a_3)}$
belongs to the interval $I_1=[\cosh(a_1)-1,\cosh(a_1)+1]$ $\rho$ has a
non-hyperbolic curve. But the same is true if $u_3$ belongs to the interval
$I_2=[\cosh(a_2)-1,\cosh(a_2)+1]$ and again, by Formula \eqref{com3bis},
the same is true if $u_3\in I_0=[0,2]$. 
Moreover, the maximal value of $u_3$ is reached for $t_3=a_3$ and $a_1=a_2$.
We have precisely
$u_3\le \frac{\cosh(\frac{a_2}{2})}{\cosh(\frac{a_3}{2})}\sqrt{\cosh(a_2+\frac{a_3}{2})\cosh(a_2-\frac{a_3}{2})}$.
A straightforward computation ensures that this is lower or equal to $\cosh(a_2)+1$,
as long as $a_2\leq a_3$.
Hence, we have $u_3\le \cosh(a_2)+1$ and either $I_0\cap I_1=\emptyset$ or
$I_1\cap I_2=\emptyset$.  These two cases imply respectively the two
inequalities of the lemma. 
\end{proof}

\begin{lemma}[Equilateral strategy]\label{equ}
Suppose that one has $a_1\le a_2\le a_3$
and suppose that the following holds:
\begin{itemize}
\item[(0)] $\cosh(\frac{b_3}{2})^2\cosh(\frac{a_3}{2})\le\cosh(a_3)+\sinh(\frac{b_3}{2})^2$.
\end{itemize}
Define $\lambda\ge 0$ by the following equation:
$$ \cosh(\frac{b_3}{2})^2\cosh(\frac{a_3+\lambda}{2})=\cosh(a_3)+\sinh(\frac{b_3}{2})^2,$$
and suppose also that $|t_i|\le a_i$ and that the following inequalities hold:
\begin{enumerate}
\item $\cosh(b_3)\sinh(\frac{3a_3-\lambda}{4})^2-\cosh(\frac{3a_3-\lambda}{4})^2\le 1$
\item $\cosh(\frac{\lambda}{2})\cosh(\frac{a_3}{2})-\cosh(b_1)\sinh(\frac{\lambda}{2})\sinh(\frac{a_3}{2})\le 1.$
\end{enumerate}
Then either $\rho$ has a non-hyperbolic curve or for all $i\in\{1,2,3\}$ one has $|\tr \rho(\beta_i)|\le 2\cosh(a_3)$. 
\end{lemma}
\begin{proof}
Note that we have $\lambda<a_3$, obviously from its definition.
For each $i\in\{1,2,3\}$, we introduce the function 
$$ F_i(x,y)=2\cosh(\frac{x+y}{2})\cosh(\frac{b_i}{2})^2-2\cosh(\frac{x-y}{2})\sinh(\frac{b_i}{2})^2.$$
Consider the polygon 
$$P=\{(x,y)\in [-a_3,a_3]^2\text{ s.t. }|x+y|\le a_3+\lambda\}.$$
The monotony of the function $F_i$ implies that for all $(x,y)\in P$ one has:
$$F_i(-a_3,a_3)\le F_i(x,y)\le F_i(\frac{a_3+\lambda}{2},\frac{a_3+\lambda}{2}).$$
The expression
$F_i(\frac{a_3+\lambda}{2},\frac{a_3+\lambda}{2})=2\cosh(\frac{a_3+\lambda}{2})\cosh(\frac{b_i}{2})^2
-2\sinh(\frac{b_i}{2})^2$,
viewed as a function of the two variables $a_3+\lambda$ and $b_i$, is
increasing in $b_i$, so this is not greater than $2\cosh(a_3)$.

Assume for the moment that $F_i(-a_3,a_3)\ge -2\cosh(a_3)$: if for each $i$
one has $(t_i,t_j)\in P$ then for all $i$ one has
$|\tr \rho(\beta_i)|\le 2\cosh(a_3)$ and we are done. Otherwise, there
exists $i$ such that $t_j\in [-a_j,a_j], t_k\in [-a_k,a_k]$ and
$|t_j+t_k|\ge a_3+\lambda$. If $t_j>0$ we replace it with $t_j-2a_j$, else
we replace $t_k$ with $t_k+2a_k$. In any case, the new pair $(t'_j,t'_k)$
belongs to a square $Q$ which has the following end-points:
$(-a_3,\lambda),(-\lambda,a_3),(-a_3,2a_3-\lambda),(-2a_3+\lambda,a_3)$.
If we can show that $|F_i(t'_j,t'_k)|\le 2$ then we show that $\rho$ has
some non-hyperbolic curves, the one obtained by twisting $\beta_i$ along
$\gamma_j$ or $\gamma_k$. The monotony of $F_i$ implies that for all
$(x,y)\in Q$ one has: 
$$F_i(\frac{-3a_3+\lambda}{2},\frac{3a_3-\lambda}{2})\le F_i(x,y)\le F_i(-\lambda,a_3).$$ 
The assumptions (1) and (2) imply respectively the estimations
$F_i(-\lambda,a_3)\le2$ and $F_i(\frac{-3a_3+\lambda}{2},\frac{3a_3-\lambda}{2})\ge -2$.
The point $(-a_3,a_3)$ being the center of $Q$, we have
$F_i(-a_3,a_3)\ge -2$ and the lemma is proved. 
\end{proof}

\begin{lemma}[Isosceles strategy]\label{isozero}
Suppose that one has $a_1\le a_2\le a_3, |t_i|\le a_i$
and suppose the following holds:
\begin{itemize}
\item[(0)] $\cosh(\frac{b_1}{2})^2\cosh(\frac{a_3}{2})\le\cosh(a_3)+\sinh(\frac{b_1}{2})^2$.
\end{itemize}
Define $\lambda\ge0$ by the equation
$$\cosh(\frac{b_1}{2})^2\cosh(\frac{a_3+\lambda}{2})=\cosh(a_3)+\sinh(\frac{b_1}{2})^2,$$
and suppose that the following inequalities hold:
\begin{enumerate}
\item $\cosh(\frac{\lambda}{2})\cosh(\frac{a_3}{2})-\cosh(b_1)\sinh(\frac{\lambda}{2})\sinh(\frac{a_3}{2})\le 1$
\item $\sinh(\frac{b_1}{2})^2\cosh(\frac{3a_3-\lambda}{2})-\cosh(\frac{b_1}{2})^2\le 1$
\item $\cosh(\frac{a_1}{2})\cosh(\frac{a_3}{2})+\cosh(b_3)\sinh(\frac{a_1}{2})\sinh(\frac{a_3}{2})\le \cosh(a_3)$.
\end{enumerate}
Then either $\rho$ has a non-hyperbolic curve, or up to a translation
$t_2\mapsto t_2-2a_2$ or $t_3\mapsto t_3+2a_3$ one has
$|\tr \rho(\beta_i)|<2\cosh(a_3)$ for $i=1,2,3$. 
\end{lemma}
\begin{proof}
We follow the strategy of Lemma \ref{equ} but only for $i=1$.
Define $P$ and $Q$ accordingly.
The definition of $\lambda$ ensures that $F_1(x,y)\le 2\cosh(a_3)$ for all $(x,y)\in P$. 
On the other hand, we have for any $(x,y)\in Q$:
$$F_1(\frac{-3a_3+\lambda}{2},\frac{3a_3-\lambda}{2})\le F_1(x,y)\le F_1(-\lambda,a_3).$$ 
The assumptions (1) and $(2)$ are respectively equivalent to $F_1(-\lambda,a_1)\le 2$ and $F_1(\frac{-3a_3+\lambda}{2},\frac{3a_3-\lambda}{2})\ge -2$. 
Hence we can suppose that $(t_2,t_3)$ are in $P$ because otherwise, up to a translation $\rho(\beta_1)$ becomes non-hyperbolic. 

The other traces satisfy the following estimation for $i=2,3$:
$$|\tr\rho(\beta_i)|\le 2\cosh(\frac{a_1}{2})\cosh(\frac{a_3}{2})+\cosh(b_3)\sinh(\frac{a_1}{2})\sinh(\frac{a_3}{2}).$$
This is less than $2\cosh(a_3)$ thanks to the third assumption of the lemma. 
\end{proof}

{\bf Conclusion}.

If $\cosh(a_3)\leq 3$ and $\cosh(a_2)\leq \cosh(a_1)+2$, then
Lemma \ref{intervalles} provides a non hyperbolic curve. This covers the region $X_1$ in Figure \ref{graph1}. Hence we are left
with the two following cases corresponding to the regions $X_2$ and $X_3$ in the Figure. 

Suppose that $\cosh(a_1)>3$. Thus,
$3<\cosh(a_1)\leq \cosh(a_2)\leq \cosh(a_3)\leq \cosh(\frac{B_2}{2})$.
In this case, we use the
Equilateral strategy: the inequalities we need to check are the
assumptions of Lemma \ref{equ}.
Consider the condition $(0)$ of that lemma as a condition on independent
variables $a_3$ and $b_3$. It suffices to be checked when $b_3$ takes
its biggest possible value, that is,
$\mathrm{arcch}(\frac{\cosh(a_3)+9}{8})$.
Straightforward computation proves that, in order to check condition $(0)$,
it suffices to check that the real number $x=\cosh(\frac{a_3}{2})$
satisfies the inequality $x^3+8x+8\leq 17x^2$. And $\cosh(a_3)=2x^2-1$
is in the interval $(3,4.68)$, so $x\in(\sqrt{2},\sqrt{\frac{5.68}{2}})=(x_0,x_1)$.
We actually have $x_1^3+8x_1+8<17x_0^2$, so the condition $(0)$ of Lemma \ref{equ}
is automatic.

Checking condition $(1)$ of Lemma \ref{equ} amounts to check the inequality
$(\cosh(b_3)-1)\sinh(\frac{3a_3-\lambda}{4})^2\leq 2$. As before, it suffices
to check this when replacing $(\cosh(b_3)-1)$ by $\frac{\cosh(a_3)+1}{8}$,
which is lower than $\frac{\cosh(B_2/2)+1}{8}$. So it suffices to
check that $\sinh(\frac{3a_3-\lambda}{4})\leq\frac{4}{\sqrt{\cosh(B_2/2)+1}}$.
Now, from the equation defining $\lambda$, (where $a_3$ and $b_3$ are
seen as independent variables) we see that $\lambda$ is decreasing in $b_3$.
So $\lambda\geq2\mathrm{arccosh}(\frac{17\cosh(a_3)+1}{\cosh(a_3)+17})-a_3$.
Putting all this together reduces the checking of condition $(1)$ to the
positivity of an explicit analytic function in one real variable $a_3$:
with the help of SAGE, we observe that this condition holds.

We are left, in the case $\cosh(a_1)>3$, with condition $(2)$ of Lemma \ref{equ}.
Because of the monotony of the function $F_1$, as noted in the proof of
Lemma \ref{equ}, the quantity
\[ 1 + \cosh(b_1)\sinh(\frac{\lambda}{2})\sinh(\frac{a_3}{2})-
\cosh(\frac{\lambda}{2})\cosh(\frac{a_3}{2}) \]
is increasing in $\lambda$ (considering $b_1$ and $a_3$ as independent variables
in this expression).
Under the conditions at hand, $\cosh(b_1)$ cannot be smaller than
$\frac{3+\cosh(a_3)^2}{\sinh(a_3)^2}$, whereas
$\cosh(b_3)$ cannot be greater than $\frac{\cosh(a_3)+9}{8}$, hence
$\lambda$ cannot be smaller than
$\lambda\geq2\mathrm{arccosh}(\frac{17\cosh(a_3)+1}{\cosh(a_3)+17})-a_3$.
Once again, putting all this together we are left with the positivity
of one explicit real analytic function of one variable $a_3$ on an
explicit segment. We check this with the help of SAGE.

All the estimations are checked by proving that some continuous function
is positive on the segment $[\mathrm{arcch}(3),\frac{B_2}{2}]$.
A fortiori these estimates can all be improved by some positive
constant, and all the equalities in the lemmas (including the definition
of $\lambda$) can be improved by a positive constant. It follows that
whe indeed obtain either a non-hyperbolic curve, or three new curves
$\beta_i$ satisfying $\max|\tr(\rho(\beta_i))|\leq 2\cosh(a_3)-\mu$
for some positive $\mu$.

\medskip

Finally, suppose that $\cosh(a_2)>\cosh(a_1)+2$. In this case we use the
Isosceles strategy. The same discussion as in the preceding case shows
that condition $(0)$ of Lemma \ref{isozero} is automatic.
Conditions $(1)$, $(2)$ and $(3)$ are checked, as in the preceding
case, with the help of SAGE and by estimating the extremal possible
values of $\lambda$, $b_1$, $b_3$ and $a_1$ subject to the relations
$\frac{B_2}{2}\geq a_3\geq a_2\geq\mathrm{arcch(3)}$ and
$\cosh(a_2)>\cosh(a_1)+2$.

\begin{figure}
\centering
\def\svgwidth{9cm}
 \executeiffilenewer{graph1.svg}{graph1.pdf}%
 {inkscape -z -D --file=graph1.svg %
 --export-pdf=graph1.pdf --export-latex}%
 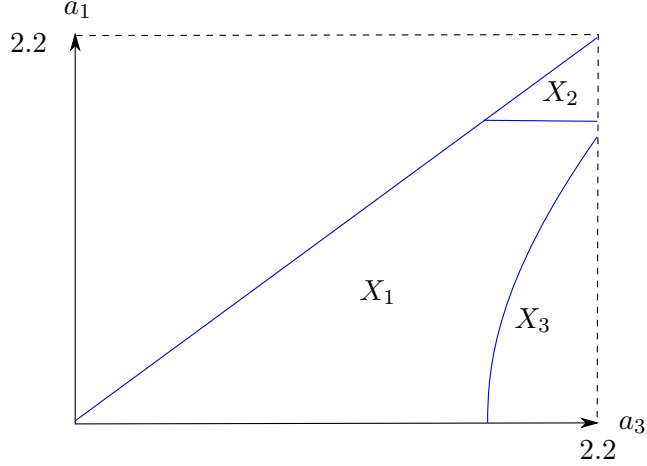%

\caption{Three strategies in Euler class 0}
\label{graph1}
\end{figure}

%% file: graph1.pdf_tex
\begingroup%
  \makeatletter%
  \providecommand\color[2][]{%
    \errmessage{(Inkscape) Color is used for the text in Inkscape, but the package 'color.sty' is not loaded}%
    \renewcommand\color[2][]{}%
  }%
  \providecommand\transparent[1]{%
    \errmessage{(Inkscape) Transparency is used (non-zero) for the text in Inkscape, but the package 'transparent.sty' is not loaded}%
    \renewcommand\transparent[1]{}%
  }%
  \providecommand\rotatebox[2]{#2}%
  \ifx\svgwidth\undefined%
    \setlength{\unitlength}{748.56474609bp}%
    \ifx\svgscale\undefined%
      \relax%
    \else%
      \setlength{\unitlength}{\unitlength * \real{\svgscale}}%
    \fi%
  \else%
    \setlength{\unitlength}{\svgwidth}%
  \fi%
  \global\let\svgwidth\undefined%
  \global\let\svgscale\undefined%
  \makeatother%
  \begin{picture}(1,0.70746995)%
    \put(0,0){\includegraphics[width=\unitlength]{graph1.pdf}}%
    \put(0.5124867,0.24139946){\color[rgb]{0,0,0}\makebox(0,0)[lb]{\smash{$X_1$}}}%
    \put(0.78271812,0.53605861){\color[rgb]{0,0,0}\makebox(0,0)[lb]{\smash{$X_2$}}}%
    \put(0.7414964,0.20017773){\color[rgb]{0,0,0}\makebox(0,0)[lb]{\smash{$X_3$}}}%
    \put(-0.00354846,0.60781495){\color[rgb]{0,0,0}\makebox(0,0)[lb]{\smash{$2.2$}}}%
    \put(0.83768046,0.00628289){\color[rgb]{0,0,0}\makebox(0,0)[lb]{\smash{$2.2$}}}%
    \put(0.07584157,0.66430401){\color[rgb]{0,0,0}\makebox(0,0)[lb]{\smash{$a_1$}}}%
    \put(0.89569624,0.05055804){\color[rgb]{0,0,0}\makebox(0,0)[lb]{\smash{$a_3$}}}%
  \end{picture}%
\endgroup%

%% file: ChasseElliptique2.tex
\subsection{The case of Euler class $\pm 1$}
Following the discussion of Section \ref{StrategieGenerale}, we start with
a triple of disjoint non-separating curves $(\gamma_1,\gamma_2,\gamma_3)$, each
mapped to a hyperbolic element of $\psl$, and cutting
$\Sigma$ in two pants $P_1$ $P_2$ with, say, $\epsilon_1=0,\epsilon_2=-1$.

\subsubsection{The case $\Delta<0$}
From Equation \eqref{comdeltaneg} and the equality
$\frac{\sinh(d_1)}{\sinh(a_1)}=\frac{\sinh(d_3)}{\sinh(a_3)}$ we get:
$$\tr\rho(\delta_3)=2\frac{\sinh(a_1)^2-\sinh(a_2)^2}{\sinh(a_3)^2}-2\cosh(t_3)\frac{\sinh(a_2)^2\sinh(a_1)^2}{\sinh(a_3)^2}\sinh(b_3)\sinh(d_3).$$
Using the majoration $\sinh(x)\le \cosh(x)$ and the identity 
$$\cosh(b_3)\cosh(d_3)=\frac{\cosh(a_3)^2-\cosh(a_1)^2\cosh(a_2)^2}{\sinh(a_1)^2\sinh(a_2)^2}$$
we obtain:
$$\sinh(a_3)^2|\tr(\delta_3)|\le2|\sinh(a_1)^2-\sinh(a_2)^2|+2\cosh(t_3)(\cosh(a_3)^2-\cosh(a_1)^2\cosh(a_2)^2).$$
From $\cosh(a_1)^2\cosh(a_2)^2\ge 1$ we get finally
$$|\tr\rho(\delta_3)|\le 2+2\cosh(t_3).$$
Up to applying a power of the Dehn twist $\tau_3$ to $\delta_3$, we can
suppose that $t_3$ belongs to $[-a_3,a_3]$. Hence, the maximal value of
$|\tr\rho(\delta_3)|$ is $2+2\cosh(a_3)\le 2+2\cosh(B_2/2)\le 11.35<18$. 

The curve $\delta_3$ separates $\Sigma$ into two pants, one of them having Euler class $0$. 
Using again Theorem \ref{dixhuit}, we conclude that $\rho$ has a non-hyperbolic curve.

\subsubsection{The case $\Delta=0$}
By applying Dehn twists $\tau_3$ to $\delta_3$, the trace $\tr(\rho(\delta_3))$
can be taken arbitrarily close to $2-4\sinh(b_1/2)^2\sinh(a_2)^2$. So it
suffices to prove that $\sinh(b_1/2)^2\sinh(a_2)^2<4$, ie
$(\cosh(b_1)-1)\sinh(a_2)^2<8$.
We compute:
\begin{align*}
(\cosh(b_1)-1)\sinh(a_2)^2 & =
\frac{\cosh(a_1)+\cosh(a_2)\cosh(a_3)-\sinh(a_2)\sinh(a_3)}{\sinh(a_2)\sinh(a_3)}\sinh(a_2)^2 \\
& = \frac{2\cosh(a_1)\sinh(a_2)}{\sinh(a_3)},
\end{align*}
since $a_3=a_1+a_2$. This quantity is actually lower than 2,
because $\sinh(a_3)=\cosh(a_1)\sinh(a_2)+\cosh(a_2)\sinh(a_1)$.
Goldman's Theorem \ref{dixhuit} concludes the existence of a closed curve
mapped to a non-hyperbolic element of $\psl$.

\subsubsection{The case $\Delta>0$}

We reinterpret Formula \eqref{delta1} by using Heron formula simultaneously
for the triangle and the right-angled hexagon with lengths $a_1,a_2,a_3$. 

Precisely, we obtain: 
\begin{eqnarray*}
DD'&=&\sin(\theta_1)\sinh(b_1)\sinh(a_2)^2\sinh(a_3)^2\\
&=&\sqrt{\sinh(2s)\sinh(2s-2a_1)\sinh(2s-2a_2)\sinh(2s-2a_3)}=\frac{D_2}{2}
\end{eqnarray*}
where $D_2$ is the Heron invariant associated to a triangle whose lengths are $2a_1,2a_2,2a_3$. 
In this way the formula \eqref{delta1} becomes:
\begin{equation}\label{delta2}
\tr\rho(\delta_3)=2\frac{\sinh(a_1)^2-\sinh(a_2)^2}{\sinh(a_3)^2}+\frac{D_2\sinh(t_3)}{\sinh(a_3)^2}.
\end{equation}
Let now $2\alpha_1,2\alpha_2,2\alpha_3$ be the angles of the triangle whose lengths are $2a_1,2a_2,2a_3$.
Formula \eqref{longtri} gives:
\begin{eqnarray*}
\cos(\alpha_1)^2&=&\frac{\cosh(2a_2+2a_3)-\cosh(2a_1)}{2\sinh(2a_2)\sinh(2a_3)}=\frac{\cosh(a_2+a_3)^2-\cosh(a_1)^2}{\sinh(2a_2)\sinh(2a_3)}\\
&=& (1+\cosh(\theta_1))(1+\cosh(b_1))\tanh(a_2)\tanh(a_3)/4\\
&=&\cos(\frac{\theta_1}{2})^2\cosh(\frac{b_1}{2})^2\tanh(a_2)\tanh(a_3)
\end{eqnarray*}
and similarly
$\sin(\alpha_1)^2=\frac{\cosh(2a_1)-\cosh(2a_2-2a_3)}{2\sinh(2a_2)\sinh(2a_3)}=\sin(\frac{\theta_1}{2})^2\sinh(\frac{b_1}{2})^2\tanh(a_2)\tanh(a_3)$.
Using these identities, Formula \eqref{beta1} becomes
\begin{equation}\label{beta2} 
\tr \rho(\beta_1)=\frac{-2}{\sqrt{\tanh(a_2)\tanh(a_3)}}\left(\cos(\alpha_1)\cosh(\frac{t_2+t_3}{2})+\sin(\alpha_1)\sinh(\frac{t_2-t_3}{2})\right).
\end{equation}

In these series of lemmas $a_1,a_2,a_3$ are three positive numbers satisfying the triangle inequality, $t_1,t_2,t_3$ are three real numbers and $\rho=\rho^{0,\pm 1}_{a_1,a_2,a_3}(t_1,t_2,t_3)$.
\begin{lemma}\label{com}
Suppose that $a_1\le a_2\le a_3$ and $|t_3|\le a_3$. Then we have
$|\tr \rho(\delta_3)|\le 2\Phi(a_1,a_3)$ where 
$$\Phi(a_1,a_3)=\frac{\sinh(a_3)^2-\sinh(a_1)^2}{\sinh(a_3)^2}+\sqrt{\sinh(2a_3)^2-\sinh(a_1)^2}\frac{\sinh(a_1)}{\sinh(a_3)}.$$
Hence if $\Phi(a_1,a_3)\le 9$ then $\rho$ has a non-hyperbolic curve.
This occurs for instance as soon as $a_3<1.459$.
\end{lemma}
\begin{proof}
From the equation $D_2=\sinh(2\alpha_2)\sinh(2a_1)\sinh(2a_3)$ we see
that the maximal value for $D_2$ is obtained for $\alpha_2$ maximal, that
is when $a_2=a_3$. The Heron formula then gives
$D_2=2\sqrt{\sinh(2a_3+a_1)\sinh(2a_3-a_1)\sinh(a_1)^2}$ and
the estimation follows. 
\end{proof} 

\begin{lemma}[Boum strategy]\label{majo}
The following inequality and its cyclic companions hold:
$$|\tr \rho(\beta_3)|\le 2\sqrt{\frac{\cosh(t_1)\cosh(t_2)}{\tanh(a_1)\tanh(a_2)}}.$$
As an application, if for all $i\in\{1,2,3\}$ $t_i\in[-a_i,a_i]$ and the inequalities $\frac{\cosh(a_1)^2}{\sinh(a_1)}\le\sinh(a_3)$ and $\frac{\cosh(a_2)^2}{\sinh(a_2)}\le\sinh(a_3)$ hold then for all $i\in\{1,2,3\}$ the following holds:
$$|\tr \rho(\beta_i)|\le2\cosh(a_3).$$
\end{lemma}
\begin{proof}
From Cauchy-Schwarz inequality in Formula \eqref{beta2} we get 
$|\tr\rho(\beta_1)|^2\le\frac{4}{\tanh(a_2)\tanh(a_3)}(\cosh(\frac{t_1+t_2}{2})^2+\sinh(\frac{t_1-t_2}{2})^2)=\frac{4\cosh(t_1)\cosh(t_2)}{\tanh(a_1)\tanh(a_2)}$.
\end{proof}

\begin{lemma}[Equilateral strategy]\label{equi}
Suppose that one has $a_1\le a_2\le a_3$ and that the inequality
\begin{itemize}
\item[(0)]$\cosh(a_3)\tanh(a_1)^2\ge 1$
\end{itemize}
holds. Define $\lambda\ge 0$ by the formula
$\cosh(\lambda)=\cosh(a_3)\tanh(a_1)^2$,
and suppose also that $|t_i|\le a_i$ and that the following inequalities hold:
\begin{enumerate}
\item $\sinh(2a_1)+\sinh(a_3)^2\le 2\sinh(a_1)^2\cosh(a_3)$
\item $-\cos(\alpha_M)+\sin(\alpha_M)\sinh(\frac{3a_3-\lambda}{2})\le \tanh(a_1)$
\item $\cos(\alpha_m)\cosh(\frac{a_3-\lambda}{2})-\sin(\alpha_m)\sinh(\frac{a_3+\lambda}{2})\le \tanh(a_1)$,
\end{enumerate}
where $\alpha_m$ and $\alpha_M$ are defined by the equations 
$$\sin(\alpha_m)= \frac{\sinh(a_1)}{\sinh(2a_3)}\text{ and }\sin(\alpha_M)=\frac{\sinh(a_3)}{\sinh(2a_1)}.$$
Then  either $\rho$ has a non-hyperbolic curve or for all $i\in\{1,2,3\}$ one has $|\tr\rho(\beta_i)|\le2\cosh(a_3)$. 
\end{lemma}
\begin{proof}
In cyclic notation, introduce the function 
$$F_i(x,y)=\frac{2}{\sqrt{\tanh(a_j)\tanh(a_k)}}(\cos(\alpha_i)\cosh(\frac{x+y}{2})+\sin(\alpha_i)\sinh(\frac{x-y}{2})).$$

Define $\lambda\ge 0$ by the equation $\cosh(\lambda)=\cosh(a_3)\tanh(a_1)^2$. This makes sense by condition (0) and it satisfies $\lambda\le a_3$. 
Consider the set 
$$P=\{(x,y)\in [-a_3,a_3]\times[-a_3,a_3]\text{ s.t. }|x+y|\le a_3+\lambda\}.$$
By the symmetries and monotony of the function $F_i$, it reaches its maximum either on the right or on the bottom side of $P$. 
Moreover, the function $F_i$ is convex on these sides because
$\partial_x^2F_i=\partial_y^2F_i=F_i/4$ and $F_i$ is positive.
Hence, the maximum of $F_i$ on $P$ is reached either at the point
$(a_3,\lambda)$ or at the point $(a_3,-a_3)$. 

One has the formula:
$$F_i(a_3,-a_3)=\frac{2}{\sqrt{\tanh(a_j)\tanh(a_k)}}(\cos(\alpha_i)+\sin(\alpha_i)\sinh(a_3)).$$
When $a_2$ runs in $[a_1,a_3]$, the maximal value of $\alpha_i$ denoted by $\alpha_M$ is obtained for a
triangle with lengths $2a_1,2a_1,2a_3$. It satisfies
$\sin(\alpha_M)=\frac{\sinh(a_3)}{\sinh(2a_1)}$. Hence one has the inequality
$F_i(a_3,-a_3)\leq 2\cosh(a_3)$ if
$1+\frac{\sinh(a_3)^2}{\sinh(2a_1)}\le \cosh(a_3)\tanh(a_1)$. This inequality
is provided by condition (1).

On the other hand, the estimation of Lemma \ref{majo} gives 
$$F_i(a_3,\lambda)\le 2\sqrt{\frac{\cosh(\lambda)\cosh(a_3)}{\tanh(a_j)\tanh(a_k)}}\le 2\sqrt{\cosh(\lambda)\cosh(a_3)}/\tanh(a_1).$$
This is less than $2\cosh(a_3)$ by definition of $\lambda$.
The minimal value of $F_i$ is reached at $(-a_3,a_3)$. The explicit
formula shows that it is in absolute value less than $F_i(a_3,-a_3)$.

We proceed as in Lemma \ref{equ}. If $(t_i,t_j)$ is not in $P$, we can translate it so that it belongs to a square $Q$ with vertices 
$(-a_3,\lambda),(-\lambda,a_3),(-a_3,2a_3-\lambda),(-2a_3+\lambda,a_3)$.
If we can show that $|F_i|\le 2$ on $Q$, then we are done.  
The properties of $F_i$ give directly that the minimum of $F_i$ is reached
at $(\frac{-3a_3+\lambda}{2},\frac{3a_3-\lambda}{2})$ and its maximum is
reached at $(-\lambda,a_3)$. Hence, to conclude it is sufficient to show that
$$-2\le F_i(\frac{-3a_3+\lambda}{2},\frac{3a_3-\lambda}{2})\text{ and }F_i(-\lambda,a_3)\le 2.$$
These inequations are implied by the following ones
$$-\cos(\alpha_i)+\sin(\alpha_i)\sinh(\frac{3a_3-\lambda}{2})\le \sqrt{\tanh(a_j)\tanh(a_k)}$$
$$\cos(\alpha_i)\cosh(\frac{a_3-\lambda}{2})-\sin(\alpha_i)\sinh(\frac{a_3+\lambda}{2})\le \sqrt{\tanh(a_j)\tanh(a_k)}.$$
The two left hand sides are monotone functions of $\alpha_i$.
It is sufficient to show that the first one is satisfied for $a_j=a_k=a_1$ and $\alpha_i=\alpha_M$ whereas the second one has to be satisfied for $a_j=a_k=a_1$ and $\alpha_i=\alpha_m$ where $\sin(\alpha_m)=\frac{\sinh(a_1)}{\sinh(2a_3)}$.
This is precisely the content of the inequalities (3) and (4). \end{proof}

\begin{lemma}[Isosceles strategy]\label{iso}
Suppose that $a_1\le a_2\le a_3$ and that the following inequations hold:
$$\sinh(a_3)\ge 2\text{ and }\sinh(a_1)+\sinh(a_1)^{-1}\le \sinh(a_3).$$
Define $\lambda$ as the largest solution of the equation $\cosh(\lambda)^2=\sinh(\lambda)\sinh(a_3)$. 
Suppose that $|t_i|\le a_i$ for all $i$ and that the following assumptions hold:
\begin{enumerate}
\item $a_2\ge \lambda$
\item $1+\sin(\alpha_1^M)\sinh(a_3)\le\sqrt{\cosh(\lambda)\cosh(a_3)}$
\item $-\cos(\alpha_1^M)+\sin(\alpha_1^M)\sinh(\frac{3a_3-\lambda}{2})\le \sqrt{\cosh(\lambda))\cosh(a_3)}$
\item $\cos(\alpha_1^m)\cosh(\frac{a_3-\lambda}{2})-\sin(\alpha_1^m)\sinh(\frac{a_3+\lambda}{2})\le \sqrt{\cosh(\lambda)\cosh(a_3)}$
\item $\cosh(a_1)^2\cosh(2a_3-\lambda)\le \cosh(a_3)\sinh(a_3)\sinh(a_1)$
\end{enumerate}
$$\text{where }\sin(\alpha_1^m)=\frac{\sinh(a_1)}{\sinh(2a_3)}\text{ and }\cos(2\alpha_1^M)=\frac{\cosh(2\lambda)\cosh(2a_3)-\cosh(2a_1)}{\sinh(2\lambda)\sinh(2a_3)}.$$ 
Then up to a translation $t_2\mapsto t_2-2a_2$ or $t_3\mapsto t_3+2a_3$ one has 
$|\tr \rho(\beta_i)|\le 2\cosh(a_3)$ for $i=1,2,3$. 
\end{lemma}
\begin{proof}
We apply the same strategy as in Lemma \ref{equi}, but only for $i=1$. Recall that we introduced the following function: 
$$F_1(x,y)=\frac{2}{\sqrt{\tanh(a_2)\tanh(a_3)}}(\cos(\alpha_1)\cosh(\frac{x+y}{2})+\sin(\alpha_1)\sinh(\frac{x-y}{2})).$$
Notice that $\lambda$ satisfies $0\le \lambda\le a_3$ and Lemma \ref{majo} gives $F_1(a_3,\lambda)\le 2\cosh(a_3)$ under the hypothesis (1). 
We have also $F_1(a_3,-a_3)\le 2\cosh(a_3)$ by the assumption (2). Hence, $F_1$ is less than $2\cosh(a_3)$ on the polygon $P$. 

If $(t_2,t_3)\in P$ then it follows that $|\tr \rho(\delta_1)|\le 2\cosh(a_3)$.
The estimation of Lemma \ref{majo} and the assumption $\sinh(a_1)+\sinh(a_1)^{-1}\le\sinh(a_3)$ imply that the traces of $\beta_2$ and $\beta_3$ are also less than $2\cosh(a_3)$. 

If $(t_2,t_3)\notin P$ then we set $t_2'=t_2-2a_2$ if $t_2>0$ and
$t_3'=t_3+2a_3$ if $t_2<0$. As for Lemma \ref{equi}, $(t'_2,t'_3)$
belongs to the polygon $Q$ and we have the estimation 
$$F_1(\frac{-3a_3+\lambda}{2},\frac{3a_3-\lambda}{2})\le F_1(t'_2,t'_3)\le F_1(-\lambda,a_3).$$
Hence, to prove the lemma it is sufficient to have
$$ -\cos(\alpha_1)+\sin(\alpha_1)\sinh(\frac{3a_3-\lambda}{2})\le \sqrt{\cosh(a_3)\cosh(\lambda)}$$
$$ \cos(\alpha_1)\cosh(\frac{a_3-\lambda}{2})-\sin(\alpha_1)\sinh(\frac{a_3+\lambda}{2})\le \sqrt{\cosh(a_3)\cosh(\lambda)}.$$
It is sufficient to prove the first inequality for $\alpha_1=\alpha_M$ and the second one for $\alpha_1=\alpha_m$.
This is the content of Equations (3) and (4).
Finally, we need to check that $F_2(t_1,t'_3)$ and $F_3(t_1,t'_2)$ are lower
than $2\cosh(a_3)$. Using Lemma \ref{majo} it is sufficient to have the following inequality, provided by the assumption (5):
$$\sqrt{\frac{\cosh(a_1)\cosh(2a_3-\lambda)}{\tanh(a_1)\tanh(a_3)}}\le \cosh(a_3).$$ 
\end{proof}

{\bf Conclusion}.

Set $l_1(a_3)=-0.9(a_3-1.695)+1.18$ and $l_2(a_3)=0.8(a_3-1.695)+1.18$.
Let $X=\{(a_1,a_2,a_3)\in \R^3\text{ s.t. }0\le a_1\le a_2\le a_3\le 2.23\}$. We divide $X$ into the following compact sets shown in Figure \ref{graphique}:
$$X_1=\{(a_1,a_2,a_3)/a_1\le l_1(a_3)\}$$
$$X_2=\{(a_1,a_2,a_3)/a_1\ge l_1(a_3) \text{ and }a_1\ge l_2(a_3)\}$$
$$X_3=\{(a_1,a_2,a_3)/l_1(a_3)\le a_1\le l_2(a_3)\text{ and }\cosh(a_2)^2\le \sinh(a_2)\sinh(a_3)\}$$
$$X_4=\{(a_1,a_2,a_3)/, l_1(a_3)\le a_1\le l_2(a_3)\text{ and }\cosh(a_2)^2\ge \sinh(a_2)\sinh(a_3)\}.$$
\begin{figure}
\centering
\def\svgwidth{9cm}
 \executeiffilenewer{graph2.svg}{graph2.pdf}%
 {inkscape -z -D --file=graph2.svg %
 --export-pdf=graph2.pdf --export-latex}%
 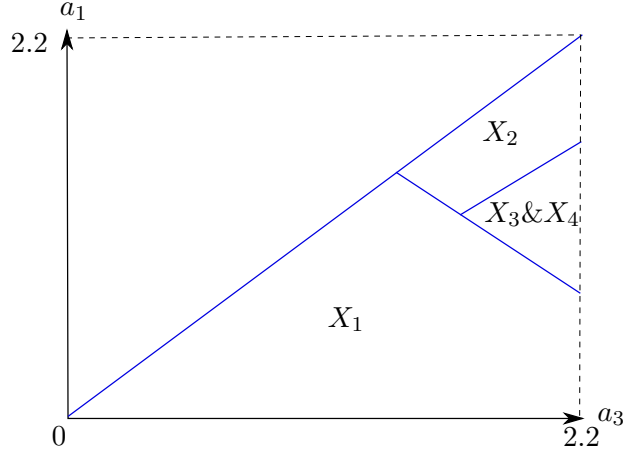%

\caption{Three strategies in Euler class $\pm 1$}
\label{graphique}
\end{figure}
Let $\rho=\rho^{0,\pm 1}_{a_1,a_2,a_3}(t_1,t_2,t_3)$ be a representation
such that $a_3\le 2.23$. Set $A=(a_1,a_2,a_3)$: 
\begin{itemize}
\item[-] if $A\in X_1$ then $\Phi(a_1,a_3)<9$ and Lemma \ref{com} implies
  that $\rho$ has a non-hyperbolic curve. 
\item[-] If $A\in X_2$ and $a_3\ge 1.69$, then $a_1\le l_1(a_3)$ and a
  computer check shows that all inequalities of Lemma \ref{equi} are
  satisfied for $a_1=l_1(a_3)$ and greater. For $a_3\ge 1.69$, the same
  hold for $a_1=l_2(a_3)$ and greater.
\item[-] Suppose now that $A$ satisfies $ l_1(a_3)\le a_1\le l_3(a_1)$.
  If $A\in X_3$, that is $\cosh(a_2)^2\le \sinh(a_2)\sinh(a_3)$, then as we
  also have $\cosh(a_1)^2\le \sinh(a_1)\sinh(a_3)$,  Lemma \ref{majo} applies.
\item[-] Else $A$ belongs to $X_4$ and a computer check shows that the
  inequalities of Lemma \ref{iso} are satisfied. 
\end{itemize}
In any case, the strategy either produces non-hyperbolic curves or strictly
decreases the value of $2\cosh(a_3)$.

%

%% file: graph2.pdf_tex
\begingroup%
  \makeatletter%
  \providecommand\color[2][]{%
    \errmessage{(Inkscape) Color is used for the text in Inkscape, but the package 'color.sty' is not loaded}%
    \renewcommand\color[2][]{}%
  }%
  \providecommand\transparent[1]{%
    \errmessage{(Inkscape) Transparency is used (non-zero) for the text in Inkscape, but the package 'transparent.sty' is not loaded}%
    \renewcommand\transparent[1]{}%
  }%
  \providecommand\rotatebox[2]{#2}%
  \ifx\svgwidth\undefined%
    \setlength{\unitlength}{759.64511719bp}%
    \ifx\svgscale\undefined%
      \relax%
    \else%
      \setlength{\unitlength}{\unitlength * \real{\svgscale}}%
    \fi%
  \else%
    \setlength{\unitlength}{\svgwidth}%
  \fi%
  \global\let\svgwidth\undefined%
  \global\let\svgscale\undefined%
  \makeatother%
  \begin{picture}(1,0.68210607)%
    \put(0,0){\includegraphics[width=\unitlength]{graph2.pdf}}%
    \put(0.05517732,0.64107422){\color[rgb]{0,0,0}\makebox(0,0)[lb]{\smash{$a_1$}}}%
    \put(0.84953317,0.04229831){\color[rgb]{0,0,0}\makebox(0,0)[lb]{\smash{$a_3$}}}%
    \put(0.45085078,0.18070886){\color[rgb]{0,0,0}\makebox(0,0)[lb]{\smash{$X_1$}}}%
    \put(0.67952902,0.45602537){\color[rgb]{0,0,0}\makebox(0,0)[lb]{\smash{$X_2$}}}%
    \put(0.68253798,0.3356684){\color[rgb]{0,0,0}\makebox(0,0)[lb]{\smash{$X_3 \& X_4$}}}%
    \put(0.79838148,0.00468676){\color[rgb]{0,0,0}\makebox(0,0)[lb]{\smash{$2.2$}}}%
    \put(-0.01703687,0.58691354){\color[rgb]{0,0,0}\makebox(0,0)[lb]{\smash{$2.2$}}}%
    \put(0.04314163,0.00619124){\color[rgb]{0,0,0}\makebox(0,0)[lb]{\smash{$0$}}}%
  \end{picture}%
\endgroup%

%% file: Ergo.tex
As in the preceding section, and until the end of this text, we will
say for simplicity that a representation {\em has a
non-hyperbolic curve} if it maps some closed simple loop to a
non-hyperbolic element of $\psl$.

Let $\Sigma$ be a closed surface, and let
$\mathcal{S}$ denote the set of simple closed curves in $\Sigma$
(we view it as a subset of the set of conjugacy classes in
$\pi_1(\Sigma)\smallsetminus\{1\}$). Let
\[
\mathcal{NH}^k=\{ [\rho]\in\mathcal{M}^k;\, 
\exists [\gamma]\in\mathcal{S}\text{ s.t. }|\tr(\rho(\gamma))|\leq 2 \}
\]
be the set of non-elementary representations which have a
non-hyperbolic curve.
Obviously $\mathcal{NH}^k$ is invariant under the action of $\mcg(\Sigma)$.
Also, let $\mathcal{S}^{ns}$ be the subset of $\mathcal{S}$ consisting of non-separating simple closed curves. We define
\[ \mathcal{E}^k=\{ [\rho]\in\mathcal{M}^k;\, \exists [\gamma]\in\mathcal{S}^{ns}\text{ s.t. }\rho(\gamma)\text{ is elliptic}\}. \]
It is a $\mcg(\Sigma)$-invariant open subset of $\mo^k$.
We will also consider its subset
\[ \mathcal{EI}^k=\{ [\rho]\in\mathcal{M}^k;\, \exists [\gamma]\in\mathcal{S}^{ns}\text{ s.t. }\rho(\gamma)\text{ is elliptic of infinite order}\}. \]
It comes for free that $\mathcal{EI}^k$ has full measure in $\mathcal{E}^k$. 
 Indeed, $\mathcal{E}^k$ is open, and
up to the action of the mapping class group there
is only one non-separating simple closed loop, and it is easy to see
that its trace defines an algebraic function which is non constant, on each
connected component of $\mathcal{M}$.

The aim of this section is to prove the the following statements:
\begin{theorem}\label{Ergod}
  Let $\Sigma$ be a closed surface of genus $g\geq 2$, and let
  $k\in\{3-2g,\ldots,2g-3\}$. Then $\mathcal{EI}^k$ is non-empty, connected,
  and the action of $\mcg(\Sigma)$ on $\mathcal{E}^k$ is ergodic.
\end{theorem}

\begin{proposition}\label{EDense}
  Suppose that $(g,k)\neq (2,0)$. 
  Then $\mathcal{E}^k$ has full measure in $\mathcal{NH}^k$.
\end{proposition}
These two statements, together with Theorem \ref{Bowditch}, imply
Theorems \ref{ergodicite} and \ref{ergodgeneral}.

Now we draw a proof Theorem \ref{Ergod}, using some propositions which will
be proved in the rest of this section.

\begin{proposition}\label{EConnexe}
 The space $\mathcal{EI}^k$ is connected.
\end{proposition}

The second step takes its inspiration from Goldman and Xia's paper
\cite{GoldmanXia11}.
To every closed loop $\gamma$ in $\Sigma$,
associate the function
$f_\gamma\colon\mathcal{M}^k\rightarrow\R_+$ defined by
$f_\gamma(\rho)=\tr(\rho(\gamma))^2$. 

We set $\mathcal{U}^k$
to be the set of $[\rho]\in\mathcal{E}^k$ such that there exist  simple curves $\gamma_1,\ldots,\gamma_{6g-6}$ such that the
cotangent $T^*_{[\rho]}\mathcal{M}^k$ is generated by the differentials
$df_{\gamma_1},\ldots,df_{\gamma_{6g-6}}$ and such that $|\tr \rho(\gamma_i)|<2$ for all $i$.
\begin{proposition}\label{ToutesLesDirections}
  The space $\mathcal{U}^k$ is an open subspace of $\mathcal{M}^k$ containing $\mathcal{EI}^k$.
\end{proposition}
This should be compared to Theorem 2.1 of \cite{GoldmanXia11}
which is a key step for their proof of ergodicity
of $\mcg(\Sigma)$ on the $\mathrm{SU}_2$-character variety. Proposition
\ref{ToutesLesDirections} thus enable to adapt part of their
strategy here.

The following fact is then directly adapted from \cite{GoldmanXia11}:
\begin{proposition}\label{Voisinage}
  Let $f\colon\mathcal{U}^k\rightarrow\R$ be a measurable function, invariant
  under $\mcg(\Sigma)$. Then every $[\rho]\in\mathcal{U}^k$ has a
  neighbourhood $\mathcal{V}_{[\rho]}$ such that $f$ is almost everywhere
  constant on $\mathcal{V}_{[\rho]}$.
\end{proposition}
\begin{proof}[Sketch of proof]
Pick $[\rho]\in\mathcal{U}^k$ and let $\gamma_1,\ldots,\gamma_{6g-6}\in \mathcal{S}$ be simple curves such that $df_1,\ldots,df_{6g-6}$ form a basis of $T^*_{[\rho]}\mathcal{M}^k$ and such that $0<f_i(\rho)<4$ for all $i$ in $\{1,\ldots,6g-6\}$. For every $i$, let $X_i$ be the Hamiltonian vector field of the function $h_i=\arccos(\sqrt{f_i}/2)$. Its flow $\Phi_i$ is $2\pi$-periodic and the Dehn twist $\tau_i$ along $\gamma_i$ acts on $\mo^k$ by the formula 
$$\tau_i. \phi= \Phi_i^{h_i(\phi)}. \phi.$$
This implies that if $h_i(\phi)\notin 2\pi\Q$, a function $f:\mo^k\to \R$ invariant by the Dehn twist $\tau_i$ satisfies $f(\Phi^\theta_i(\phi))=f(\phi)$ for almost every $\theta\in \R/2\pi\Z$. There is a neighborhood $\mathcal{V}_{[\rho]}$ of $[\rho]$ in $\mo^k$ such that for all $i\in\{1,\ldots,6g-6\}$ and $[\phi]\in \mathcal{V}_{[\rho]}$ one has $0<f_i(\phi)<4$ and 
$$\operatorname{Span}\{X_1(\phi),\ldots,X_{6g-6}(\phi)\}=T_{[\phi]}\mathcal{M}^k.$$
Up to shrinking $\mathcal{V}_{[\rho]}$, one can suppose that the flows $\Phi_i$ act transitively on $\mathcal{V}_{[\rho]}$ and on almost any orbit of these flows, $f$ is almost constant. A standard measure theoretic argument implies that $f$ is almost constant on $\mathcal{V}_{[\rho]}$, see \cite{GoldmanXia11} for details.
\end{proof}

Now we give a proof of Theorem \ref{Ergod}, assuming  Proposition \ref{EConnexe}
and \ref{ToutesLesDirections}.

\begin{proof}[Proof of Theorem \ref{Ergod}]
  Let $f\colon\mathcal{E}^k\rightarrow\R$ be a measurable function,
  invariant under $\mcg(\Sigma)$.
  Proposition
  \ref{ToutesLesDirections} implies that $\mathcal{U}^k$ has full measure
  in $\mathcal{E}^k$. Therefore it suffices to prove that $f$ is
  almost everywhere constant on $\mathcal{U}^k$.
  Proposition \ref{Voisinage} enables to define a function
  $\bar{f}\colon\mathcal{U}^k\rightarrow\R$ as follows: for every
  $[\rho]\in\mathcal{U}^k$, $f$ is almost everywhere constant on
  $\mathcal{V}_{[\rho]}$; let $\bar{f}([\rho])$ be this constant.
  Clearly $\bar{f}$ is well-defined, and locally constant on
  $\mathcal{U}^k$. By Propositions \ref{EConnexe} and \ref{ToutesLesDirections}, $\mathcal{U}^k$ is connected.
  Hence $\bar{f}$ is constant on $\mathcal{U}^k$ and $f$ - being 
  almost everywhere equal to $\bar{f}$ - is also almost
  everywhere constant on $\mathcal{U}^k$.
\end{proof}

Propositions \ref{EConnexe} and \ref{ToutesLesDirections}
are proved in Sections \ref{SectionEConnexe} 
and \ref{SectionToutesLesDirections}, respectively.
Finally, Proposition \ref{EDense} is proved in Section \ref{SectionEDense}.


\subsection{Connectedness of $\mathcal{EI}^k$}\label{SectionEConnexe}

Let $i\colon\mathcal{S}\times\mathcal{S}\rightarrow\N$ denote
the minimal geometric intersection number. Note that $i(a,b)=1$
implies that $a$ and $b$ are non-separating, and that a neighborhood
of $a\cup b$ is homeomorphic to a one-holed torus embedded in $\Sigma$.
Let $\mathcal{C}$ denote the set of $(a,b)\in\mathcal{S}\times\mathcal{S}$
such that $i(a,b)=1$.

If $(a,b)\in\mathcal{C}$, denote by $\mathcal{EI}_{(a,b)}^k$ the set of classes
of representations $\rho$ of Euler class $k$ such that
$\rho(a)$ and $\rho(b)$ are elliptic, do not commute with each
other, and such that at least one of them has infinite order. Notice that this implies that $\rho([a,b])$ is hyperbolic. 

\begin{lemma}\label{ChacunConnexe}
  For every $(a,b)\in\mathcal{C}$,
  for every $k\in\{3-2g,\ldots,2g-3\}$,
  $\mathcal{EI}_{(a,b)}^k$ is non-empty and connected.
\end{lemma}

\begin{proof}
  Let $\Sigma_0$ be a one-holed torus containing $a$ and $b$ in $\Sigma$,
  and let $\Sigma_1$ be the complementary genus $g-1$ surface with one
  boundary component.
  
  First we analyse the space $\mathcal{M}_0$ of
  conjugacy classes of representations
  of $\pi_1(\Sigma_0)$ sending $a$ and $b$ to non commuting elliptic
  elements. If $\rho\in\mathcal{M}_0$, the distance $d(\rho)$ between
  the fixed points of $\rho(a)$ and $\rho(b)$ is positive, and the
  angles $\theta_a(\rho)$ and $\theta_b(\rho)$ of these two rotations
  lie in $(0,2\pi)$. The map
  $\mathcal{M}_0\rightarrow(0,+\infty)\times(0,2\pi)\times(0,2\pi)$
  defined by $(d,\theta_a,\theta_b)$ is a homeomorphism. In $(0,2\pi)\times(0,2\pi)$, the set of
  pairs $(\theta_a,\theta_b)$ such that $\theta_a/\pi$ or $\theta_b/\pi$ is irrational is path-connected.
  Hence, the subspace $\mathcal{M}_0'$ of $\mathcal{M}_0$ consisting of
  classes of representations such that $\rho(a)$ or $\rho(b)$ has
  infinite order is path-connected.

  Now, given two elements $\rho$, $\rho'$ in $\mathcal{EI}_{(a,b)}^k$,
  we can consider their restrictions $\rho_0$, $\rho_0'$ to
  $\pi_1(\Sigma_0)$, and connect them by a path in $\mathcal{M}_0'$.
  Theorem \ref{fibr1} allows us to complete it to a path in
  $\mathcal{EI}_{(a,b)}^k$.
\end{proof}

We define an equivalence relation $\sim$ on $\mathcal{C}$
generated by: $(a,b)\sim(a',b')$ if
$\mathcal{EI}_{(a,b)}^k\cap\mathcal{EI}_{(a',b')}^k\neq\emptyset$.

The rest of this subsection will consist in the proof of the
following fact:

\begin{proposition}\label{RelationTransitive}
  The equivalence relation $\sim$ has a unique class in $\mathcal{C}$.
\end{proposition}

Lemma \ref{ChacunConnexe} and Proposition \ref{RelationTransitive}
together imply that the union
$\cup_{(a,b)\in\mathcal{C}}\mathcal{EI}_{(a,b)}^k$
is connected. Therefore, in order to prove Proposition \ref{EConnexe}
we just have to notice that
$\cup_{(a,b)\in\mathcal{C}}\mathcal{EI}_{(a,b)}^k=\mathcal{EI}^k$.
The first inclusion follows from the definitions.
Now let $[\rho]\in\mathcal{EI}^k$. By definition, there exists a non-separating simple
closed curve $a$ such that $\rho(a)$ is elliptic of infinite order.
Since $\rho$ is non-elementary, there exists a simple closed curve
$b$ such that $(a,b)\in\mathcal{C}$ and such that $\rho(a)$ and
$\rho(b)$ do not commute
(indeed, curves $b$ such that $(a,b)\in\mathcal{C}$ generate $\pi_1(\Sigma)$).
And for every $n\in\Z$, we also have $(a,ba^n)\in\mathcal{C}$. The
existence of some $n\in\Z$ such that $[\rho]\in\mathcal{EI}_{(a,ba^n)}^k$
follows from the following remark.

\begin{lemma}\label{multiplicationdespetitselliptiques}
  Let $A,B\in\psl$. Suppose that $A$ is elliptic of infinite order.
  Then there exists $n\in\Z$ such that $BA^n$ is elliptic, not of order 2. 
\end{lemma}

\begin{proof}
  In an adapted basis we may write
  $A=\pm\left(\begin{array}{cc}\cos\theta & \sin\theta \\ -\sin\theta & \cos\theta \end{array}\right)$
  and
  $B=\pm\left(\begin{array}{cc}x & y \\ z & t \end{array}\right)$.
  This gives $\pm \tr(B\cdot A^n)=(x+t)\cos(n\theta)+(z-y)\sin(n\theta)$.
  This is the scalar product of the vectors $(\cos(n\theta),\sin(n\theta))$ and
  $((x+t),(z-y))$ of $\R^2\smallsetminus\{(0,0)\}$. For a suitable $n$, this trace can be taken in a dense neighborhood of $0$.
\end{proof}

Now we focus on the proof of Proposition \ref{RelationTransitive}.

\begin{lemma}\label{LemConnex0}
  Let $(a,b)\in\mathcal{C}$. For all $n\in\Z$, $(a,b)\sim(a,ba^n)$.
  In other words, we do not leave an equivalence class when we apply
  Dehn twists in the handle defined by $(a,b)$.
\end{lemma}

\begin{proof}
  Put $A=\pm\left(\begin{array}{cc}0 & -1 \\ 1 & 0\end{array}\right)$, and
  $B=\pm\left(\begin{array}{cc}1 & -1 \\ 1-\varepsilon & \varepsilon\end{array}\right)$,
  for $\varepsilon\in(0,1)$. Then $|\tr(BA)|=|2-\varepsilon|$. Since $A$ is of
  order $2$, for all $n\in\Z$, $A$, $B$ and $BA^n$ are elliptic. Moreover
  $\varepsilon$ can be chosen so that $B$ has infinite order. Now $[A,B]$ is
  hyperbolic. We can define a representation of the fundamental group of the
  one holed torus $\langle a,b\rangle$ by sending $a$ to $A$ and $b$ to $B$.
  This representation is hyperbolic at the boundary, and has Euler class $0$
  (indeed, eg apply Theorem 3.4 of \cite{Goldman88} when $M$ is
  a one holed torus). By Theorem \ref{fibr0} we can complete this
  representation to a representation of $\pi_1(\Sigma)$ of Euler class
  $k$, provided that $|k|\leq 2g-3$. By construction, this representation
  lies in $\mathcal{EI}_{(a,b)}^k\cap\mathcal{EI}_{(a,ba^n)}^k$, which therefore
  is non-empty, as claimed.
\end{proof}

\begin{lemma}\label{LemConnex1}
  Let $(a,b)\in\mathcal{C}$, and let $c$ be a non-separating simple
  closed curve, disjoint from $b$, and such that $(a,c)\in\mathcal{C}$.
  Then $(a,b)\sim(a,c)$.
\end{lemma}

\begin{proof}
  Let $[\rho]\in\mathcal{EI}_{(a,b)}^k$ be such that $\rho(a)$ has infinite
  order. Let $\varphi\in \Aut(\pi_1\Sigma)$ be associated to a Dehn twist of
  order $n$ along a curve freely homotopic to $a$. Up to conjugacy we have
  $\phi(a)=a$, $\phi(b)=ba^n$ and $\phi(c)=ca^n$. It follows from
  Lemma \ref{multiplicationdespetitselliptiques} that for a suitable $n$,
  $\rho\circ\varphi(c)$ is elliptic. Thus $[\rho\circ\varphi]\in\mathcal{EI}_{(a,c)}^k$.
  Also, $\phi(ba^{-n})=b$, hence $\rho\circ\varphi(ba^{-n})=\rho(b)$ is elliptic.
  Therefore $[\rho\circ\varphi]\in\mathcal{EI}_{(a,c)}^k\cap\mathcal{EI}_{(a,ba^{-n})}^k$,
  and $(a,c)\sim(a,ba^{-n})$. By Lemma \ref{LemConnex0} we have $(a,b)\sim(a,ba^{-n})$,
  therefore $(a,b)\sim(a,c)$.
\end{proof}

\begin{lemma}\label{LemConnex2}
  Let $b$ be a non-separating simple closed curve and let $a$ and $a'$
  be such that $(a,b)\in\mathcal{C}$ and $(a',b)\in\mathcal{C}$.
  Then $(a,b)\sim(a',b)$.
\end{lemma}

\begin{proof}
 Up to homeomorphism of $\Sigma$, and up to free homotopy, $a$ and $b$ look as in Figure \ref{FigureInHandle}.
  \begin{figure}[htbp]
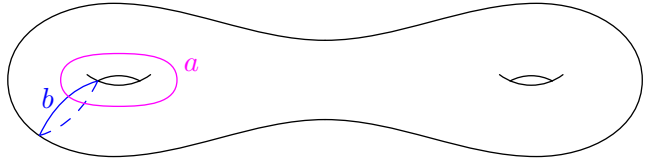

  \begin{asy}
    import geometry;
    unitsize(2pt);

    //
    //
    path contg=((0,0){dir(180)}..(-40,-7){dir(180)}..(-60,7.5){dir(90)}..(-40,22){dir(0)}..(0,15){dir(0)});
    draw(contg);
    path contd=shift((0,15))*rotate(180)*contg;
    draw(contd);
    //

    //
    //
    picture trou;
    draw (trou,(-6,8.5)..(0,6.5)..(6,8.5));
    draw (trou,(-4,7.3)..(0,8.3)..(4,7.3));
    //
    picture troudgauche=shift(-39,0)*trou;
    add(troudgauche);
    //
    picture trouddroite=shift(39,0)*trou;
    add(trouddroite);

    //
    //draw ((0,0)..(-2.5,7.5)..(0,15),red);
    //draw ((0,0)..(2.5,7.5)..(0,15),red);

    draw ((-54,-3){dir(65)}..(-43,7.3){dir(15)},blue);
    draw ((-54,-3){dir(15)}..(-43,7.3){dir(65)},blue+dashed);

    label("{\small $b$}",(-50,1),NW,blue);

    //
    draw ((-50,7.5)..(-39,12.5)..(-28,7.5)..(-39,2.5)..cycle,magenta);
    label("{\small $a$}",(-28,8),NE,magenta);

    //label("{\small $*$}",(0,0),NW);
  \end{asy}
  \caption{In a handle}
  \label{FigureInHandle}
  \end{figure}
   If $a'$ can be freely homotoped inside the one holed torus defined by
  thickening $a\cup b$, then we apply Lemma \ref{LemConnex0} and conclude
  the proof in this case. Otherwise, we proceed by induction on the
  intersection number $i(a,a')$. If this number can be decreased by
  applying a Dehn twist along $b$ to $a'$, then we are done, because
  applying such a twist does not change the class of $(a',b)$ by
  Lemma \ref{LemConnex0}.
  Otherwise, in this one-holed torus $a'$ looks as in Figure \ref{FigureIntersect}
  (we cut this picture along $b$ for graphical convenience).
  \begin{figure}
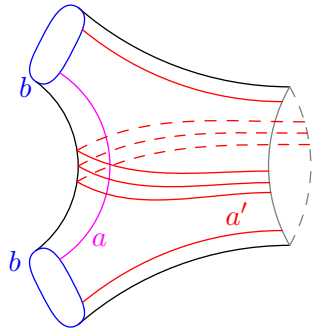

  \begin{asy}
    import geometry;
    unitsize(2pt);

    draw ((-40,30)..(0,15){right});
    draw ((-40,-30)..(0,-15){right});
    path LigneGauche=((-48,16)..(-40,0)..(-48,-16));
    draw(LigneGauche);

    path LigneBleueHaut=((-41,22)..(-48,16)..(-47,24)..(-40,30)..cycle);
    path LigneBleueBas=((-41,-22)..(-48,-16)..(-47,-24)..(-40,-30)..cycle);
    draw (LigneBleueHaut,blue);
    draw (LigneBleueBas,blue);

    path BordDevant=((0,15)..(-4,0)..(0,-15));
    draw (BordDevant,grey);
    path BordDerriere=((0,15)..(4,0)..(0,-15));
    draw (BordDerriere,grey+dashed);

    draw (relpoint(LigneBleueHaut,0.13)..(-34,0)..relpoint(LigneBleueBas,0.13),magenta);

    point Devant1=relpoint(BordDevant,0.6);
    point Gauche1=relpoint(LigneGauche,0.5);
    point Derriere1=relpoint(BordDerriere,0.3);
    point Devant2=relpoint(BordDevant,0.66);
    point Gauche2=relpoint(LigneGauche,0.58);
    point Derriere2=relpoint(BordDerriere,0.37);
    point Devant3=relpoint(BordDevant,0.53);
    point Gauche3=relpoint(LigneGauche,0.42);
    point Derriere3=relpoint(BordDerriere,0.23);
    draw (Gauche1{dir(-30)}..Devant1{right},red);
    draw (Gauche1{dir(30)}..Derriere1{right},red+dashed);
    draw (Gauche2{dir(-30)}..Devant2{right},red);
    draw (Gauche2{dir(30)}..Derriere2{right},red+dashed);
    draw (Gauche3{dir(-30)}..Devant3{right},red);
    draw (Gauche3{dir(30)}..Derriere3{right},red+dashed);

    draw (relpoint(LigneBleueHaut,0.89)..relpoint(BordDevant,0.1){right},red);
    draw (relpoint(LigneBleueBas,0.89)..relpoint(BordDevant,0.9){right},red);

    label("{\small $a$}",(-36,-14),magenta);
    label("{\small $a'$}",(-10,-9),red);
    label("{\small $b$}",(-50,15),blue);
    label("{\small $b$}",(-52,-18),blue);

    //label("{\small $*$}",(-8,0));
  \end{asy}
  \caption{Reducing the intersection number}
  \label{FigureIntersect}
  \end{figure}

  The intersection number $i(a,a')$ is seen as the number
  of horizontal red strings
  in Figure \ref{FigureIntersect}. If $i(a,a')=0$, then Lemma \ref{LemConnex0}
  can be applied. Now suppose that $i(a,a')=N>0$. In Figure 
  \ref{FigureIntersect}, we may define a new curve $a''$ as follows. Start from
  the lower intersection point between $a'$ and $b$, and follow
  the curve $a'$ to the right. Then follow $a'$ until you
  reach the uppermost horizontal red string, and then stop following
  $a'$ and go directly to hit $b$ (at its upper representative in
  the picture). This defines a curve $a''$ which, up to applying
  to it a Dehn twist along the blue curve $b$, does not intersect
  $a'$. Hence by Lemma \ref{LemConnex0} we have $(a',b)\sim(a'',b)$.
\end{proof}

We can now prove Proposition \ref{RelationTransitive}.

\begin{proof}[Proof of Proposition \ref{RelationTransitive}]
  Let $(a,b), (a',b')\in\mathcal{C}$. It is well-known that
  the 1-skeleton of the curve complex of $\Sigma$ is connected
  (see \cite{MasurMinsky99}, Lemma 2.1). In other words, there
  exist $b_0=b$, \ldots, $b_n=b'$ such that for every
  $i\in\{0,\ldots,n-1\}$, $b_i$ and $b_{i+1}$ are disjoint.
  Up to deleting some terms of this sequence, and up to
  inserting others, we may also suppose that for all $i$,
  $b_i$ is a non-separating simple curve. Thus, for all
  $i\in\{0,\ldots,n-1\}$,
  we can find a non-separating simple curve $a_i$ such that
  $(a_i,b_i)\in\mathcal{C}$ and $(a_i,b_{i+1})\in\mathcal{C}$.
  Put $a_n=a'$.
  By Lemma \ref{LemConnex1} we have $(a_i,b_i)\sim(a_i,b_{i+1})$
  for every $i\in\{0,\ldots,n-1\}$. By Lemma \ref{LemConnex2}
  we also have $(a_i,b_{i+1})\sim(a_{i+1},b_{i+1})$, thus
  $(a_i,b_i)\sim(a_{i+1},b_{i+1})$ for every $i\in\{0,\ldots,n-1\}$.
  Again by Lemma \ref{LemConnex2} we have $(a,b)\sim(a_0,b)$ and $(a_n,b')\sim(a',b')$ .
  Finally we have $(a,b)\sim(a',b')$.
\end{proof}


\subsection{$\mathcal{EI}^k$ is a subset of $\mathcal{U}^k$}\label{SectionToutesLesDirections}

Let $[\rho]\in\mathcal{EI}^k$ et $\gamma\in\mathcal{S}^{ns}$ such that
$\rho(\gamma)$ is elliptic of infinite order. We denote by
$D_{[\rho]}\subset T_{[\rho]}^*\mathcal{M}^k$ the linear subspace
generated by the differentials $df_{\gamma}$ of the traces of the
curves $\gamma$ such that $\rho(\gamma)$ is elliptic. We already have $df_\gamma\in D_{[\rho]}$.

\begin{lemma}\label{inter1}
  Let $\delta$ be such that $i(\gamma,\delta)=1$.
  Then $df_\delta\in D_{[\rho]}$.
\end{lemma}

\begin{proof}
We use the intersection point of $\gamma$ and $\delta$ as a base point and choose $A$ (resp. $B$) a lift of $\rho(\gamma)$ (resp. $\rho(\delta)$) to $\sldeuxR$. We are in the situation already described in Lemma \ref{multiplicationdespetitselliptiques}: with the same notation we have
\begin{equation}\label{twisttrace}
\pm\tr \rho (\delta_n)=(x+t)\cos(n\theta)+(z-y)\sin(n\theta).
\end{equation}
Take a neighborhood $\mathcal{V}$ of $[\rho]$ small enough so that there exist continuous lifts $A(\rho)$ and $B(\rho)\in$ SL$_2(\R)$  of $\rho(\gamma)$ and $\rho(\delta)$. We define the maps $F_n:\mathcal{V}\to \R$ by the formula $F_n(\rho)=\tr(BA^n)$ so that we have $f_{\delta_n}=F_n^2$. In particular, the derivatives $df_n$ and $dF_n$ are proportional so that it is sufficient to show that $dF_0\in D_{[\rho]}$. We also set $G(\rho)=\tr A(\rho)$. By assumption, we have $dG\in D_{\rho}$ because $f_\gamma(\rho)<4$. 

The trace identity implies that $F_{n+1}+F_{n-1}=G F_n$. By derivating this functional equation on $\mathcal{V}$ we obtain
$$dF_{n+1}+dF_{n-1}=dG F_n+GdF_n =GdF_n\mod D_{[\rho]}.$$
We conclude that the sequence $dF_n\in T^*_{[\rho]}\mathcal{M}^k/D_{\rho}$ satisfies an order 2 recursion so that there exists $a,b\in T^*_{[\rho]}\mathcal{M}^k$ such that 
$$ dF_n= \cos(n\theta)a+\sin(n\theta)b \mod D_{[\rho]}.$$

Observe from Equation \eqref{twisttrace} that $|\tr(BA^n)|<2$ for an infinitely many $n$'s. This implies that $dF_n\in D_{[\rho]}$ for such $n's$. As $\theta$ is irrational, this is impossible unless $a=b=0$. Finally we have $dF_n\in D_{[\rho]}$ for all $n$, in particular for $n=0$ and the lemma is proved. 
\end{proof}
\begin{lemma}\label{inter0}
  Let $\delta$ be such that $i(\gamma,\delta)=0$.
  Then $df_\delta\in D_{[\rho]}$.
\end{lemma}
\begin{proof}
Consider first the case when $\delta$ is non-separating. Then there exists a curve $\zeta$ such that $i(\gamma,\zeta)=i(\delta,\zeta)=1$ as shown in Figure \ref{TraceIdentities}. 
Up to replacing $\zeta$ by $\tau_\gamma^n\zeta$, we can suppose that $\tr \rho(\zeta)\ne 0$ as shown in Lemma \ref{multiplicationdespetitselliptiques}. 

\begin{figure}[htbp]
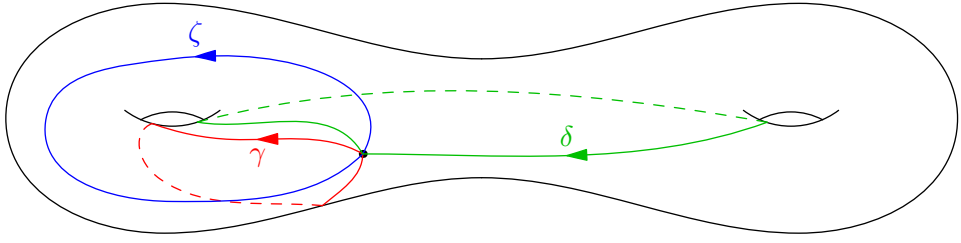

\centering
\begin{asy}
import geometry;
unitsize(3pt);

  //
  //
  path contg=((0,0){dir(180)}..(-40,-7){dir(180)}..(-60,7.5){dir(90)}..(-40,22){dir(0)}..(0,15){dir(0)});
  draw(contg);
  path contd=shift((0,15))*rotate(180)*contg;
  draw(contd);
  //

  //
  //
  picture trou;
  draw (trou,(-6,8.5)..(0,6.5)..(6,8.5));
  draw (trou,(-4,7.3)..(0,8.3)..(4,7.3));
  //
  picture troudgauche=shift(-39,0)*trou;
  add(troudgauche);
  //
  picture trouddroite=shift(39,0)*trou;
  add(trouddroite);

  point PBas=(-15,3);
  dot(PBas);

  // D'abord gamma
  //
  path BasTrouG=shift(-39,0)*((-6,8.5)..(0,6.5)..(6,8.5));
  point MBAG=relpoint(BasTrouG,0.3);

  point BasGamma=relpoint(contg,0.15);

  path DebutGamma=(PBas{dir(150)}..(MBAG+(10,-2))..MBAG{dir(160)});
  draw (DebutGamma,red,Arrow(Relative(0.5)));
  label("{\small $\gamma$}",relpoint(DebutGamma,0.5),S,red);
  draw (MBAG{dir(200)}..(MBAG+(1,-7))..BasGamma{dir(-5)},red+dashed);
  draw (BasGamma{dir(40)}..PBas{up},red);

  // Maintenant zeta
  path zeta=(PBas{dir(60)}..(-39,15)..(-55,7)..(-39,-3)..PBas{dir(45)});
  draw (zeta,blue,Arrow(Relative(0.3)));
  label("{\small $\zeta$}",relpoint(zeta,0.3),N,blue);

  // Maintenant delta
  point PDelt1=relpoint(BasTrouG,0.75);
  point PDelt2=(36,7);
  path delta1=(PBas{dir(120)}..PDelt1{dir(170)});
  path delta2=(PDelt1{dir(15)}..PDelt2{dir(-10)});
  path delta3=(PDelt2{dir(-160)}..PBas{left});
  draw (delta1,heavygreen);
  draw (delta2,heavygreen+dashed);
  draw (delta3,heavygreen,Arrow(Relative(0.5)));
  label("{\small $\delta$}",relpoint(delta3,0.5),N,heavygreen);
\end{asy}
\caption{The non-separating case}
\label{TraceIdentities}
\end{figure}

Take the intersection of $\gamma$ and $\zeta$ as a base point and consider a neighborhood of $[\rho]$ denoted by $\mathcal{V}$ so that $\rho(\gamma),\rho(\zeta)$ and $\rho(\delta)$ have lifts in SL$_2(\R)$ that we denote respectively by $A(\rho),B(\rho)$ and $C(\rho)$. 
The trace identity $\tr(C)\tr(B)=\tr(CB)+\tr(CB^{-1})$ can be reinterpreted as 
$$F_\delta F_\zeta=F_{\delta\zeta}+F_{\delta\zeta^{-1}}$$
where $F_\delta(\rho)=\tr (C(\rho)),F_\zeta(\rho)=\tr(B(\rho)),F_{\delta\zeta}(\rho)=\tr (C(\rho)B(\rho))$ and $F_{\delta\zeta^{-1}}(\rho)=\tr (C(\rho)B^{-1}(\rho))$.
From Lemma \ref{inter1}, we now that $dF_{\delta\zeta},dF_{\delta\zeta^{-1}}$ and $dF_\zeta$ belong to $D_{[\rho]}$ as $\delta\zeta,\delta\zeta^{-1}$ and $\zeta$ intersect $\gamma$ once. Hence from Leibnitz formula we get $dF_\delta F_\zeta\in D_{[\rho]}$. As $F_\zeta(\rho)\ne 0$ we finally get $dF_\delta\in D_{[\rho]}$. 

\begin{figure}[htbp]
\centering
\begin{asy}
  import geometry;
  unitsize(3pt);

  //
  //
  path contg=((0,0){dir(180)}..(-40,-7){dir(180)}..(-60,7.5){dir(90)}..(-40,22){dir(0)}..(0,15){dir(0)});
  draw(contg);
  path contd=shift((0,15))*rotate(180)*contg;
  draw(contd);
  //

  //
  //
  picture trou;
  draw (trou,(-6,8.5)..(0,6.5)..(6,8.5));
  draw (trou,(-4,7.3)..(0,8.3)..(4,7.3));
  //
  picture troudgauche=shift(-39,0)*trou;
  add(troudgauche);
  //
  picture trouddroite=shift(39,0)*trou;
  add(trouddroite);

  point PBas=(-15,3);
  dot(PBas);

  // D'abord gamma
  //
  path BasTrouG=shift(-39,0)*((-6,8.5)..(0,6.5)..(6,8.5));
  point MBAG=relpoint(BasTrouG,0.3);

  point BasGamma=relpoint(contg,0.15);
  //dot(BasGamma);

  path DebutGamma=(PBas{dir(150)}..(MBAG+(10,-2))..MBAG{dir(160)});
  draw (DebutGamma,red,Arrow(Relative(0.5)));
  label("{\small $\gamma$}",relpoint(DebutGamma,0.5),S,red);
  //draw (MBAG{dir(200)}..(MBAG+(-2,-10))..BasGamma{dir(-45)},Arrow(Relative(0.5)),red+dashed);
  draw (MBAG{dir(200)}..(MBAG+(1,-7))..BasGamma{dir(-5)},red+dashed);
  draw (BasGamma{dir(40)}..PBas{up},red);

  //Maintenant delta
  point PDelt1=relpoint(contg,0.9);
  point PDelt2=relpoint(contg,0.1);
  path delta1=(PBas{dir(80)}..PDelt1{dir(70)});
  path delta2=(PDelt1{dir(-70)}..PDelt2{dir(-115)});
  path delta3=(PDelt2{dir(115)}..PBas{dir(110)});
  draw(delta1,heavygreen,Arrow(Relative(0.8)));
  label("{\small $\delta$}",relpoint(delta1,0.8),W,heavygreen);
  draw(delta2,heavygreen+dashed);
  draw(delta3,heavygreen);

  // Maintenant zeta1
  point PZ1=relpoint(delta1,0.15);
  point PZ2=PZ1+(-5,6);
  path zeta1=(PZ1{left}..PZ2{dir(160)}..(-39,15)..(-55,7)..(-39,-3)..PBas{dir(25)});
  draw (zeta1,blue,Arrow(Relative(0.3)));
  label("{\small $\zeta_1$}",relpoint(zeta1,0.3),N,blue);
  // Maintenant zeta2
  path zeta2=(PBas{dir(15)}..(PBas+(20,0))..(39,-3)..(55,7)..(39,15)..PZ1{left});
  draw (zeta2,blue,Arrow(Relative(0.4)));
  label("{\small $\zeta_2$}",relpoint(zeta2,0.4),N,blue);
\end{asy}
\caption{The separating case}
\label{Figuuuure}
\end{figure}

Suppose now that $\delta$ is separating: as before we can choose $\zeta$ so that $i(\zeta,\gamma)=1$, $i(\zeta,\delta)=2$ and $\tr \rho(\zeta)\ne 0$. 
With the notation suggested in Figure \ref{Figuuuure} we have the following trace identity
\begin{eqnarray*}
\tr (B)\tr (C)&=& \tr (B_1B_2)\tr (C)=\tr(B_1B_2C)+\tr(B_1B_2C^{-1})\\
&=&\tr (B_1) \tr(B_2C)-\tr(B_1C^{-1}B_2^{-1})+\tr(B_2)\tr(B_1C^{-1})\\
&&-\tr(B_1B_2^{-1}C^{-1}).
\end{eqnarray*}
We interpret this equality in terms of trace functions in the following way 
\begin{equation}\label{tracensep}
F_\zeta F_\delta=F_{\zeta_1}F_{\zeta_2\delta}+F_{\zeta_2} F_{\zeta_1\delta^{-1}}-F_{\zeta_1\delta^{-1}\zeta_2^{-1}}-F_{\zeta_1\delta^{-1}\zeta_2^{-1}}.
\end{equation}
All curves $\xi$ involving $\zeta_1$ satisfy $dF_\xi\in D_{[\rho]}$ thanks to Lemma \ref{inter1} and all curves $\xi$ involving $\zeta_2$ are non-separating and  hence the same conclusion follows from the first case. Hence, we conclude again by derivating Equation \eqref{tracensep}.
\end{proof}

\begin{lemma}\label{poisson}
There exists a curve $\delta$ such that $i(\gamma,\delta)=1$ and $\{f_\gamma,f_{\delta}\}\ne 0$. In particular $df_\gamma\ne 0$ and $\mathcal{M}_\theta=f_\gamma^{-1}(4\cos(\theta)^2)$ is a subvariety of $\mathcal{M}(\Sigma)$
\end{lemma}
\begin{proof}
One can suppose with the notation of Lemma \ref{multiplicationdespetitselliptiques} that $\rho(\gamma)=A(\theta)$ and $\rho(\delta)=B$. We have $\{f_\gamma,f_\delta\}=0$ if and only if $\frac{d}{d\theta}\big{|}_{\theta=0}\tr(A(\theta)B)=\tr (A'(0)B)=0$. If for all $n\in\Z$ we have $\{f_\gamma,f_{\tau_\gamma^n\delta}\}=0$ then it follows that $\tr(A'(0)A(n\theta)B)=0$ but this is impossible because $\theta$ is irrational. 
\end{proof}
We can now finish the proof of Proposition \ref{SectionToutesLesDirections} by showing that $D_{[\rho]}=T^*_{[\rho]}\mo^k$. 

Let $\xi\in T_{[\rho]}\mo^k$ be orthogonal to $D_{[\rho]}$. Then we have $df_\gamma(\xi)=0$ which means that $\xi$ is tangent to the subvariety $\mathcal{M}_\theta(\Sigma)$. 
Let $r:\mathcal{M}_\theta(\Sigma)\to \mathcal{M}_\theta(\Sigma\setminus\gamma)$ be the restriction map. The representation $r(\rho)$ is again non-elementary for the following reason: as $\rho$ is non-elementary, there exists $\delta$ such that $i(\gamma,\delta)=1$ and such that $\rho(\gamma)$ and $\rho(\delta)$ do not commute. The commutator is a separating curve in $\Sigma\setminus \gamma$ whose image is a hyperbolic element. This prevents the restriction of $\rho$ to be elementary. In particular, the representation $r(\rho)$ is a smooth point of $\mathcal{M}(\Sigma\setminus\gamma)$. 

It is well-known that the differentials $df_{\delta}$ for $\delta$ disjoint from $\gamma$ generate the cotangent space of $\mathcal{M}(\Sigma\setminus \gamma)$, cf \cite{GoldmanXia11}, Lemma 3.1. As $df_\delta(\xi)=0$ by Lemma \ref{inter0}, we conclude that $Dr(\xi)=0$. 
On the other hand, the fiber of the map $r$ is given by the action of the twist flow $\Phi_\gamma$. In other words, the space $\ker Dr$ is generated by the symplectic gradient $X_{\gamma}$ of $f_\gamma$. Hence there exists $\lambda\in \R$ so that $\xi=\lambda X_\gamma$. 

Let $\delta$ be a curve given by Lemma \ref{poisson}. Then by Lemma \ref{inter1} we have $df_\delta(\xi)=0=\lambda df_\delta(X_\gamma)=\lambda\{f_\gamma,f_\delta\}$. As $\{f_\gamma,f_\delta\}\ne 0$ we have $\lambda=0$ and hence $\xi=0$ and the proposition is proved.


\subsection{$\mathcal{EI}^k$ has full measure in $\mathcal{NH}^k$}\label{SectionEDense}
As already observed in the beginning of Section \ref{SectionErgod}, 
for every $\gamma$ representing a simple closed curve,
and for every real $t$, the set
$\{[\rho]\in\mo\, |\,|\tr(\rho(\gamma))|=t\}$
is a proper semi-algebraic subvariety of each connected component of
$\mo$,
and the set $\mathcal{N}\subset\mathcal{M}$ of $[\rho]$
which map no simple closed curve to a parabolic or to an elliptic element
of finite order or to the identity, has full measure in $\mo$.

Thus, it suffices to prove that $\mathcal{EI}^k\cap\mathcal{N}$ has
full measure in $\mathcal{NH}^k\cap\mathcal{N}$.

Let $[\rho]\in\mathcal{NH}^k\cap\mathcal{N}(\Sigma)$.
There exists a simple closed curve $\gamma$ such
that $\rho(\gamma)$ is elliptic of infinite order.
If $\gamma$ is homotopic to a non-separating
curve then $[\rho]\in\mathcal{EI}^k(\Sigma)$. Therefore, it remains
to prove that if $\gamma$ is homotopic to a separating curve, then
we can find another curve which is mapped to an elliptic element of $\psl$.
As we saw in Section \ref{SectionNonErgod},
this is not true if $(g,k)=(2,0)$, but we will see that it
is true in every other case.

We will first do it in the simplest yet significative case:
\begin{lemma}\label{SepNonSepLem1}
  Let $\rho\colon\Gamma_2\rightarrow\psl$ be a representation
  of Euler class $\pm 1$, sending some separating simple
  closed curve to an elliptic element of infinite order. Then $\rho$
  sends a non-separating simple closed curve to an elliptic element.
\end{lemma}

\begin{proof}
  We use the same presentation of $\Gamma_2$ as in Section \ref{SectionNotation},
  and the same notation for
  the matrices as well.
  Up to conjugating $\rho$ by an inversion, we may suppose that
  $\eu(\rho)=1$. Then both $\rho([a_1,b_1])$ and $\rho([a_2,b_2])$,
  well-defined in $\psltild$, are positive rotations in the
  trigonometric direction. It follows that the axes of $\pm A_1$
  and $\pm B_1$ are as in Figure \ref{FigureComm}, and that
  the axis of $\pm A_1$ turns negatively around the fixed point
  of $\pm [A_1,B_1]$ (in other words, when we travel along the axis of
  $\pm A_1$ towards its attractive point, we have the fixed point of
  $\pm [A_1,B_1]$ at our right hand). The same is true for $A_2$, and remains true if we replace $A_1$ by $B_1^NA_1$ and
  $A_2$ by $B_2^NA_2$, with $N$ large. This replacement amounts to
  precomposing $\rho$ with Dehn twists, hence it does not change the Euler
  class of $\rho$ and guarantees that both displacements $\lambda(A_i)$'s
  are as large as we may want.

  Let $C=[A_1,B_1]$. The formula obtained in Remark \ref{RemarqueComm}
  implies that
  the axes of $\pm A_1$ and $\pm A_2$ are at distance close to
  $\operatorname{arcsinh}(\operatorname{cotan}\theta)$ of the fixed point of
  $C$, $\theta$ being its rotation angle. Now, up to conjugating $A_2$ by a suitable power
  of $C$ (ie, up to Dehn twists), the situation is as
  in the left part of Figure \ref{Figurea1a2},
  where the axes of $\pm A_1$ and $\pm A_2$ are approximately symmetric
  around the fixed point of $C$. Here, the product
  $\pm A_1A_2$ can be constructed as a product of two reflections, as suggested
  in the left part of Figure \ref{Figurea1a2}.
  The displacements of $\pm A_1$ and $\pm A_2$ being
  large, and the distance between their axes being bounded from zero, the picture
  is indeed as in this figure, and $\pm A_1A_2$ is hyperbolic.

  Conjugating $A_2$ by suitable powers of $C$ results (at least)
  in the freedom
  of choosing the parameter $\ell$, in the right part of Figure \ref{Figurea1a2}, in a dense
  subset of the segment $[0,d]$, where $d$ is the distance between
  the axes of $A_1$ and $A_2$.
  \begin{figure}[htbp]
  \begin{asy}
    import math;
    import hyperbolic_geometry;
    real taille=140;

    size(taille,taille);
    hyperbolic_point fixC=hyperbolic_point(0.1,85);
    hyperbolic_point pa1=hyperbolic_point(0.6,180);
    hyperbolic_point pb1=hyperbolic_point(0.9,0);
    hyperbolic_line cent=hyperbolic_line(pa1,pb1);
    hyperbolic_line axA=hyperbolic_normal(cent,pa1);
    hyperbolic_line axB=hyperbolic_normal(cent,pb1);
    hyperbolic_line constr=hyperbolic_line(hyperbolic_point(1.5,100),hyperbolic_point(1.5,80));
    hyperbolic_line rA=common_perpendicular(axA,constr);
    hyperbolic_line rB=common_perpendicular(axB,constr);

    draw(unitcircle);
    draw(cent,blue);
    dot(fixC,red);
    label("{\small ${\rm Fix}_C$}",fixC.get_euclidean(),N);
    draw(reverse(axA.to_path()),Arrow(Relative(0.2)));
    draw(reverse(axB.to_path()),Arrow(Relative(0.8)));
    label("\small $A_1$ \normalsize",relpoint(axA.to_path(),0.8),dir(-55));
    label("\small $A_2$ \normalsize",relpoint(axB.to_path(),0.2),dir(235));
    label("\small $s_2$ \normalsize",relpoint(cent.to_path(),0.1),dir(35),blue);
    label("\small $s_3$ \normalsize",relpoint(rA.to_path(),0.2),dir(35),blue);
    label("\small $s_1$ \normalsize",relpoint(rB.to_path(),0.2),dir(55),blue);
    draw(rA,blue);
    draw(rB,blue);
  \end{asy}
  \quad
  \begin{asy}
    import math;
    import hyperbolic_geometry;
    real taille=120;
    size(taille,taille);
    hyperbolic_point fixC=hyperbolic_point(0.1,85);
    hyperbolic_point pa1=hyperbolic_point(0.6,180);
    hyperbolic_point pb1=hyperbolic_point(0.9,0);
    hyperbolic_line cent=hyperbolic_line(pa1,pb1);
    hyperbolic_line axA=hyperbolic_normal(cent,pa1);
    hyperbolic_line axB=hyperbolic_normal(cent,pb1);
    hyperbolic_line constr=hyperbolic_line(hyperbolic_point(1.5,100),hyperbolic_point(1.5,80));
    //hyperbolic_point pa2=intersection(axA,constr);
    //hyperbolic_point pb2=intersection(axB,constr);
    //hyperbolic_line rA=hyperbolic_normal(axA,pa2);
    //hyperbolic_line rB=hyperbolic_normal(axB,pb2);
    hyperbolic_line rA=common_perpendicular(axA,constr);
    hyperbolic_line rB=common_perpendicular(axB,constr);
    hyperbolic_point pa2=intersection(axA,rA);
    hyperbolic_point pb2=intersection(axB,rB);
    hyperbolic_point ma=midpoint(pa1,pa2);
    hyperbolic_point mb=midpoint(pb1,pb2);

    draw(unitcircle);
    draw(cent,blue);
    draw(axA);
    draw(axB);
    //dot(fixC,red);
    //label("{\small ${\rm Fix}_C$}",fixC.get_euclidean(),N);
    //draw(reverse(axA.to_path()),Arrow(Relative(0.2)));
    //draw(reverse(axB.to_path()),Arrow(Relative(0.8)));
    //label("\small $A$ \normalsize",relpoint(axA.to_path(),0.8),dir(-55));
    //label("\small $B$ \normalsize",relpoint(axB.to_path(),0.2),dir(235));
    //// draw(common_perpendicular(rA,rB));
    //draw(cent,blue);
    draw(rA,blue);
    draw(rB,blue);
    draw(hyperbolic_segment(pa1,pb1),blue+1.4pt);
    draw(hyperbolic_segment(pa1,pa2),black+1.2pt);
    draw(hyperbolic_segment(pb1,pb2),black+1.2pt);
    label("{\small $\ell$}",(0,0),S,blue);
    label("{\small $\frac{\lambda(A_1)}{2}$}",ma.get_euclidean(),W);
    label("{\small $\frac{\lambda(A_2)}{2}$}",mb.get_euclidean(),E);
  \end{asy}
  \caption{}
  \label{Figurea1a2}
  \end{figure}
  For a suitable $\ell$, the axes of $s_1$ and $s_3$ can be made to
  intersect each other. The element $\rho(a_1 c^{N} a_2 c^{-N})$,
  where $N$ is chosen properly, is then elliptic, and it is
  the image of a simple non-separating closed curve.
\end{proof}
We are left with the case $g\geq 3$. In order to adapt the proof
of Lemma \ref{SepNonSepLem1} to this case, the technical statement
to prove is the following.
\begin{lemma}\label{SepNonSepLem2}
  Let $\Sigma$ be a surface of genus $g\geq 2$ with one boundary
  component, freely homotopic to $\gamma\in\pi_1(\Sigma)$.
  Let $\rho\colon\pi_1(\Sigma)\rightarrow\psl$ such that:
  \begin{itemize}
  \item[-] It sends no simple closed curve to an elliptic
  element of finite order, or to the identity, or to a parabolic
  element. 
  \item[-] It sends $\gamma$ to an elliptic element (of infinite order).
  \item[-] It sends every non-separating simple
  closed loop to a hyperbolic element. Let $x_0\in\HH$ be the
  fixed point of $\rho(\gamma)$.
  \end{itemize}
  Then there exists a positive real number $D>0$ and a sequence of closed
  non-separating simple loops $\gamma_n$ such that the displacement
  $\lambda(\gamma_n)$ is equal to $D$ for all $n$, and such that
  the distance between $x_0$ and the axis of $\gamma_n$ tends
  to $+\infty$.
\end{lemma}

Now let us use this lemma to prove the following one, which concludes
the proof of Proposition \ref{EDense}. We will then prove
Lemma \ref{SepNonSepLem2}.

\begin{lemma}
  Let $\Sigma$ be a closed surface of genus $g\geq 3$. Let
  $\rho:\Gamma_g\rightarrow\psl$
  be a representation such that $[\rho]\in\mathcal{N}^k$,
  with $|k|\leq 2g-3$. Suppose moreover that $\rho$ sends some
  separating simple closed loop $\gamma$ to an elliptic element
  (thus, of infinite order).

  Then $\rho$ also sends some non-separating simple closed loop
  to an elliptic element.
\end{lemma}

\begin{proof}
  Base $\Gamma_g$ on a simple closed separating curve
  $\gamma$ and suppose that $\rho(\gamma)$ is elliptic
  (of infinite order).

  Let us say that three positive real numbers $\ell_1$, $d$, $\ell_2$
  satisfy the condition $H(\ell_1,d,\ell_2)$ if there exists a hyperbolic
  right-angled hexagon with three consecutive lengths equal to
  $\ell_1$, $d$ and $\ell_2$. It is easy to restate condition
  $H(\ell_1,d,\ell_2)$ in terms of hyperbolic trigonometric functions,
  but all we will need is the remark that
  $H(x_n,y_n,z_n)$ holds for all $n$ large enough,
  whenever $x_n$, $y_n$ and $z_n$ are sequences such that none of them
  accumulates to $0$, and $y_n$ goes to infinity or $x_n$ and $z_n$
  both go to infinity. Observe, also, that if $H(x,y,z)$ holds and if
  $x'\geq x$, $y'\geq y$ and $z'\geq z$, then $H(x',y',z')$ holds.
  Also, the condition $H(x,y,z)$ is open in $(x,y,z)$.

  Let $\Sigma_1$ and $\Sigma_2$ be the two pieces of
  $\Sigma\smallsetminus\gamma$.
  Lemma \ref{SepNonSepLem2} either finds some simple closed loop mapped
  by $\rho$ to an elliptic element, or guarantees that there exist a simple,
  non-separating loop $a\in\pi_1(\Sigma_1)$ and a simple, non-separating
  loop $b\in\pi_1(\Sigma_2)$ such that the condition
  $H(\lambda(\rho(a)),|D_a-D_b|,\lambda(\rho(b)))$ holds,
  where $D_a$ (resp. $D_b$) is the distance between the axis of
  $\rho(a)$ (resp. $\rho(b)$) to the fixed point $x_0$ of $\rho(\gamma)$.

  Up to replacing $b$ by $b^{-1}$, we may suppose that $a\cdot b$ is a
  non-separating simple closed curve.

  Denote $A=\rho(a)$, $B=\rho(b)$, and $C=\rho(\gamma)$.
  Up to conjugating $b$ by an adequate power of $\gamma$,
  we may suppose that the fixed point of $C$ lies very close to the segment
  perpendicular to the axes of $A$ and $B$.

  Orient the axes of $A$ and $B$ accordingly to the displacements of
  $A$ and $B$ along these axes. Then either these orientations agree,
  or they do not.

  If they do agree, then we end the argument exactly as we did in the
  proof of Lemma \ref{SepNonSepLem1}: since the condition
  $H(\lambda(\rho(a)),D_a+D_b,\lambda(\rho(b)))$ holds, the picture is
  indeed as in Figure \ref{Figurea1a2}, upon replacing $A_1$, $A_2$
  by $A$, $B$. Then we proceed the same argument and find a non-separating 
  simple closed loop sent to an elliptic element of $\psl$.

  If they do not agree, then up to conjugating $B$ by a suitable power
  of $\rho(\gamma)$, we are this time in the situation of Figure \ref{FigureAB}.
  \begin{figure}[htbp]
  \begin{asy}
    import math;
    import hyperbolic_geometry;
    real taille=120;

    size(taille,taille);
    hyperbolic_point fixC=hyperbolic_point(0.9,175);
    hyperbolic_point pa1=hyperbolic_point(0.3,180);
    hyperbolic_point pb1=hyperbolic_point(1.3,0);
    hyperbolic_line cent=hyperbolic_line(pa1,pb1);
    hyperbolic_line axA=hyperbolic_normal(cent,pa1);
    hyperbolic_line axB=hyperbolic_normal(cent,pb1);
    hyperbolic_line constr=hyperbolic_line(hyperbolic_point(2.1,75),hyperbolic_point(2.1,65));
    //hyperbolic_point pa2=intersection(axA,constr);
    //hyperbolic_point pb2=intersection(axB,constr);
    //hyperbolic_line rA=hyperbolic_normal(axA,pa2);
    //hyperbolic_line rB=hyperbolic_normal(axB,pb2);
    hyperbolic_line rA=common_perpendicular(axA,constr);
    hyperbolic_line rB=common_perpendicular(axB,constr);
    hyperbolic_point pa2=intersection(axA,rA);
    hyperbolic_point pb2=intersection(axB,rB);
    hyperbolic_point ma=midpoint(pa1,pa2);
    hyperbolic_point mb=midpoint(pb1,pb2);

    draw(unitcircle);
    draw(cent,blue);
    //draw(axA);
    //draw(axB);
    dot(fixC,red);
    label("{\small ${\rm Fix}_C$}",fixC.get_euclidean(),N);
    draw(reverse(axA.to_path()),Arrow(Relative(0.2)));
    draw(reverse(axB.to_path()),Arrow(Relative(0.8)));
    label("\small $A$ \normalsize",relpoint(axA.to_path(),0.8),dir(-55));
    label("\small $B$ \normalsize",relpoint(axB.to_path(),0.2),dir(235));
    //// draw(common_perpendicular(rA,rB));
    //draw(cent,blue);
    draw(rA,blue);
    draw(rB,blue);
    draw(hyperbolic_segment(pa1,pb1),blue+1.4pt);
    //draw(hyperbolic_segment(pa1,pa2),black+1.2pt);
    //draw(hyperbolic_segment(pb1,pb2),black+1.2pt);
    label("{\scriptsize $\simeq\! |\!D_a\!-\!D_b|$}",midpoint(pa1,pb1).get_euclidean()+(-0.03,0),S,blue);
    //label("{\small $\frac{d_A}{2}$}",ma.get_euclidean(),W);
    //label("{\small $\frac{d_B}{2}$}",mb.get_euclidean(),E);
  \end{asy}
  \caption{}
  \label{FigureAB}
  \end{figure}
  The picture is indeed as in this figure, because the condition
  $H(\lambda(\rho(a)),|D_a-D_b|,\lambda(\rho(b)))$ holds.
  Now we conclude the proof in exactly the same way as in the preceding case.
\end{proof}

We conclude with the proof of Lemma \ref{SepNonSepLem2}.

\begin{proof}[Proof of Lemma \ref{SepNonSepLem2}]
  Consider the curves $\delta$, $a$ and $b$ as in
  Figure \ref{FigureFinDeChasse}.
  \begin{figure}[htbp]
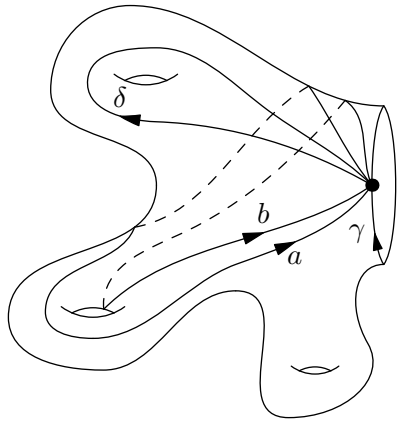

  \begin{asy}
    import geometry;
    unitsize(2pt);

    // Le Bord
    path bord1 = ((3,-15) .. controls (0,-11) and (0,11) .. (3,15));
    path bord2 = rotate(180,(3,0))*bord1;

    draw(bord1,Arrow(Relative(0.2))); draw(bord2);
    label("{\small $\gamma$}",relpoint(bord1,0.2),W);
    point Orig = relpoint(bord1,0.5);

    //
    path cont1 = ((3,15){left}..(-45,34){left}..(-60,18){down}..(-40,0){down});
    draw(cont1);
    path cont2 = ((-40,0){down}..(-68,-25){down}..(-50,-35){right}..(-25,-20){right});
    draw(cont2);
    path cont3 = ((-25,-20){right}..(-15,-44){right}..(0,-30){dir(120)}..(3,-15){right});
    draw(cont3);

    //
    //
    picture trou;
    path chtrb=((-6,1)..(0,-1)..(6,1));
    path chtrh=((-4,-0.2)..(0,0.8)..(4,-0.2));
    draw (trou,(-6,1)..(0,-1)..(6,1));
    draw (trou,(-4,-0.2)..(0,0.8)..(4,-0.2));
    //
    picture trouhaut=shift(-42,20)*trou;  //rien \`a voir avec Barbara ;-)
    add(trouhaut);
    //
    picture troumilieu=shift(-52,-24)*trou;
    add(troumilieu);
    //
    picture troubas = shift(-10,-35)*scale(0.75)*trou;
    add(troubas);

    // Maintenant delta
    path delta=(Orig{dir(150)}..(-40,12)..(-53,20)..(-40,26){right}..(-20,15)..Orig{dir(-45)});
    draw (delta,Arrow(Relative(0.4)));
    label("{\small $\delta$}",relpoint(delta,0.4),N);

    // Puis a
    path a1 = (Orig{dir(130)}..relpoint(cont1,0.14){dir(125)});
    path a2 = (relpoint(cont1,0.14){dir(210)}..relpoint(cont2,0.1){dir(190)});
    path a3 = (relpoint(cont2,0.1){dir(-110)}..(-61,-24){down}..(-52,-29){right});
    path a4 = ((-52,-29){right}..(-25,-15){dir(20)}..Orig{dir(55)});
    draw(a1); draw(a2,dashed); draw(a3); draw(a4,Arrow(Relative(0.7)));
    label("{\small $a$}",relpoint(a4,0.7),S);

    // Puis b
    path chemtroumil=shift(-52,-24)*chtrh;
    point Pb = relpoint(chemtroumil,0.75);
    path b1 = (Orig{dir(115)}..relpoint(cont1,0.07){dir(145)});
    path b2 = (relpoint(cont1,0.07){dir(230)}..(Pb+(14,14))..(Pb+(0,2))..Pb{dir(-80)});
    path b3 = (Pb{dir(45)}..Orig{dir(35)});
    draw(b1); draw(b2,dashed); draw(b3,Arrow(Relative(0.6)));
    label("{\small $b$}",relpoint(b3,0.6),N);

    // Et puis le point base !
    dot(Orig,black+5pt);
  \end{asy}
  \caption{One-holed surface of genus at least $2$}
  \label{FigureFinDeChasse}
  \end{figure}
  If $\delta$, $a$ or $b$ is sent by $\rho$ to a non-hyperbolic element,
  then we are done.

  If $\delta$ and $a$ are sent to hyperbolic elements with different
  axes, then we may replace $\delta$ by its image under a big power
  of a Dehn twist along a curve homotopic to $a$, close to $a$ in
  Figure \ref{FigureFinDeChasse}. 
  In other words, we may replace
  $\delta$ by $a^N\delta a^{-N}$. Letting $N$ go to $+\infty$
  or to $-\infty$ we obtain elements $\rho(a^N\delta a^{-N})$ which
  satisfy the conclusion of Lemma \ref{SepNonSepLem2}.

  We will proceed the same way if $\delta$ and $b$ are sent to
  hyperbolic elements with different axes. We are left with the case when $\delta$, $a$ and $b$ are all sent
  to commuting hyperbolic elements. In that case, since $i(a,b)=1$,
  up to replacing $b$ by its inverse, $[a,b]=aba^{-1}b^{-1}$
  is a simple closed loop in $\Sigma$, sent by $\rho$ to the identity;
  this contradicts the assumption that $\rho\in\mathcal{N}$.
\end{proof}